\documentclass[a4paper,smallextended,numbook,runningheads]{svjour3}

\usepackage[utf8]{inputenc}

\usepackage{latexsym}
\usepackage{amsmath}
\usepackage{amssymb}

\usepackage[ruled,vlined,linesnumbered]{algorithm2e}

\usepackage{graphicx}
\usepackage{caption,subcaption}
\usepackage{geometry}
\usepackage{float}

\usepackage{color}
\newcommand{\fm}[1]{\textcolor{black}{#1}}
\newcommand{\we}[1]{\textcolor{black}{#1}}
\newcommand{\wee}[1]{\textcolor{black}{#1}}

\newcommand{\Int}{\scalebox{1.2}{\ensuremath{{\cal I}}}}

\usepackage{hyperref}

\title{On Kosloff Tal-Ezer Least-Squares Quadrature Formulas }
\author{G. Cappellazzo  \and
        W. Erb     \and
        F. Marchetti \and \\ 
        D. Poggiali
}

\institute{Giacomo Cappellazzo \at
              Dipartimento di Matematica \lq\lq Tullio Levi-Civita\rq\rq, Universit\`a di Padova, Italy \\
              \email{giacomo.cappellazzo@studenti.unipd.it}    
            \and
            Wolfgang Erb \at
            Dipartimento di Matematica \lq\lq Tullio Levi-Civita\rq\rq, Universit\`a di Padova, Italy \\
              \email{wolfgang.erb@unipd.it}
             \and
             Francesco Marchetti \at
              Dipartimento di Matematica \lq\lq Tullio Levi-Civita\rq\rq, Universit\`a di Padova, Italy \\
              \email{francesco.marchetti@unipd.it} 
            \and 
           Davide Poggiali  \at
          PNC - Padova Neuroscience Center, Universit\`a di Padova, Italy \\
           \email{poggiali.davide@gmail.com}   
}

\date{Revision of the manuscript: 2. November 2022}

\journalname{BIT}

\begin{document}

\maketitle

\begin{abstract}

In this work, we study a global quadrature scheme for analytic functions on compact intervals based on function values on \we{quasi-uniform} grids of quadrature nodes. In practice it is not always possible to sample functions at optimal nodes \wee{that give well-conditioned and quickly converging interpolatory quadrature rules at the same time}. Therefore, we go beyond classical interpolatory quadrature by lowering the degree of the polynomial approximant and by applying auxiliary mapping functions that map the original quadrature nodes to more suitable fake nodes. More precisely, we investigate the combination of the Kosloff Tal-Ezer map and least-squares approximation (KTL) for numerical quadrature: a careful selection of the mapping parameter ensures \wee{stability of the scheme,} a high accuracy of the approximation and, at the same time, an asymptotically optimal ratio between the degree of the polynomial and the spacing of the grid. We will investigate the properties of this KTL quadrature and focus on the symmetry of the quadrature weights, \we{the limit relations for the mapping parameter}, as well as the computation of the quadrature weights in the standard monomial and in the Chebyshev bases with help of a cosine transform. Numerical tests on equispaced nodes show that a static choice of the map's parameter improve the results of the composite trapezoidal rule, while a dynamic 
approach achieves larger stability and faster convergence, even when the sampling nodes are perturbed. From a computational point of view the proposed method is practical and can be implemented in a simple and efficient way.

\end{abstract}

\section{Introduction}

A classical problem in numerical analysis is the numerical approximation of the integral 
\begin{equation*} 
\Int(f,\Omega) := \int_{\Omega} f(x) {\rm d} x, \quad \Omega=[a,b]\subset \mathbb{R},
\end{equation*} 
from discrete function samples $\boldsymbol{f}=(f_0,\dots,f_m)^{\top},\; f_i=f(x_i),$ on a set of distinct quadrature nodes $\mathcal{X}=\{x_i, \; i=0,\ldots,m\}\subset \Omega$. If the nodes $\mathcal{X}$ can be chosen freely in the interval $\Omega$, the interpolatory Gauss quadrature formulas or Clenshaw-Curtis quadrature formulas are the standard choices to approximate the integral $\Int(f,\Omega)$ if the function $f$ is smooth. In spectral methods, the classical choices for quadrature and interpolation nodes in the reference interval $\Omega = I = [-1,1]$ are the Chebyshev and Chebyshev-Lobatto nodes given as
\begin{equation*}
    \mathcal{C}_{m+1} := \left\{ -\cos \left( \frac{(2i+1) \pi}{2m+2}\right) , \; i=0,\ldots,m \right\} \quad \text{and} \quad     \mathcal{U}_{m+1} := \left\{ -\cos \left( \frac{i \pi}{m}\right) , \; i=0,\ldots,m \right\}.
\end{equation*}
\wee{These nodes display exceptionally suitable properties for the conditioning and the convergence of quadrature rules. For instance, for the Clenshaw-Curtis rule based on the Chebyshev-Lobatto nodes $\mathcal{U}_{m+1}$, the quadrature weights are all positive \cite{imhof63} and the respective quadrature rule well-conditioned. Furthermore, as soon as the underlying function is analytic on $\Omega$ and analytically continuable on an open Bernstein ellipse around $\Omega$, this quadrature formula will convergence geometrically towards the integral $\Int(f,\Omega)$ \cite[Chapter 19]{trefappr}.}

However, in practice it is not always possible to sample the function $f$ on arbitrary positions of the interval $\Omega$ or the knowledge of the function values might be restricted to some a priori fixed node sets $\mathcal{X}$. In these cases, the efficient quadrature schemes based on, for instance, the Chebyshev nodes $\mathcal{C}_{m+1}$, $\mathcal{U}_{m+1}$ or the Gauss-Legendre nodes are not directly accessible.
Moreover, if only \we{quasi-equispaced} nodes are available, then a high-order interpolatory quadrature scheme gets highly ill-conditioned. 
\wee{This ill-conditioning is linked to a highly oscillatory behavior of the quadrature weights and can be described formally by the rapidly increasing sums of the moduli of the quadrature weights. In particular, it can be shown that for interpolatory quadrature formulas on equidistant nodes (the Newton-Cotes formulas) this sum grows exponentially with order at least $\mathcal{O}(\frac{2^{m+1}}{(m+1)^3})$ \cite[Theorem 5.2.1]{brasspetras2011}. For the respective polynomial interpolation problem on equidistant nodes a similar exponential ill-conditioning is known, and usually referred to as Runge phemonon \cite{Runge,Turetskii}.}  

If resampling of the function $f$ is not practicable, stable  \we{and well-conditioned} alternatives to an interpolatory formula have to be found. A classical solution in this regard is the usage of composite quadrature schemes in which the node set $\mathcal{X}$ is subdivided into smaller portions endowed with a local interpolatory formula. If the set $\mathcal{X}$ is equidistant, this approach leads to the well-known composite Newton-Cotes formulas. For smooth functions $f$, a disadvantage of such a composite scheme \we{is the limitation in the achieved} convergence rates compared to geometric convergence rates that are possible for interpolatory schemes (if the node set $\mathcal{X}$ is the right one). 

\wee{A further standard approach to stabilize a quadrature rule on a fixed node set $\mathcal{X}$ is given by least-squares quadrature formulas \cite{Glaubitz20,Huybrechs09,Migliorati22,Wilson70}. In this case, the integral of a polynomial least-squares solution is used to define the quadrature rule. The reduced degree $n < m$ of the polynomial space in the least-squares approach leads to a better conditioning of the respective formulas. For equidistant nodes it is shown in \cite{Wilson70b} that the grid size $m$ has to scale as $n^2$ in terms of the polynomial degree $n$ in order to obtain positive least-squares quadrature weights. Thus, although the least-squares approach leads to stable quadrature rules, the required quadratic scaling between $m$ and $n$ is a considerable drawback of polynomial least-squares formulas if simultaneously a high convergence rate for smooth functions should be achieved. }
\we{When leaving the polynomial setting, other well-established solutions exist in the literature. One prominent approach which performs very well especially at equispaced nodes is given by quadrature formulas based on rational interpolation,  \cite{Bos12,Floater07,Hormann08,Hormann16}.}

Here, we will focus on an alternative idea based on a mapping function $S$ that maps a set of quasi-uniform \wee{quadrature nodes $\mathcal{X}$ to a node set $S(\mathcal{X})$ in which the nodes are shifted towards the boundary of $\Omega$. These shifted nodes display an improved behavior in terms of conditioning and convergence of the respective quadrature formulas. When combining the map $S$ with classical interpolatory or least-squares quadrature formulas on the mapped nodes $S(\mathcal{X})$, the resulting quadrature formulas on the orginal set $\mathcal{X}$ are non-polynomial, can however be interpreted as polynomial quadrature formulas on the mapped nodes $S(\mathcal{X})$}. This idea has been investigated thoroughly for interpolation problems in \cite{platte} in terms of resulting mapped basis functions and in \cite{FakeNodes} in which the nodes $S(\mathcal{X})$ have been referred to as \lq\lq fake nodes". In this global approach, which has been also investigated in the contexts of barycentric rational approximation \cite{BDEM} and multivariate approximation \cite{FakeNodesMulti}, no resampling of the function $f$ is necessary as the given sampling values are directly used on the new nodes $S(\mathcal{X})$. Furthermore, \wee{for the interpolation problem the well-conditioning can be ensured by an inheritance property of the Lebesgue constant for the mapped nodes $S(\mathcal{X})$ (cf. \cite[Proposition 3.4]{FakeNodesMulti}).}  

For the domain $\Omega = I = [-1,1]$, a prominent example of such a mapping function is the \textit{Kosloff Tal-Ezer} (KT) map $M_{\alpha}: I \to I$ given by (cf. \cite{Kosloff})
\begin{equation}\label{eq:kte_map}
	M_{\alpha}(x) := \frac{\sin(\alpha \pi x/2)}{\sin(\alpha \pi/2)}, \qquad x \in I, \quad \text{for} \; 0 < \alpha \leq 1,
\end{equation} 
and $M_{0}(x) := \lim_{\alpha \to 0^{+}} M_{\alpha}(x) = x$ for $\alpha = 0$. While $M_0$ is the identity map on $I$, the KT function $M_{\alpha}$ with $\alpha =1$ maps the open and closed equidistant Newton-Cotes quadrature nodes to the Chebyshev $\mathcal{C}_{m+1}$ and Chebyshev-Lobatto nodes $\mathcal{U}_{m+1}$. \wee{Similarly, if $\mathcal{X}$ is a quasi-uniform grid in the interval $I$, the distribution of the mapped nodes $M_{\alpha}(\mathcal{X})$ with the parameter $\alpha$ close to $1$ clusters towards the ends of the interval $I$ which improves the conditioning of corresponding interpolation and least-squares procedures}. The properties of \wee{the KT map for the conditioning and the accuracy of the weighted} least-squares approximation of a function $f$ have been thoroughly studied in \cite{platte}. In the literature also other maps have been studied. In \cite{tref}, for instance, conformal maps have been applied for the acceleration of Gauss-type quadrature schemes. Approximation methods and numerical quadrature through mapped nodes have also been used extensively in the context of spectral methods for PDEs; see \cite[Chapter 16]{boyd} for a large overview.

The goal of this work is to show that the KT map in combination with the least-squares approximant introduced in \cite{platte} leads to quickly converging quadrature formulas for analytic functions starting from function values on \we{a quasi-uniform} grid $\mathcal{X}$. Our starting point is the general setting for interpolatory quadrature formulas described in \cite{FakeQuadrature}, then we will move on to the following main setting:
\begin{enumerate}
    \item 
    We will give up the interpolatory conditions used in \cite{FakeQuadrature} and consider more general types of quadrature formulas based on least-squares approximation. We will see that this transition leads to a faster convergence of the quadrature formulas.
    \item
    As underlying mapping we will consider the KT map \eqref{eq:kte_map} with a general parameter $\alpha$ in $[0,1]$. We will study the role of the parameter $\alpha$ \we{in the convergence} of the quadrature formulas. In doing so, the theoretical investigation carried out in \cite{platte} plays a central role.
\end{enumerate}

{\bfseries \noindent Main results}

\begin{itemize}
    \item 
    In addition to a result given in \cite{FakeQuadrature} we show how the composite midpoint rule and the composite Cavalieri-Simpson formula are related to a mapped interpolatory quadrature formula.
    \item
    We introduce and analyse the Kosloff Tal-Ezer map as stabilizing component of interpolatory and least-squares quadrature formulas (referred to as KTI and KTL formulas) for \we{quasi-uniform} node sets. A careful selection of the mapping parameter $\alpha$ ensures on one hand \wee{a high} accuracy of the approximation and on the other hand an asymptotically optimal ratio between the degree of the polynomial \wee{approximation} and the spacing of the grid. 
    \item
    We study the symmetry of the KTI and KTL quadrature weights, limit relations for $\alpha$ converging to $0^{+}$ and $1^{-}$, as well as the computation of the quadrature weights in the standard monomial and the Chebyshev basis with help of a cosine transform.
\end{itemize}

\vspace{1mm}


{\bfseries \noindent Organization of this paper}\vspace{5pt}

In {\bfseries Section \ref{sec:sezione_fake_quadrature}}, we review interpolatory quadrature formulas combined with a mapping of the quadrature nodes. We shortly recapitulate how these formulas can be computed and we provide three examples of mapped interpolatory formulas.

In {\bfseries Section \ref{sec:KTLquadrature}} and {\bfseries Section \ref{sec:KTLproperties}}, we introduce and study the KTI and KTL quadrature formulas, and we investigate analytic and numerical properties of the corresponding quadrature schemes.

In the last {\bfseries Section \ref{sec:numericalexperiments}}, we conduct a series of numerical experiments to compare different parameter choices and the efficiency of the method with regard to other quadrature rules. Our experiences with the KTL quadrature scheme show that the formulas are practical, simple and can be implemented efficiently.

\section{Interpolatory quadrature formulas based on mapped nodes and mapped basis functions}\label{sec:sezione_fake_quadrature}

We start this work with some preliminary facts about interpolatory quadrature formulas and the respective adaptions if an additional mapping is involved.  

Let $\mathcal{X} = \{x_0, \ldots, x_m\}$ be an increasingly ordered set of quadrature nodes in the interval $\Omega = [a,b]$ and $f$ a continuous function on $\Omega$. An interpolatory quadrature formula $\Int_{m,\mathcal{X}}(f,\Omega)$ is built upon the unique polynomial $P_{m} (f)$ of degree $m$ interpolating $f$ at the nodes $\mathcal{X}$. This interpolant can be written in terms of the monomial basis $\{ 1,x,...,x^m\}$ as
\begin{equation*}
    P_{m}(f)(x) = \sum_{i=0}^{m} \gamma_{i} x^i, \quad x \in \Omega,
\end{equation*}
where the coefficients $\gamma_{0},\ldots,\gamma_{m}$ are determined by the $m + 1$ interpolatory conditions $P_{m}(f)(x_i) = f_i = f(x_i)$. Alternatively, $P_{m} (f)$ can be expressed in terms of the Lagrange basis $\{ \ell_{0},...,\ell_{m}\}$ as
\begin{equation*}
    P_{m}(f)(x) = \sum_{i=0}^{n} f_{i}\ell_{i}(x), \quad x \in \Omega,
\end{equation*}
where the Lagrange polynomials are given as
\begin{equation}\label{lagrange_pol_def}
    \ell_{i}(x) = \prod_{\substack{0\leqslant j \leqslant m \\ j\neq i}}^{} \frac{x-x_{j}}{x_{i}-x_{j}}.
\end{equation}
With the vector $\boldsymbol{w}=(w_0,\ldots,w_m)^{\top}$ consisting of the interpolatory quadrature weights $w_i=\Int(\ell_i,\Omega),\;i=0,\ldots,m\:$, the interpolatory quadrature formula $\Int_{m,\mathcal{X}}(f,\Omega)$ is given as
\begin{equation}\label{eq:quadrature_formula}
    \Int_{m,\mathcal{X}}(f,\Omega) := \Int(P_{m,\mathcal{X}}(f),\Omega) = \boldsymbol{w}^{\top} \boldsymbol{f} \approx \Int(f,\Omega).
\end{equation}

Going one step further, we consider an additional injective map $S:\Omega\longrightarrow\mathbb{R}$ included in the interpolation process. The idea of the so-called fake nodes approach (FNA) introduced in \cite{FakeNodes} is to obtain an interpolant on the nodes $\mathcal{X}$ by constructing a polynomial interpolant on the fake nodes $S(\mathcal{X})$. More precisely, if $P_{m,S(\mathcal{X})} (g)$ denotes the unique polynomial interpolant of the function $g = f \circ S^{-1}$ on the nodes $S(\mathcal{X})$, the interpolant of $f$ on $\mathcal{X}$ is defined as
\begin{equation*}
R^S_{m,\mathcal{X}}(f)(x) := P_{m,S(\mathcal{X})}(g)(S(x)),
\end{equation*}
The interpolant $R^S_{m,\mathcal{X}}(f)$ can be expressed in terms of the mapped Lagrange basis $\{ \lambda^S_{0},...,\lambda^S_{m}\}$, i.e.,
\begin{equation*}
    R^S_{m,\mathcal{X}}(f)(x) =\sum_{i=0}^{m} f_i\lambda_{i}^{S}(x),
\end{equation*}
where $\lambda^S_i:= \ell^{S}_i\circ S$ and $\ell^{S}_i$ is the $i$-th Lagrange polynomial on the node set $S(\mathcal{X})$. Then, similarly to \eqref{eq:quadrature_formula}, we obtain the interpolatory quadrature formula
\begin{equation}\label{eq:quad_fake}
    \Int^S_{m,\mathcal{X}}(f,\Omega):= \Int(R^S_{m,\mathcal{X}}(f),\Omega) = (\boldsymbol{w}^S)^{\top} \boldsymbol{f} \approx \Int(f,\Omega),
\end{equation}
where $\boldsymbol{w}^S=(w_0^S,\ldots,w_m^S)^{\top}$ are the quadrature weights determined by the conditions
\begin{equation*}
    w_i^S=\Int(\lambda_i^S,\Omega),\quad i=0,\dots,m,
\end{equation*}
that hold true for every interpolatory quadrature formula. We point out that the vector $\boldsymbol{w}^S$ can be computed by solving the linear system
\begin{equation}\label{eq:sistema_mono}
    (\mathbf{A}^{S})^{\top}\boldsymbol{w}^S=\boldsymbol{\tau}^S,
\end{equation}
where $\mathbf{A}^{S}_{ij} := S(x_i)^j$,  $i,j=0,\ldots,m$, is the well-known Vandermonde matrix for the set $S(\mathcal{X})$, and $\boldsymbol{\tau}^S=(\tau_0^S,\ldots,\tau_m^S)^{\top}$ is the vector of moments related to the basis $\{1, S(x), \ldots, S(x)^m\}$, i.e., 
\begin{equation*}
    \tau_i^S=\Int(S^i,\Omega),\quad i=0,\dots,m\:.
\end{equation*}
If the map $S$ is at least $C^1(\Omega)$, we define for $y=S(x)$ the function 
\begin{equation} \label{eq:eqcb}
\tilde{S}(y) :=  \dfrac{{\rm d} S^{-1}(y)}{{\rm d} y} = \dfrac{1}{S'(S^{-1}(y))}\:.
\end{equation}
\we{Assuming $S$ is injective}, the inverse $S^{-1}$ is well-defined on $S(\Omega)$, and we obtain
\begin{equation} \label{eq:eq_q}
{\cal I }(R^S_{m,\mathcal{X}}(f),\Omega) = \Int\left( \sum_{i=0}^m f_i \,(\ell^{S}_i\circ S),\Omega \right) =  \Int\left( \sum_{i=0}^m f_i \ell^{S}_i\cdot\tilde{S},S(\Omega) \right) =  {\cal I }(P_{m,S(\mathcal{X})}(g)\cdot\tilde{S},S(\Omega)),
\end{equation}
which leads to the formula
\begin{equation}\label{eq:lagrange}
    w^S_i = {\cal I }(\lambda_i^S,\Omega)= {\cal I }(\ell^{S}_i\cdot\tilde{S},S(\Omega)).
\end{equation}
\fm{We point out that the smoothness of the map $S$ (and $S^{-1}$) is relevant for our objective, since otherwise the regularity of the mapped function $g$ gets affected compared to the original underlying function $f$. Furthermore, we note that we use the expression \eqref{eq:sistema_mono} for theoretical purposes only, since computing the quadrature weights by means of such a formula is unstable due to the usage of the monomial basis. For this, we will later on consider a more stable basis defined upon the Chebyshev polynomials of the first kind. }

\begin{example}[Composite midpoint rule]
\label{form_mid_teo}
Let $\mathcal{X} = \{ x_{i} = a+(i+\frac{1}{2})\frac{b-a}{m+1}, \; i=0,\ldots,m\}$ and consider the bijective map
\begin{equation}\label{eq:form_trap_teo}
    S : \Omega \longrightarrow I = [-1,1],\; x \longmapsto -\cos \left( \frac{x-a}{b-a} \pi \right).
\end{equation}
Then, $S(\mathcal{X})$ corresponds to the Chebyshev nodes $C_{m+1}$ on $[-1,1]$ and the quadrature weights related to the interpolatory quadrature formula \eqref{eq:quad_fake} are given by
$$
w_{i}^{S}= \frac{b-a}{m+1},\quad i=0,\ldots,m.
$$
\end{example}
\begin{proof}
A straightforward check shows that $S(\mathcal{X})=\mathcal{C}_{m+1}$ are the Chebyshev nodes on $[-1,1]$. Then, if $\nu: I \longrightarrow \mathbb{R}$ is the Chebyshev weight function, i.e.,
$$
\nu(x)=\frac{1}{\sqrt{1-x^{2}}},
$$
and
$$
S^{-1}(y)=\arccos(-y)\frac{b-a}{\pi}+a,
$$
we get
$$
\tilde{S}(y) = \dfrac{1}{S'(S^{-1}(y))} = \frac{1}{\sin\left( \frac{S^{-1}(y)-a}{b-a} \pi \right)\frac{\pi}{b-a}} = \frac{b-a}{\pi} \frac{1}{\sin(\arccos(-y))} = \frac{b-a}{\pi} \nu(y).
$$
From \eqref{eq:lagrange}, we obtain
$$
w_{i}^{S} = \frac{b-a}{\pi} \int_{-1}^{1} \ell^{S}_{i}(y)\nu(y) \mathrm{d}y,
$$
where $\ell^{S}_{i}, i=0,\ldots ,m$, are the Lagrange polynomials computed at the Chebyshev nodes $\mathcal{C}_{m+1}$.\\
Using the classical quadrature formula at the Chebyshev nodes \cite{Chebyshev,Mason02}
$$
\int_{-1}^{1} f(x) \nu (x) \mathrm{d}x \approx \sum_{i=0}^{m} f(S(x_{i}))\frac{\pi}{m+1},
$$
which has degree of accuracy at least $m$, we obtain the weights
$$
w_{i}^{S} = \frac{b-a}{\pi} \int_{-1}^{1} \ell^{S}_{i}(y)\nu(y) \mathrm{d}y = \frac{b-a}{\pi} \sum_{j=0}^{m} \ell^{S}_{i}(S(x_{j}))\frac{\pi}{m+1} = \frac{b-a}{\pi} \sum_{j=0}^{m} \delta_{i,j}\frac{\pi}{m+1} = \frac{b-a}{m+1}.
$$
\wee{We note that, since the degree of the Lagrange polynomial $\ell^{S}_{i}$ is $m$, the second equality holds due to the exactness of an interpolatory quadrature rule. }
\end{proof}

\we{In addition to Example \ref{form_mid_teo}, we add the following two quadrature rules based on the cosine map in \eqref{eq:form_trap_teo}.}

\begin{example}[Composite trapezoidal rule, {\cite[Theorem 3.1]{FakeQuadrature}}] 
\label{form_trap_teo}
Let $\mathcal{X} = \{ x_{i} = a+i \frac{b-a}{m}, \; i=0,\ldots,m\}$ be a set of $m+1$ equidistant nodes in the interval $[a,b]$ containing the two borders $a$ and $b$, and the map $S$ be given as in \eqref{eq:form_trap_teo}.
Then the mapped nodes $S(\mathcal{X})$, correspond to the Chebyshev-Lobatto nodes $\mathcal{U}_{m+1}$ on the interval $[-1,1]$ and the respective quadrature weights become
\begin{equation}\label{eq:fake_trapezi_pesi}
    w_{i}^{S}=
\begin{cases}
	\frac{b-a}{2m}, & \textrm{for } i \in \{0,m\}, \\
	\frac{b-a}{m}, & \textrm{for } i \in \{1,...,m-1\},
\end{cases}
\end{equation}
i.e., we obtain the composite trapezoidal rule based on $m$ subintervals of $[a,b]$. We point out that \eqref{eq:form_trap_teo} corresponds to $(M_1\circ H)$, where $H(x) = 2 \cdot \frac{(x-a)}{(b-a)}-1$ and $M_1$ is the Kosloff Tal-Ezer map introduced in \eqref{eq:kte_map}.
\end{example}

\begin{example}[Composite Cavalieri-Simpson formula]
Let $\mathcal{X} = \{ x_{i} = a+i \frac{b-a}{2m}, \; i=0,...,2m\}$, and $\mathcal{X}= \mathcal{X}^{\mathrm{e}} \cup \mathcal{X}^{\mathrm{o}}$ be the disjoint subdivision of $\mathcal{X}$ into the nodes with even and odd indices. Further let the map $S$ be given as in \eqref{eq:form_trap_teo}.\\
Then, the composite Cavalieri-Simpson formula at $\mathcal{X}$ is a convex combination of the quadrature scheme in Example \ref{form_mid_teo} applied to \we{$\mathcal{X}^{\mathrm{o}}$} and of the quadrature rule in Example \ref{form_trap_teo} applied to \we{$\mathcal{X}^{\mathrm{e}}$}.
\end{example}
\begin{proof}
First, we split ${\cal I }(f,\Omega)$ into the following convex combination
$$
{\cal I }(f,\Omega) = \frac{2}{3}{\cal I }(f,\Omega)+\frac{1}{3}{\cal I }(f,\Omega).
$$
Then, we apply the quadrature rule of \we{Example  \ref{form_mid_teo}} to the first integral (using the odd nodes $\mathcal{X}^{\mathrm{o}}$), and the scheme of \we{Example  \ref{form_trap_teo}} to the second (using the even nodes $\mathcal{X}^{\mathrm{e}}$), thus achieving
\begin{equation*}
	\begin{split}
		\int_{a}^{b} f(x) \mathrm{d}x & \approx \frac{2}{3} \frac{b-a}{m} \left(\sum_{i=1}^{m} f(x_{2i-1})\right) + \frac{1}{3} \frac{b-a}{m} \left( \frac{f(x_{0})}{2} + \sum_{i=1}^{m-1} f(x_{2i}) + \frac{f(x_{2m})}{2} \right) = \\
		& = \frac{b-a}{6m} \left( f(x_{0}) + 2\sum_{i=1}^{m-1} f(x_{2i}) + 4\sum_{i=1}^{m} f(x_{2i-1}) + f(x_{2m}) \right),
	\end{split}
\end{equation*}
which is the composite Cavalieri-Simpson formula on the nodes $\mathcal{X}$.
\end{proof}

\section{Kosloff Tal-Ezer Least-squares (KTL) quadrature} \label{sec:KTLquadrature}

For simplicity, we restrict the integration domain $\Omega$ to the interval $\Omega=I=[-1,1]$. In the previous Section \ref{sec:sezione_fake_quadrature}, we considered several examples of well-known composite Newton-Cotes schemes that can be interpreted as mapped Gauss-Chebyshev type formulas in which equidistant nodes are mapped onto Chebyshev or Chebyshev-Lobatto points. This particular mapping can be considered as a special case of the Kosloff Tal-Ezer (KT) map $M_{\alpha}$ with $\alpha=1$ as introduced in \eqref{eq:kte_map}.\\

In the following, we give a brief overview on KT-generated mapped polynomial methods as developed and studied in \cite{platte}. First, we observe that the KT maps given by $M_{\alpha}(x) = \frac{\sin(\alpha \pi x/2)}{\sin(\alpha \pi/2)}$ for $\alpha \in (0,1]$  and $M_0(x) = x$ for $\alpha = 0$, are strictly increasing functions on $I$. In particular, the maps $M_{\alpha}(x)$ are bijections of $I$ into itself with the inverse mappings
\begin{equation}
    M_{\alpha}^{-1}(y) = \frac{2}{\alpha \pi}\arcsin(\sin(\alpha \pi/2)y), \quad \alpha \in (0,1], \quad \text{and} \quad M_{0}^{-1}(y) = y.
\end{equation}
Further, the derivative of $M_{\alpha}$ is given by
\begin{equation}
    M_{\alpha}'(x)=\frac{\alpha \pi \cos(\alpha \pi x/2)}{2\sin(\alpha \pi/2)} = \frac{\alpha \pi}{2 \sin(\alpha \pi/2)}\sqrt{1- \sin(\alpha \pi/2)^2 M_{\alpha}(x)^{2}}, \quad \text{and} \quad M_0'(x) = 1.
\end{equation}

If $\mathbb{P}_n$ denotes the space of polynomials of degree at most $n$, we can associate to $M_{\alpha}$ the approximation space \wee{of mapped polynomials}
\begin{equation}\label{pnalpha_def}
	\mathbb{P}_{n}^{\alpha} = \{ P \circ M_{\alpha}\: :\:P \in \mathbb{P}_{n}\}.
\end{equation}
Note that $\mathbb{P}_{n}^{0}=\mathbb{P}_{n}$, while $\mathbb{P}_{n}^{1}$ is a space of functions closely related to trigonometric polynomials (cf.  \cite{platte}).

If $\alpha < 1$, it is shown in \cite{platte} that polynomial interpolation on the nodes $M_{\alpha}(\mathcal{X})$ \wee{is still} ill-conditioned if $\mathcal{X}$ is, for instance, a set of equidistant nodes in $I$. In this case, the set $\mathcal{X}$ is not mapped onto the Chebyshev or Chebyshev-Lobatto nodes, and the polynomial interpolants display Runge type artifacts. To overcome this issue, the size of the node set $\mathcal{X}$ in \cite{platte} was \wee{increased compared to the dimension of the polynomial space} such that $m>n$, and the following weighted least-squares approximant to the function $f$ was introduced:
\begin{equation}\label{fnza_def}
	F_{n,\mathcal{X}}^{\alpha}(f) := \min_{P^{\alpha} \in \mathbb{P}_{n}^{\alpha}} \sum_{i=0}^{m} \mu_{i}|f(x_i)-P^{\alpha}(x_i)|^{2},\quad x_i \in \mathcal{X},
\end{equation}
with the weights $\mu_i$ given by
\begin{equation}\label{mu_def_1}
\mu_{i} = \frac{1}{2} \int_{M_{\alpha}(x_{i-1})}^{M_{\alpha}(x_{i+1})} \frac{1}{\sqrt{1-x^2}}\mathrm{d}x = \frac{1}{2} (\arcsin(M_{\alpha}(x_{i+1}))-\arcsin(M_{\alpha}(x_{i-1}))) , \quad i=0, \ldots ,m,
\end{equation}
where $x_{-1} = -1$ and $x_{m+1} = 1$ denote the borders of the interval $I$. \wee{For $\alpha =1$, the weights $\mu_{i}, \; i=0, \ldots ,m$, correspond to the composite trapezoidal quadrature weights with respect to the the node set $\mathcal{X}$. The usage of the special weights $\mu_i$ is motivated by the fact that under some conditions on the involved parameters an upper bound for the conditioning of the least-squares approximation \eqref{fnza_def} can be found.} \wee{Note that \eqref{fnza_def} defines a non-polynomial weighted least-squares approximation for $\alpha \neq 0$ which is built upon a mapped polynomial basis.} Similarly to the interpolatory framework based on mapped basis elements outlined in Section \ref{sec:sezione_fake_quadrature}, we have also \wee{in the least-squares setting} the relation
\begin{equation*}
    F_{n,\mathcal{X}}^{\alpha}(f) = F_{n,M_{\alpha}(\mathcal{X})}^{0}(g) \circ M_{\alpha},
\end{equation*}
where $g$ is uniquely determined by the relation $f = g \circ M_{\alpha}$ \wee{and where $F_{n,M_{\alpha}(\mathcal{X})}^{0}(g)$ represents a polynomial least-squares fit of the function $g$ on the mapped nodes $M_{\alpha}(\mathcal{X})$.}

\begin{definition}[Kosloff Tal-Ezer Least-squares (KTL) quadrature]
Let $F_{n,\mathcal{X}}^{\alpha} (f)$ be the weighted least-squares approximant of a continuous function $f$ as introduced in \eqref{fnza_def}, then we define the \textbf{KTL quadrature formula} $\Int^{\alpha}_{n,\mathcal{X}}(f,I)$ as
\begin{equation} \label{KTLquad_def}
    \Int^{\alpha}_{n,\mathcal{X}}(f,I):=  \Int(F_{n,\mathcal{X}}^{\alpha} (f),I)\approx \Int(f,\Omega),
\end{equation}
In the particular case $\# \mathcal{X} = m +1 = n+1 = \mathrm{dim}(\mathbb{P}_n^{\alpha})$ the formula $\Int^{\alpha}_{n,\mathcal{X}}(f,I)$ is interpolatory and will be referred to as \textbf{KTI quadrature formula}.

\end{definition}

A first fundamental property of the quadrature formula $\Int^{\alpha}_{n,\mathcal{X}}(f,I)$ is its exactness for all mapped polynomials in the space $\mathbb{P}_{n}^{\alpha}$. For the calculation of the quadrature formula, we choose a basis $\Phi^{\alpha}=\{\phi_{i}^{\alpha} : i=0, \ldots, n\}$ for the space $\mathbb{P}_{n}^{\alpha}$. Then, we can write the least-squares approximant as
\begin{equation} \label{def_A_0}
    F_{n,\mathcal{X}}^{\alpha}(f) = \sum_{i=0}^{n} \gamma_{i} \phi_{i}^{\alpha},
\end{equation}
where the coefficient vector $\boldsymbol{\gamma}:=(\gamma_0,\dots,\gamma_n)^{\top}$ is determined by the least-squares solution of the following linear system 
\begin{equation}\label{def_A_1}
	\mathbf{W} \mathbf{A}^{\alpha} \boldsymbol{\gamma} = \mathbf{W} \boldsymbol{f}.
\end{equation}
In this system
\[\mathbf{W} = \mathrm{diag}(\sqrt{\mu_0}, \ldots, \sqrt{\mu_m}) \in \mathbb{R}^{(m+1) \times (m+1)}\]
denotes the matrix with the roots of the least-squares weights \wee{$\mu_i$ given in \eqref{mu_def_1}}, the matrix $\mathbf{A}^{\alpha} \in \mathbb{R}^{(m+1) \times (n+1)}$ is defined by the entries $\mathbf{A}_{ij}^{\alpha} = \phi_{j}^{\alpha}(x_{i})$ and
$\boldsymbol{f} = (f(x_0), \ldots, f(x_m))^\top$ is the vector with all samples of $f$ on $\mathcal{X}$.  

\we{We now have} two possibilities to calculate the KTL quadrature formula:
\begin{enumerate}
    \item Based on the decomposition \eqref{def_A_0} we have the formula
    \[\Int^{\alpha}_{n,\mathcal{X}}(f,I) = \boldsymbol{\gamma}^\top \boldsymbol{\tau}^{\alpha},\]
    where $\boldsymbol{\gamma} = ((\mathbf{A}^{\alpha})^{\top} \mathbf{W}^2 \mathbf{A}^{\alpha})^{-1} (\mathbf{A}^{\alpha})^{\top}\mathbf{W}^2 \boldsymbol{f}$ is the least-squares solution of the weighted system \eqref{def_A_1} and $\boldsymbol{\tau}^{\alpha} \in \mathbb{R}^{n+1}$ is a moment vector with the entries  $\tau_i^{\alpha}=\Int(\phi_i^{\alpha},I),\; i=0,\ldots, n$.
\item The formula above can be rewritten in an alternative form as
 \[\Int^{\alpha}_{n,\mathcal{X}}(f,I) = (\boldsymbol{w}^{\alpha})^\top \boldsymbol{f},\]
 with the quadrature weights $\boldsymbol{w}^{\alpha} \in \mathbb{R}^{m+1}$ given as the weighted least-squares solution $\boldsymbol{w}^{\alpha} = \mathbf{W}^2 \mathbf{A}^{\alpha} ((\mathbf{A}^{\alpha})^{\top} \mathbf{W}^2 \mathbf{A}^{\alpha})^{-1}  \boldsymbol{\tau}^{\alpha}$ of the linear system $$(\mathbf{A}^{\alpha})^{\top} \boldsymbol{w}^{\alpha} = \boldsymbol{\tau}^{\alpha}.$$
\end{enumerate}

Note that in the interpolatory case $n = m$ it is not necessary to construct the weight matrix $\mathbf{W}$. In this case, the matrix $\mathbf{A}^{\alpha}$ is invertible and the linear systems in 1. and 2. can be solved directly. \we{Note also that the formulas in 1. and 2. are analytically the same. From a numerical point of view, small differences can occur in finite precision arithmetic due to a switched order of the operations. The conditioning and the number of computational steps are the same for both formulas. The formula in 1. includes the moments $\boldsymbol{\tau}^{\alpha}$ explicitly in the rule, while 2. corresponds to the classic quadrature rule formulation in terms of function evaluations.}   \\ 

In the next section, we focus on a particular choice of the basis $\Phi^{\alpha}$ for the space $\mathbb{P}_{n}^{\alpha}$, and provide an efficient and stable procedure for the computation of the moment vector $\boldsymbol{\tau}^{\alpha}$ as well as a convergence analysis for the resulting quadrature formula.

\section{Computation and convergence of KTL quadrature} \label{sec:KTLproperties}
\subsection{Computation of KTL weights in the Chebyshev basis}

The usage of the Chebyshev polynomials $ \{T_{i}(x) = \cos(i \arccos (x)), \, i=0,\ldots,n\}$, as a basis for the space of polynomials of degree $n$ leads to the basis $\{ \phi_i^{\alpha}(x) = T_i(M_{\alpha}(x)), \, i=0,\ldots,n\}$, for the \wee{mapped} space $\mathbb{P}_n^{\alpha}$. 
\wee{A simple argument provided in \cite[Lemma 2.1]{platte} shows that the mapped Chebyshev polynomials $\phi_i^{\alpha}$ form an orthonormal basis with respect to a specific weighted inner product. Further, for this mapped basis an upper bound for the conditioning of the least-squares problem \eqref{def_A_1}, and, thus, for the calculation of the quadrature weights is proven in \cite{platte}. For equispaced grids $\mathcal{X}$ this bound essentially depends on the relation between $m$ and $n$ and the mapping parameter $\alpha$. We illustrate this conditioning for some ratios $n/m$ in Fig. \ref{fig:condizionamento}. It is visible that decreasing the ratio $n/m$ and choosing the parameter $\alpha$ close to $1$ has a significant impact on the conditioning of the least-squares problem.} 

\begin{figure}[H]
	\begin{subfigure}{0.34\textwidth}
		\includegraphics[width=\linewidth]{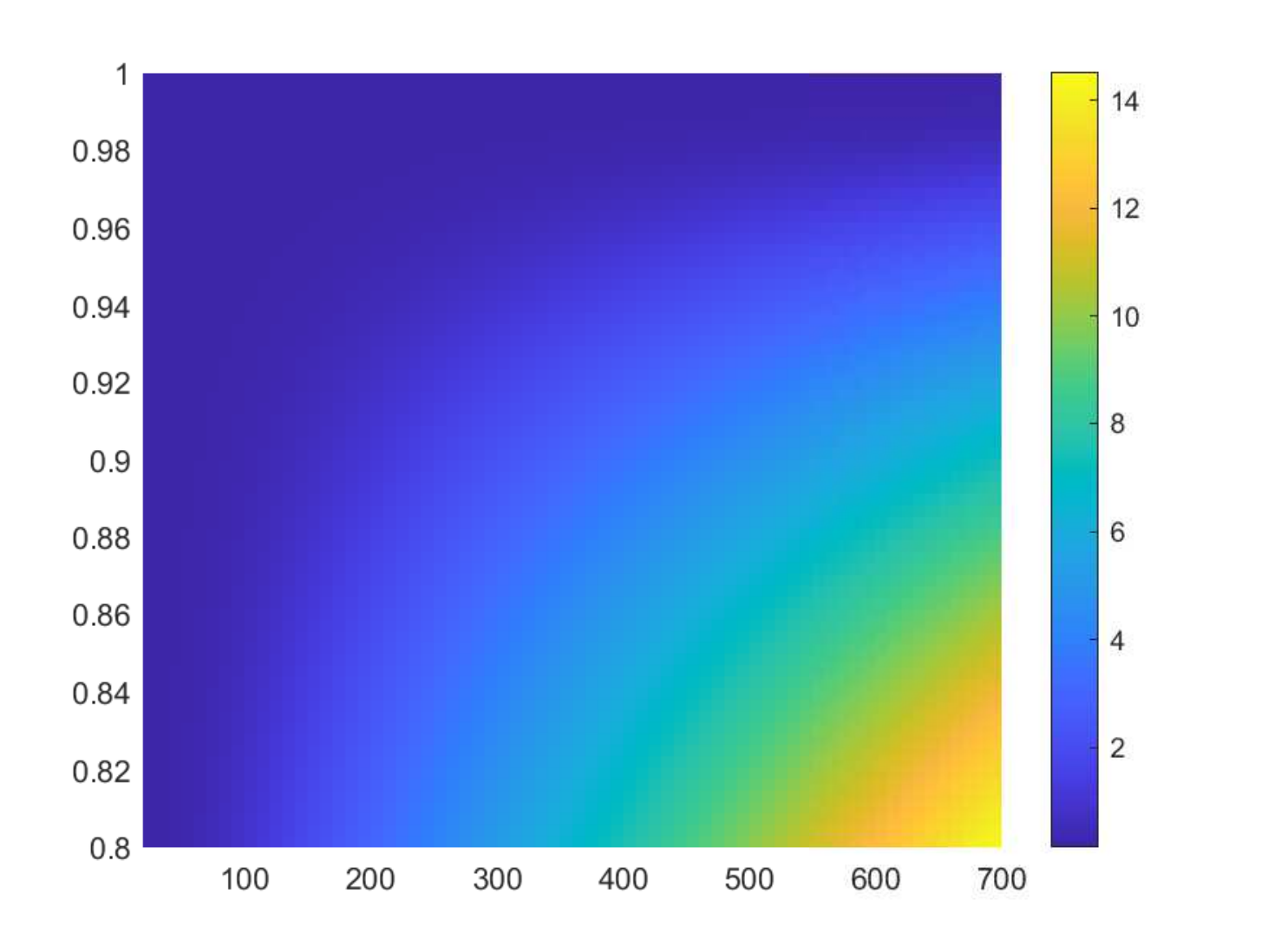}
		\caption*{$\frac{n}{m} = 0.5$}
	\end{subfigure}\hspace{-0.2cm} 
	\begin{subfigure}{0.34\textwidth}
		\includegraphics[width=\linewidth]{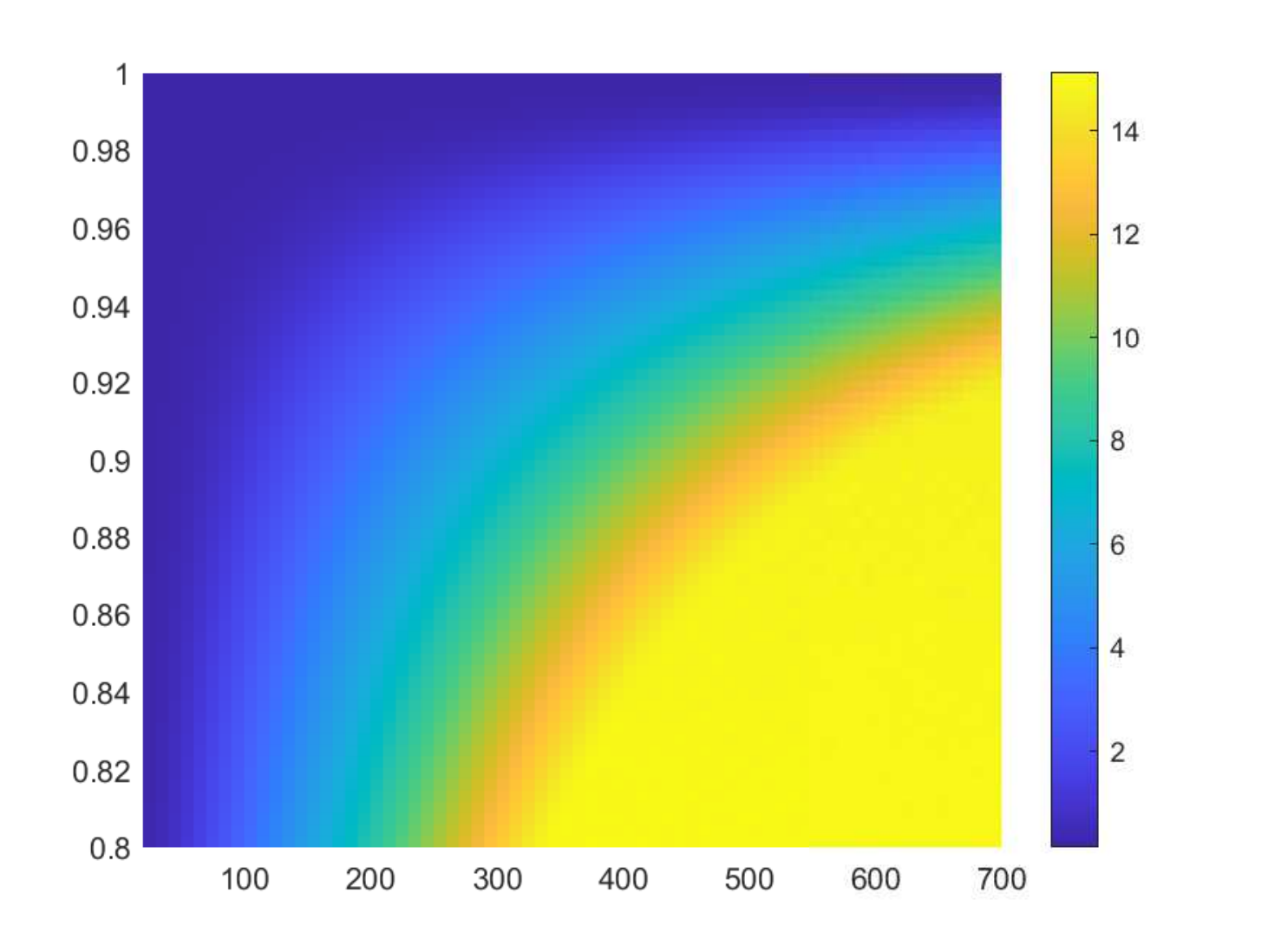}
		\caption*{$\frac{n}{m} = 0.7$}
	\end{subfigure}\hspace{-0.2cm} 
	\begin{subfigure}{0.34\textwidth}
		\includegraphics[width=\linewidth]{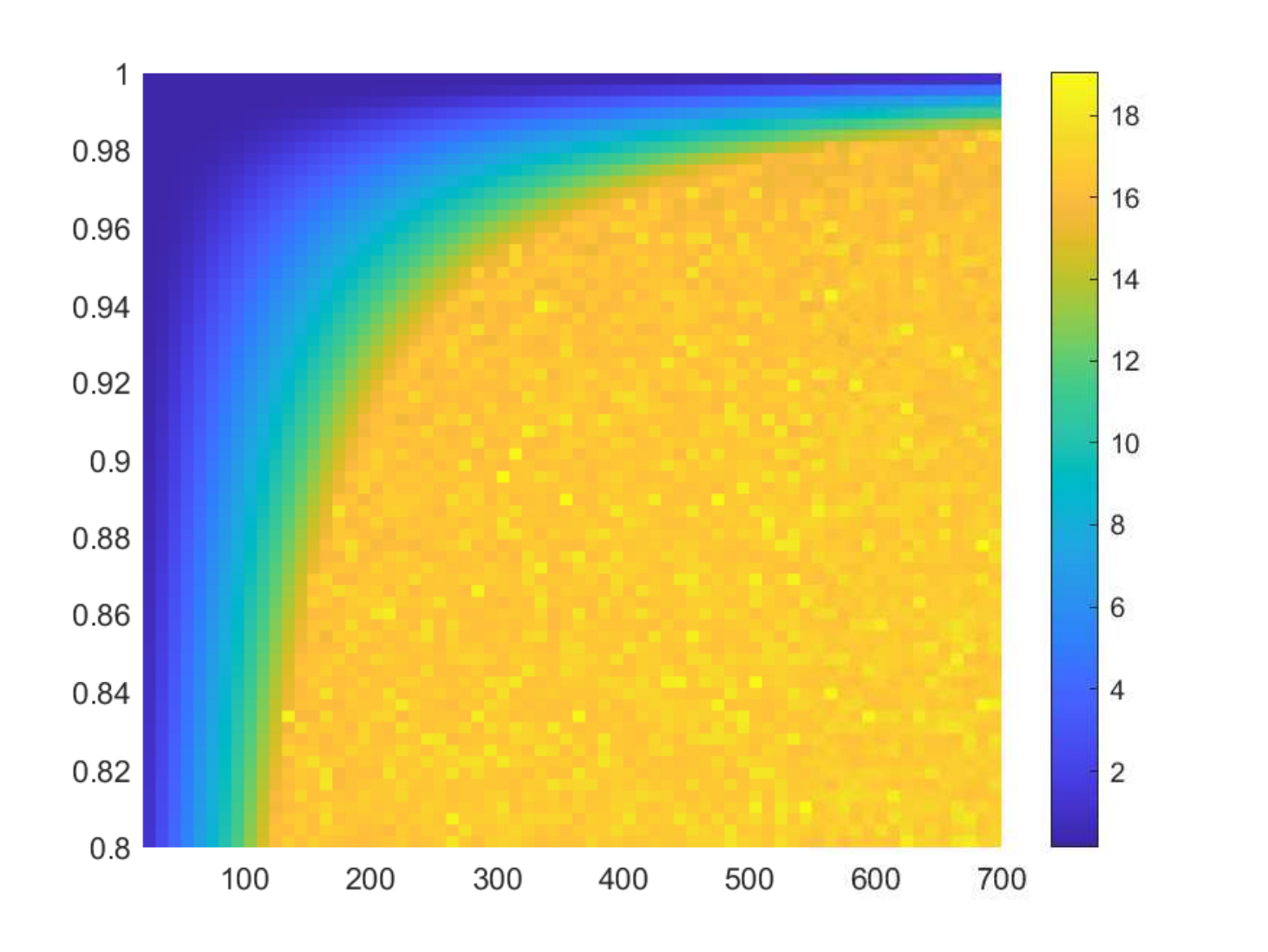}
		\caption*{$\frac{n}{m} = 1$}
	\end{subfigure}\hfil 
\caption{\wee{Condition numbers of the matrix $\mathbf{W} \mathbf{A}^{\alpha}$ in the least-squares problem \eqref{def_A_1} for different values of $m$ (size of the equispaced grid $\mathcal{X}$, displayed on the x-axis), the mapping parameter $\alpha$ (displayed on the y-axis) and the ratio $n/m$ between the degree $n$ of the polynomial space and $m$. The colors represent the value $\log_{10}(\textrm{Cond}(\mathbf{W} \mathbf{A}^{\alpha}))$.}}
\label{fig:condizionamento}
\end{figure}

\wee{Using the mapped Chebyshev polynomials as basis} allows to calculate the KTL quadrature weights $\boldsymbol{w}^{\alpha}$ in terms of a cosine and a non-equidistant fast Fourier transform. More precisely, the use of the Chebyshev basis leads to the moment vector $\boldsymbol{\tau}^{\alpha} = (\tau_{0}^{\alpha},...,\tau_{n}^{\alpha})^{\top}$ with the entries
\begin{equation}
    \tau_{i}^{\alpha} = \int_{-1}^{1} T_{i}(M_{\alpha}(x))\mathrm{d}x.
\end{equation}
Because of this particular structure, the moments $\tau_{i}^{\alpha}$ can be calculated by a cosine transform. 
\begin{theorem} \label{teo_dcs_1}
For $0 < \alpha \leq 1$ and $i \in \mathbb{N}_0$, the moment
$$
\tau_{i}^{\alpha} = \int_{-1}^{1} T_{i}(M_{\alpha}(x))\mathrm{d}x
$$
corresponds to the $i$-th coefficient $\mathcal{F}_{\cos}(g_{\alpha})(i)$ of the continuous cosine transform of the function
$$
g_{\alpha}(t) = \frac{\sin(t)}{\sqrt{\frac{1}{\sin^{2}(\alpha \pi / 2)}-\cos^{2}(t)}}\frac{1}{\alpha} ,\quad t \in [0,\pi].
$$
For $\alpha = 0$, the moments are given by 
\[\tau_{i}^{0} = \left\{ \begin{array}{ll} 0 & \text{if $i$ is odd}, \\ \frac{2}{1 - i^2} & \text{if $i$ is even}. \end{array} \right. \]
\end{theorem}
\begin{proof}
To compute the moments for $0 < \alpha \leq 1$, we use a change of variables:
\begin{equation}
    t(x)=\arccos\Bigg( \frac{\sin(\alpha \pi x / 2)}{\sin(\alpha \pi / 2)}\Bigg).
\end{equation}
By observing that
\begin{equation*}
	\begin{split}
		\frac{\mathrm{d}t}{\mathrm{d}x}(x) &= - \frac{1}{\sqrt{1-\cos^{2}(t(x))}}\frac{\alpha \pi}{2}\frac{\cos(\alpha \pi x / 2)}{\sin(\alpha \pi / 2)} \\
		&= - \frac{1}{\sin(t(x))}\frac{\alpha \pi}{2}\frac{\sqrt{1-\sin^{2}(\alpha \pi x / 2)}}{\sin(\alpha \pi / 2)} \\
		&=  - \frac{1}{\sin(t(x))}\frac{\alpha \pi}{2} \sqrt{\frac{1}{\sin^{2}(\alpha \pi / 2)}-\cos^{2}(t(x))}
	\end{split}
\end{equation*}
we obtain
$$
\int_{-1}^{1} T_{n}(M_{\alpha}(x))\mathrm{d}x = \frac{2}{\pi} \int_{0}^{\pi} \cos(nt) \underbrace{\frac{\sin(t)}{\sqrt{\frac{1}{\sin^{2}(\alpha \pi / 2)}-\cos^{2}(t)}}\frac{1}{\alpha}}_{g_{\alpha}(t)}\mathrm{d}t.
$$
As the cosine series of the Lipschitz-function $g_{\alpha}$ in $[0,\pi]$ is given by
\begin{equation}
    g_{\alpha}(t) = \frac{\mathcal{F}_{\cos}(g_{\alpha})(0)}{2} + \sum_{n=1}^{\infty} \mathcal{F}_{\cos}(g_{\alpha})(n)\cos(nt).
\end{equation}
we have showed that
$$
\int_{-1}^{1} T_{n}(M_{\alpha}(x))\mathrm{d}x = \mathcal{F}_{\cos}(g_{\alpha})(n).
$$
Finally, for $\alpha = 0$, we obtain
\[ \tau_{i}^{0} = \int_{-1}^{1} T_{i}(x)\mathrm{d}x = \int_0^\pi \cos (i t) \sin t \mathrm{d} t = \frac{1 + (-1)^i}{1 - i^2} \quad \text{if $i \neq 1$}. \]
For $i = 1$, the evaluation of the integral gives $\tau_{1}^{0} = 0$. 
\end{proof}

We conclude this section with a pseudo-code for the calculation of the KTL and KTI quadrature weights using the Kosloff-Tal-Ezer map \wee{and the mapped Chebyshev basis}. We summarize the main steps in Algorithm \ref{alg:1} and  Algorithm \ref{alg:2}.

\vspace{2mm}

\begin{algorithm}[H]  \label{alg:1}
	\SetAlgoLined
	\KwIn{\\ $\bullet\;$ $\mathcal{X} = \{ x_{i}, \, i=0,...,m\} \subseteq \mathbb{R}$ : quadrature nodes;\\
		$\bullet\;$ $n \leq m$ : polynomial degree for the approximation space;\\
		$\bullet\;$ $\boldsymbol{\mu}=(\sqrt{\mu_{0}} , ... ,\sqrt{\mu_{m}})^{\top}$ : weights for least-squares problem \eqref{mu_def_1};\\ 
		$\bullet\;$ $\mathbf{f} = (f(x_{0}),...,f(x_{m}))^{\top}$: sample vector of $f$ on $\mathcal{X}$;\\
		$\bullet\;$ $M_{\alpha} : [-1,1] \longrightarrow [-1,1]$: Kosloff-Tal-Ezer map with parameter $\alpha$.
	}
	\vspace{2mm}
	\Begin{
		Compute moments $\boldsymbol{\tau}^{\alpha} \in \mathbb{R}^{n+1}$ through discrete cosine transform of $g_{\alpha}$ (Thm. \ref{teo_dcs_1});\\
		Build diagonal matrix with the weights for the least-squares problem: $\mathbf{W} = \textrm{diag}(\sqrt{\mu_{0}} , ... ,\sqrt{\mu_{m}}) \in \mathbb{R}^{m+1} \times \mathbb{R}^{m+1}$;\\
		Construct matrix: $ \mathbf{A}^{\alpha} \in \mathbb{R}^{m+1} \times \mathbb{R}^{n+1}, \, \mathbf{A}_{ij}^{\alpha} = T_{j-1}(M_{\alpha}(x_{i-1}))$, for $i = 1, \ldots, m+1$, and $j = 1, \ldots, n+1$;\\
		Find coefficient vector $\boldsymbol{\gamma}$ as the solution of the least-squares problem $\mathbf{W} \mathbf{A}^{\alpha} \boldsymbol{\gamma}=\mathbf{W}\boldsymbol{f}$;\\
		Compute the quadrature value $\Int^{\alpha}_{n,\mathcal{X}}(f,I) = \boldsymbol{\gamma}^{\top} \boldsymbol{\tau}^{\alpha}.$
	}
	\KwOut{The value of the KTL quadrature $ \Int^{\alpha}_{n,\mathcal{X}}(f,I) $.}
	\caption{KTL quadrature formula}
\end{algorithm}

\newpage

\begin{algorithm}[H] \label{alg:2}
	\SetAlgoLined
	\KwIn{\\ $\bullet\;$ $\mathcal{X} = \{ x_{i}, \, i=0,...,m\} \subseteq \mathbb{R}$ : quadrature nodes, $m = n = \mathrm{\dim}(\mathbb{P}_n^{\alpha}) - 1$;\\
		$\bullet\;$ $\mathbf{f} = (f(x_{0}),...,f(x_{m}))^{\top}$: vector of function samples on $\mathcal{X}$;\\
		$\bullet\;$ $M_{\alpha} : [-1,1] \longrightarrow [-1,1]$: Kosloff-Tal-Ezer map with parameter $\alpha$.
	}
	\vspace{2mm}
	\Begin{
		Compute moments $\boldsymbol{\tau}^{\alpha} \in \mathbb{R}^{m+1}$ through discrete cosine transform of $g_{\alpha}$ (Thm. \ref{teo_dcs_1});\\
		Construct interpolation matrix: $ \mathbf{A}^{\alpha} \in \mathbb{R}^{m+1} \times \mathbb{R}^{m+1}, \, \mathbf{A}_{ij}^{\alpha} = T_{j-1}(M_{\alpha}(x_{i-1}))$; \\
		Find the quadrature weights $\boldsymbol{w}^{\alpha}$ as the solution of the linear system $(\mathbf{A}^{\alpha})^{\top} \boldsymbol{w}^{\alpha} = \boldsymbol{\tau}^{\alpha}$.;\\
		Compute the quadrature value $\Int^{\alpha}_{m,\mathcal{X}}(f,I) = (\boldsymbol{w}^{\alpha})^\top \boldsymbol{f}.$
	}
	\KwOut{The value of the KTI quadrature $ \Int^{\alpha}_{m,\mathcal{X}}(f,I) $.}
	\caption{KTI quadrature formula}
\end{algorithm}

\subsection{Why the monomial basis is not so suited for calculations}

We continue the previous discussion by analyzing the computation of the quadrature formula using the standard monomial basis instead of the Chebyshev basis. \wee{From a computational point of view the usage of the monomial basis is prohibitive also if an additional KT map is used. The main reason is that the matrix $\mathbf{A}^{\alpha}$ in the solution of the least-squares problem \ref{def_A_1} with the entries $\mathbf{A}^{\alpha}_{ij} = M_{\alpha}(x_{i-1})^{j-1}$ corresponds to a Vandermonde matrix that gets ill-conditioned already for small degrees $n$.} 

\wee{The monomial basis turns out to be problematic also in regard of the moment vectors $\boldsymbol{\tau}^{\alpha}$.}
If we use the monomial polynomial basis $ \{x^{i}: i=0, \ldots, n \}$ for the space $\mathbb{P}_n$ it is possible to express the moment vector
$\boldsymbol{\tau}^{\alpha}$ explicitly. We have
\begin{equation} \tau_i^{\alpha} = \int_{-1}^{1} M_{\alpha}(x)^i \textrm{d}x = 
    \int_{-1}^{1} \Big( \frac{\sin \big(  \alpha \frac{\pi}{2} x\big)}{\sin\big(  \alpha \frac{\pi}{2}\big)} \Big)^{i} \textrm{d}x = \frac{1}{\sin\big(  \alpha \frac{\pi}{2}\big)^{i}}\int_{-1}^{1} \sin \big(  \alpha \frac{\pi}{2} x\big)^{i}\textrm{d}x
\end{equation}
For the calculation of the integral $\int_{-1}^{1} \sin ( \alpha \frac{\pi}{2} x)^{i}\textrm{d}x$, we can make use of the following recursive formula:
\begin{lemma} \label{lem:1}
Let $C\in \mathbb{R} \setminus \{ 0\} $, then the following recursive formula holds for even $i \geq 2$:
	$$
	\int_{-1}^{1} \sin ( C x)^{i}\textrm{d}x = -\frac{1}{i}\Bigg[ \sin ( C x)^{i-1}\frac{\cos ( C x)}{C} \Bigg]_{-1}^{1} + \frac{(i-1)}{i}\int_{-1}^{1} \sin ( C x)^{i-2}\textrm{d}x.
	$$
For odd numbers $i \geq 1$, we have 
	$$
	\int_{-1}^{1} \sin ( C x)^{i}\textrm{d}x = 0.
	$$
\end{lemma}
\begin{proof}
For odd numbers $i$ the statement is trivial. For even $i \geq 2$, integration by parts yields
\begin{equation*}
	\begin{split}
		& \int_{-1}^{1} \sin ( C x)^{i}\textrm{d}x = \int_{-1}^{1} \sin ( C x)^{i-1}\sin ( C x)\textrm{d}x = \\
		& = \Bigg[ \sin ( C x)^{i-1}\frac{-\cos ( C x)}{C} \Bigg]_{-1}^{1}+\int_{-1}^{1} (i-1)\sin ( C x)^{i-2}\cos ( C x)^{2}\textrm{d}x = \\
		& = \Bigg[ \sin ( C x)^{i-1}\frac{-\cos ( C x)}{C} \Bigg]_{-1}^{1}+(i-1)\int_{-1}^{1} \sin ( C x)^{i-2}\textrm{d}x -(i-1)\int_{-1}^{1} \sin ( C x)^{i}\textrm{d}x.
	\end{split}
\end{equation*}
Thus, we obtain
\begin{equation*}
    \int_{-1}^{1} \sin ( C x)^{i}\textrm{d}x = -\frac{1}{i}\Bigg[ \sin ( C x)^{i-1}\frac{\cos ( C x)}{C} \Bigg]_{-1}^{1} + \frac{(i-1)}{i}\int_{-1}^{1} \sin ( C x)^{i-2}\textrm{d}x.
\end{equation*}
\end{proof}
Although Lemma \ref{lem:1} provides a simple scheme to calculate the moment vector $\boldsymbol{\tau}^{\alpha}$, we show why from a computational point of view it makes little sense to compute the moments in this way. For this, we suppose that $S_i$ is a sequence of numbers satisfying the recursion 
\begin{gather} \label{eq:recursion}
    S_i = -\frac{1}{i}\Bigg[ \sin ( C x)^{i-1}\frac{\cos ( C x)}{C} \Bigg]_{-1}^{1} + \frac{(i-1)}{i} S_{i-2}, \quad i \geq 2,
\end{gather}
and in which the initial value $S_0$ is a slight perturbation of the exact moment value $\int_{-1}^1 1 \mathrm{d}x = 2$. Then, by Lemma \ref{lem:1}, the error $\mathcal{E}_i$ between $S_i$ and $\int_{-1}^1 \sin ( C x)^{i}\textrm{d}x$ satisfies the recurrence relation
\begin{equation}
	\begin{split}
		\mathcal{E}_{i}& = S_{i}-\int_{-1}^{1} \sin ( C x)^{i}\textrm{d}x = \\
		&= \frac{(i-1)}{i} S_{i-2} - \frac{(i-1)}{i}\int_{-1}^{1} \sin ( C x)^{i-2}\textrm{d}x = \frac{(i-1)}{i}\mathcal{E}_{i-2}.
	\end{split}
\end{equation}
We will only consider the case when $i = 2k$ is even (the case $i$ odd is not relevant as the odd moments are already known). In this case, we get for the errors
\begin{gather*}
	\mathcal{E}_{2k}=\frac{(2k-1)}{2k}\mathcal{E}_{2(k-1)} = \frac{(2k-1)}{2k}\frac{(2(k-1)-1)}{2(k-1)}\mathcal{E}_{2(k-2)} = \\
	= \frac{(2k-1)}{2k}\frac{(2(k-1)-1)}{2(k-1)} \cdots \frac{(2(k-(k-1)+1)-1)}{2(k-(k-1)+1)} \mathcal{E}_{2(k-(k-1))} = \\
	= \frac{(2k-1)}{2k}\frac{(2(k-1)-1)}{2(k-1)} \cdots \frac{3}{4} \, \frac{1}{2}\mathcal{E}_{0}.
\end{gather*}
We observe that
\begin{equation*}
    \frac{(2(k+1)-1)}{2(k+1)} \geq \frac{(2k-1)}{2k} \Leftrightarrow 4k(k+1)-2k \geq 2(k+1)(2k-1) \Leftrightarrow 4k^{2}+2k \geq 4k^{2}+2k-2.
\end{equation*}
We show that for $k \longrightarrow \infty$ error diverges. We fix $k^{\star} < k$, then
\begin{equation*}
    \frac{\mathcal{E}_{2k}}{\sin\big( \alpha \frac{\pi}{2}\big)^{2k}} \geq \frac{1}{\sin\big(\alpha \frac{\pi}{2}\big)^{2k}} \Bigg[ \Big(\frac{2k^{\star}-1}{2k^{\star}}\Big)^{k-k^{\star}}\Big(\frac{1}{2}\Big)^{k^{\star}}\Bigg]\mathcal{E}_{0} = \frac{1}{\sin\big(\alpha \frac{\pi}{2}\big)^{2k}} \Big(\frac{2k^{\star}-1}{2k^{\star}}\Big)^{k} \mathcal{E}_{0} \tilde{G}
\end{equation*}
where
\begin{equation*}
    \tilde{G} = \Big(\frac{2k^{\star}-1}{2k^{\star}}\Big)^{k^{\star}}\Big(\frac{1}{2}\Big)^{k^{\star}}.
\end{equation*}
Now, if $\alpha \neq 1$ then $\sin\big( \alpha \frac{\pi}{2}\big)^{2}<1$. Furthermore, since $\frac{2k^{\star}-1}{2k^{\star}} \longrightarrow 1$ for $k^{\star} \longrightarrow \infty$, there exists a $k^{\star}$ such that $\frac{2k^{\star}-1}{2k^{\star}} > \sin\big( \alpha \frac{\pi}{2}\big)^{2}$. Since $\tilde{G}$ does not depend on $k$ we can state that
\begin{equation}
	\frac{\mathcal{E}_{2k}}{\sin\big( \alpha \frac{\pi}{2}\big)^{2k}} \longrightarrow \infty \textrm{ for } k \longrightarrow \infty.
	\label{equa_div_bmon}
\end{equation}
This implies that the calculation of the moments 
$\tau_i^{\alpha} = \int_{-1}^1 M_{\alpha}(x)^i \mathrm{d}x$ via the recursion formula \eqref{eq:recursion} is not stable.  

\subsection{Symmetry of the KTI weights}

For the quadrature nodes $\mathcal{X} = \{x_{i} : i=0,\ldots,m\} \subset I$ we denote by $z_i = M_{\alpha}(x_{i})$ the respectively mapped nodes. Then, by the simple change of variables $y = M_{\alpha}(x)$, the interpolatory KTI quadrature weights can be represented as
\begin{equation} \label{eq:intweights}
    w_{i}^{\alpha} = \int_{-1}^{1} \ell_{i}^{\alpha}(M_{\alpha}(x))\textrm{d}x = \frac{1}{\alpha \pi} \int_{-1}^{1} \ell_{i}^{\alpha}(y) \frac{2 \sin(\alpha \pi/2)}{\sqrt{1-\sin(\alpha \pi/2)^{2}y^{2}}}\textrm{d}y,
\end{equation}
where 
$$
\ell_{i}^{\alpha}(y) = \prod_{\substack{j = 0 \\ j\neq i}}^{m} \frac{y-z_{j}}{z_{i}-z_{j}}
$$
denotes the $i$-th Lagrange polynomial with respect to the mapped nodes $z_i = M_{\alpha}(x_{i})$, $i=0,\ldots,m$. 
We next show a result regarding the symmetry of the weights of the interpolatory KTI quadrature scheme. \wee{Such a symmetry of the quadrature weights is useful in cases when the integral of an even or odd function $f$ has to be calculated.}

\begin{theorem}
If the nodes $\mathcal{X} = \{x_0, \ldots, x_m\} \subset I$ are symmetric with respect to the origin, i.e., $x_{i}+x_{m-i}=0$ for $i = 0, \ldots, m$, then also the KTI weights satisfy the symmetry relations $w_i^{\alpha} = w_{m-i}^{\alpha}$.  
\end{theorem}

\begin{proof}
As we consider the interpolatory KTI quadrature formulas, the respective weights $w_i^{\alpha}$ can be computed as the integrals of the mapped Lagrange basis $\lambda_{i}^{\alpha}(x)$ relative to the nodes $x_i$, $i = 0, \ldots, m$. As $x_i$ are symmetric with respect to the origin, also the mapped nodes $z_i = M_{\alpha}(x_i)$ are symmetric and we get
\begin{equation*}
	\begin{split}
		\ell_{i}^{\alpha}(M_{\alpha}(x)) &= \prod_{\substack{j = 0 \\ j\neq i}}^{m} \frac{M_\alpha(x)-z_{j}}{z_{i}-z_{j}} = \prod_{\substack{j = 0 \\ j\neq i}}^{m} \frac{M_\alpha(x)+z_{m-j}}{(-z_{m-i})-(-z_{m-j})} \\
		&= \prod_{\substack{j = 0 \\ j\neq i}}^{m} \frac{-M_{\alpha}(x)-z_{m-j}}{z_{m-i}-z_{m-j}} = \prod_{\substack{j = 0 \\ j\neq m-i}}^{m} \frac{M_{\alpha}(-x)-z_{j}}{z_{m-i}-z_{j}} = \ell_{m-i}^{\alpha}(M_{\alpha}(-x)).
	\end{split}
\end{equation*}
This equation and the representation \eqref{eq:intweights} for the interpolatory quadrature formula give the identity
\begin{equation*}
	\begin{split}
		w_{i}^{\alpha} &= \int_{-1}^{1} \ell_{i}^{\alpha}(M_{\alpha}(x))\textrm{d}x = \int_{-1}^{1} \ell_{m-i}^{\alpha}(M_{\alpha}(-x))\textrm{d}x= \int_{-1}^{1} \ell_{m-i}^{\alpha}(M_{\alpha}(x))\textrm{d}x =  w_{m-i}^{\alpha},
	\end{split}
\end{equation*}
and therefore the symmetry of the KTI quadrature weights.
\end{proof}

\subsection{The KTI weights in the limit cases $\alpha \to 0^+$ and $\alpha \to 1^-$}
\label{sub_sec_limit_alpha}

\begin{figure}[htb]
		\centering 
	\begin{subfigure}{0.25\textwidth}
		\includegraphics[width=\linewidth]{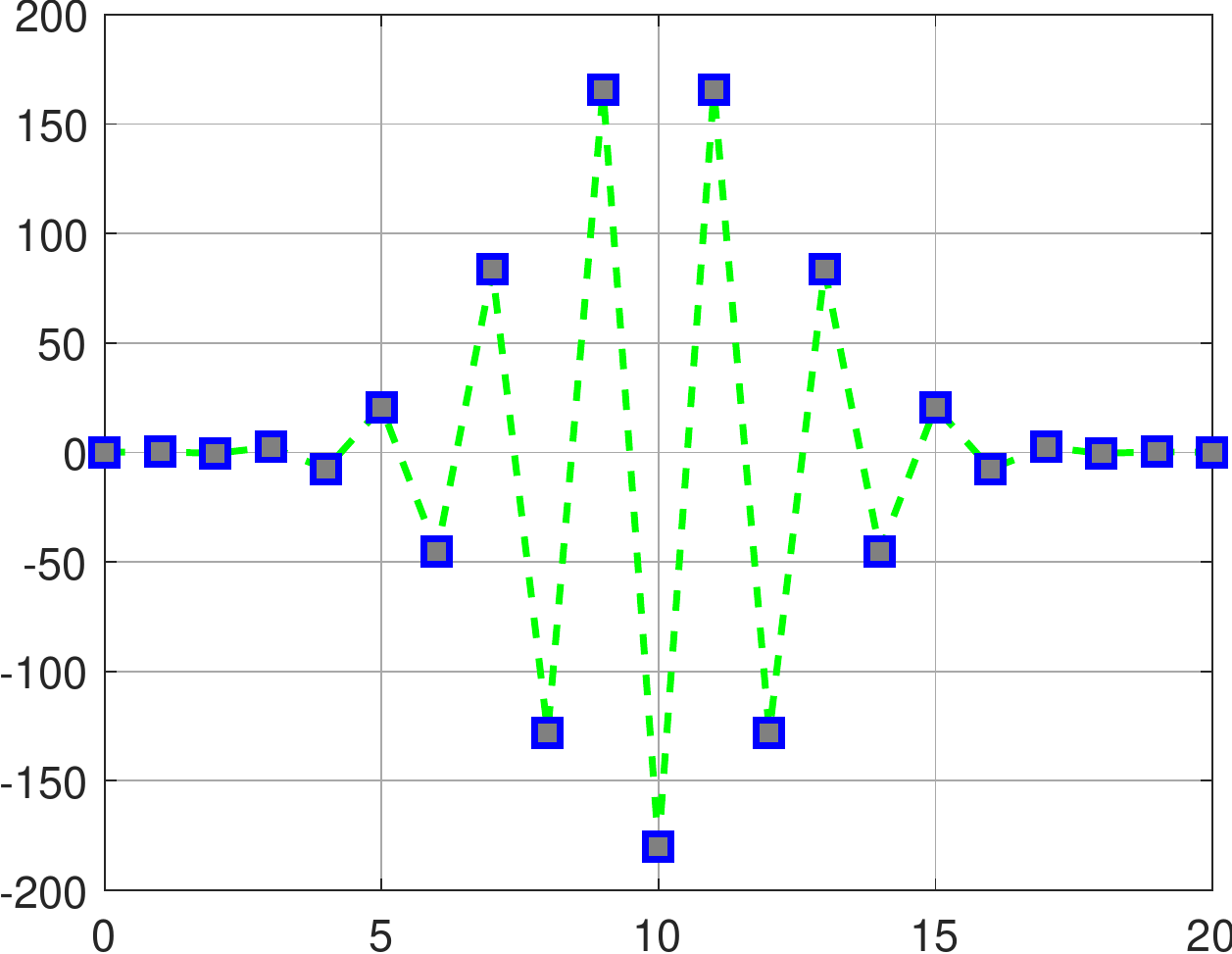}
		\caption*{$\alpha = 0$}
	\end{subfigure}\hfil 
	\begin{subfigure}{0.25\textwidth}
		\includegraphics[width=\linewidth]{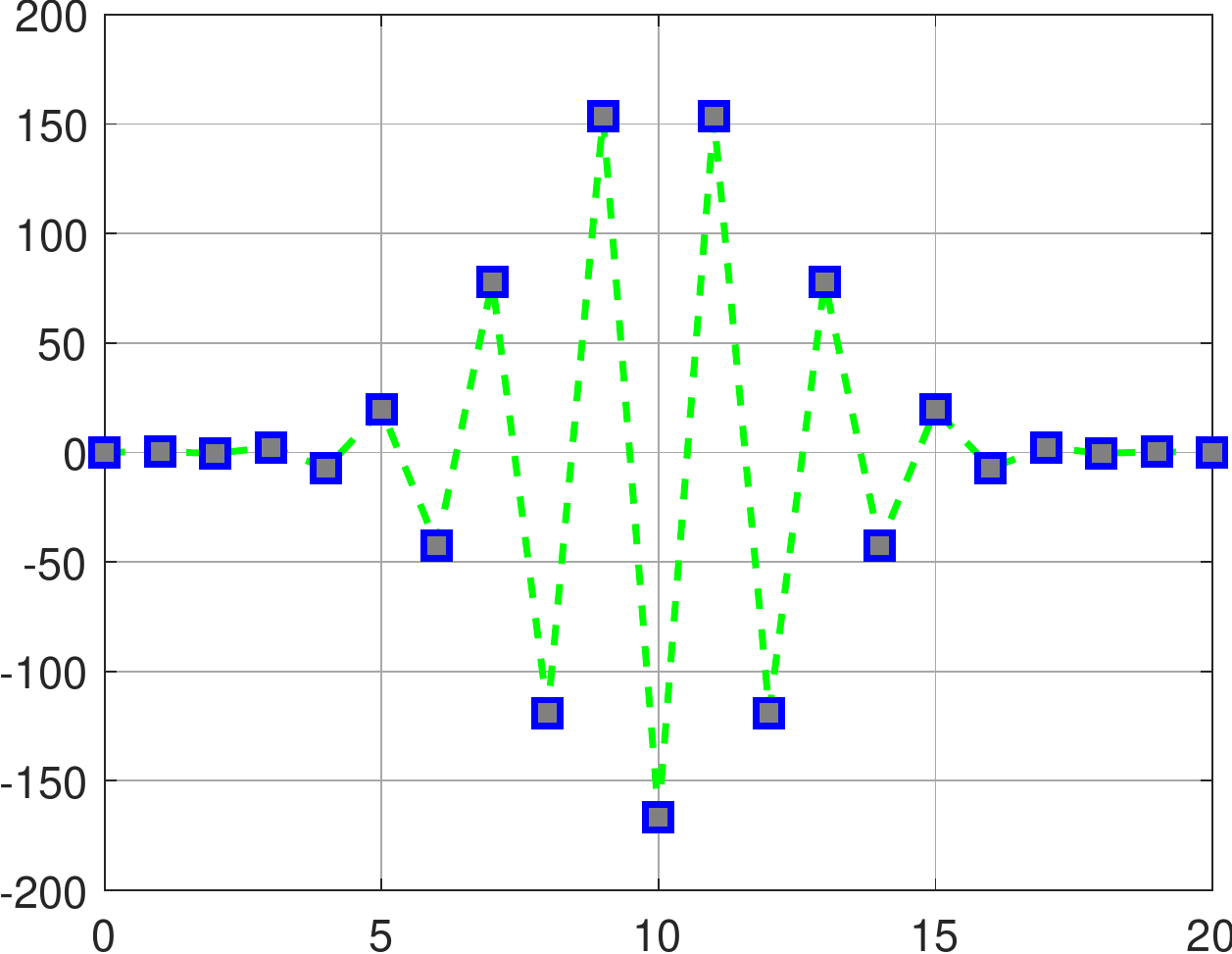}
				\caption*{$\alpha = 0.10$}
	\end{subfigure}\hfil 
	\begin{subfigure}{0.25\textwidth}
		\includegraphics[width=\linewidth]{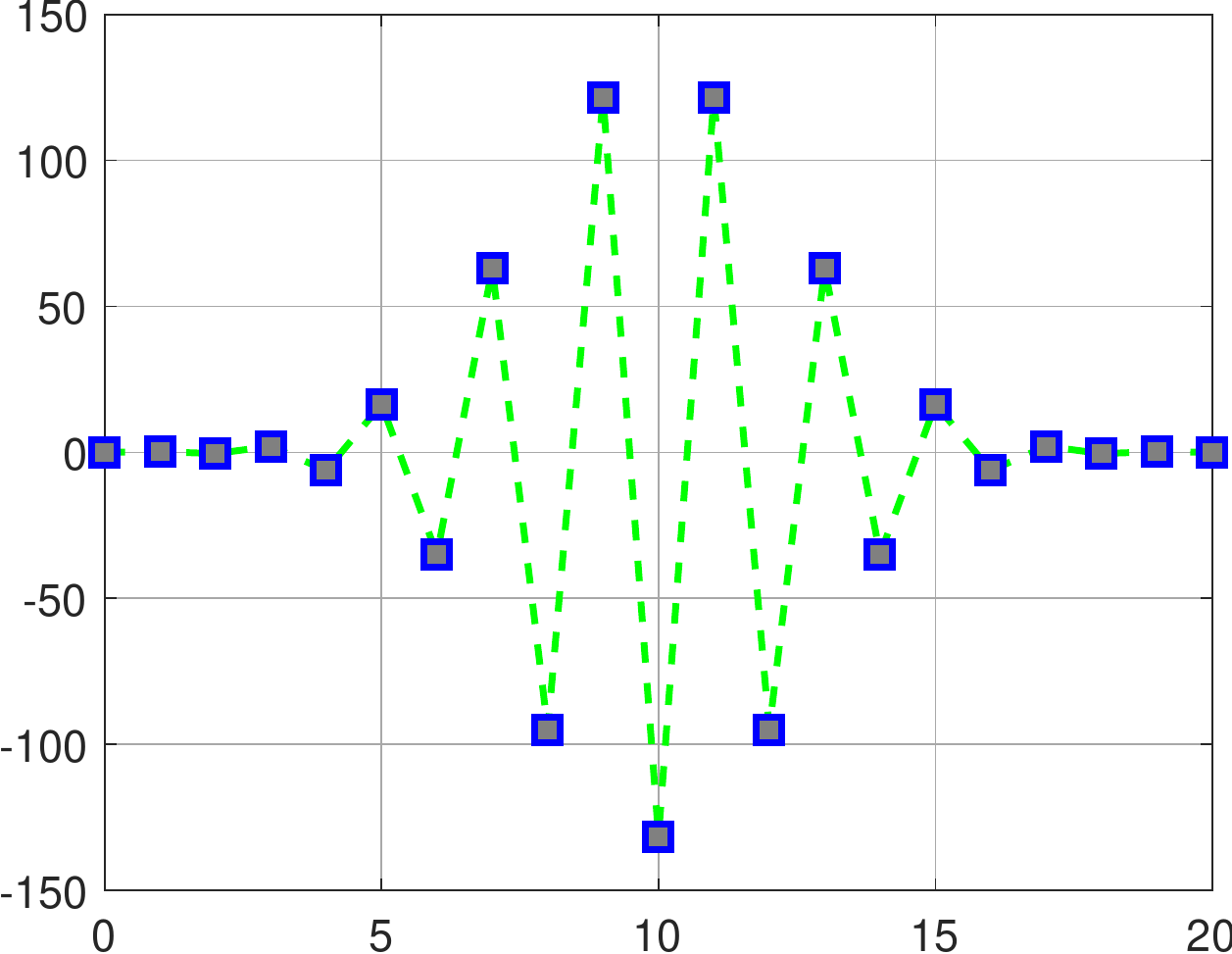}
		\caption*{$\alpha = 0.20$}
	\end{subfigure}\hfil 
	
	\begin{subfigure}{0.25\textwidth}
		\includegraphics[width=\linewidth]{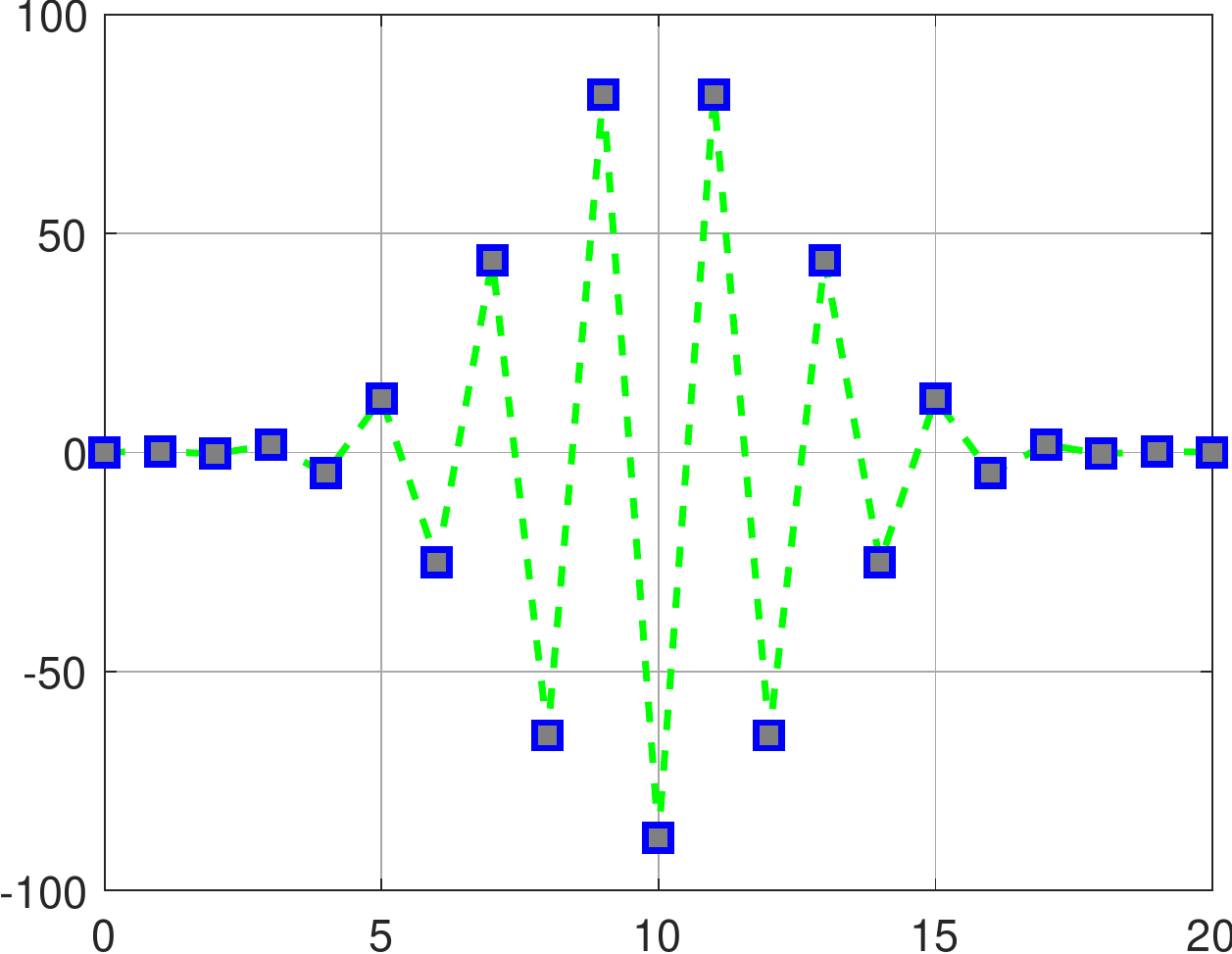}
		\caption*{$\alpha = 0.30$}
	\end{subfigure}\hfil 
	\begin{subfigure}{0.25\textwidth}
		\includegraphics[width=\linewidth]{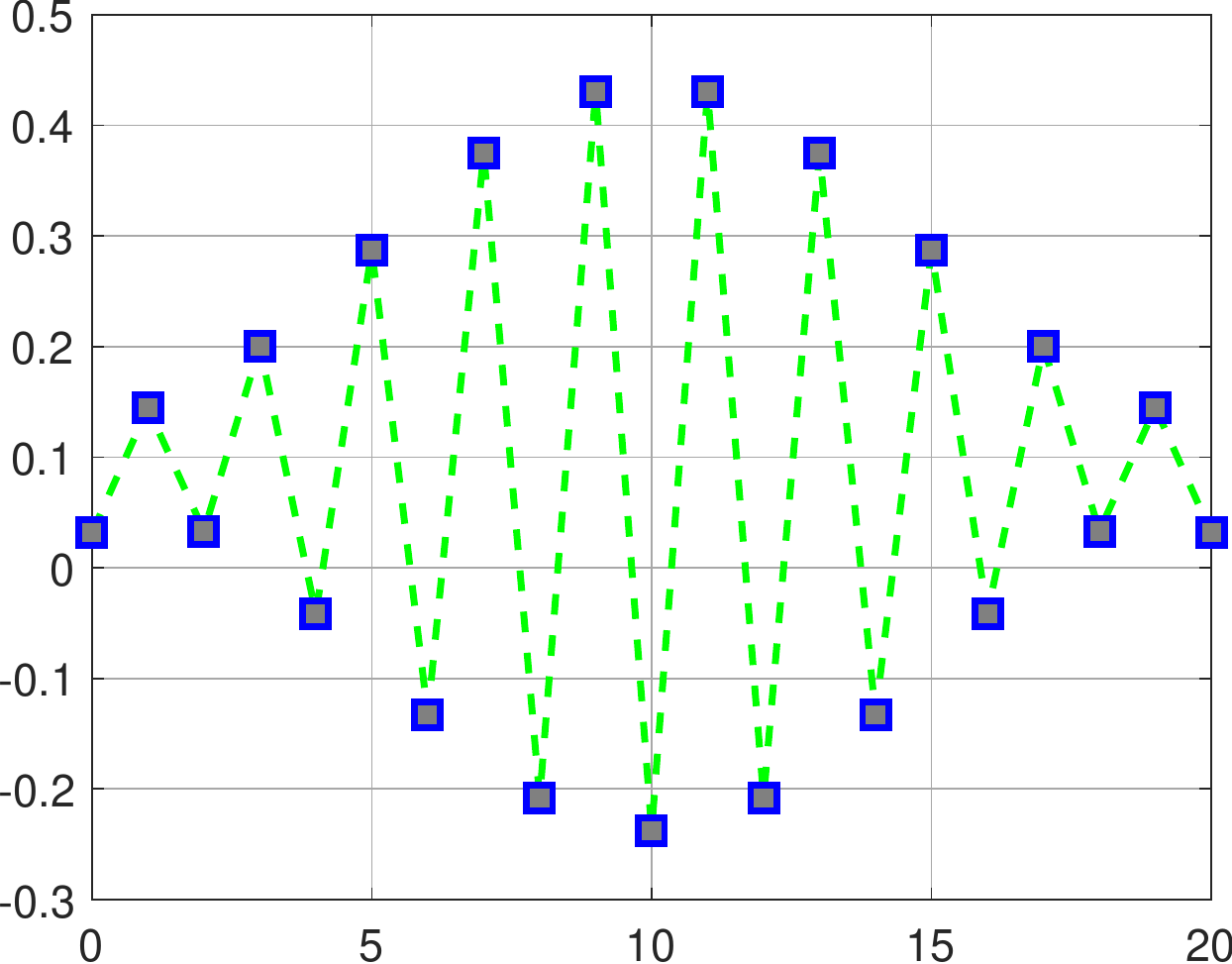}
		\caption*{$\alpha = 0.80$}
	\end{subfigure}\hfil 
	\begin{subfigure}{0.25\textwidth}
		\includegraphics[width=\linewidth]{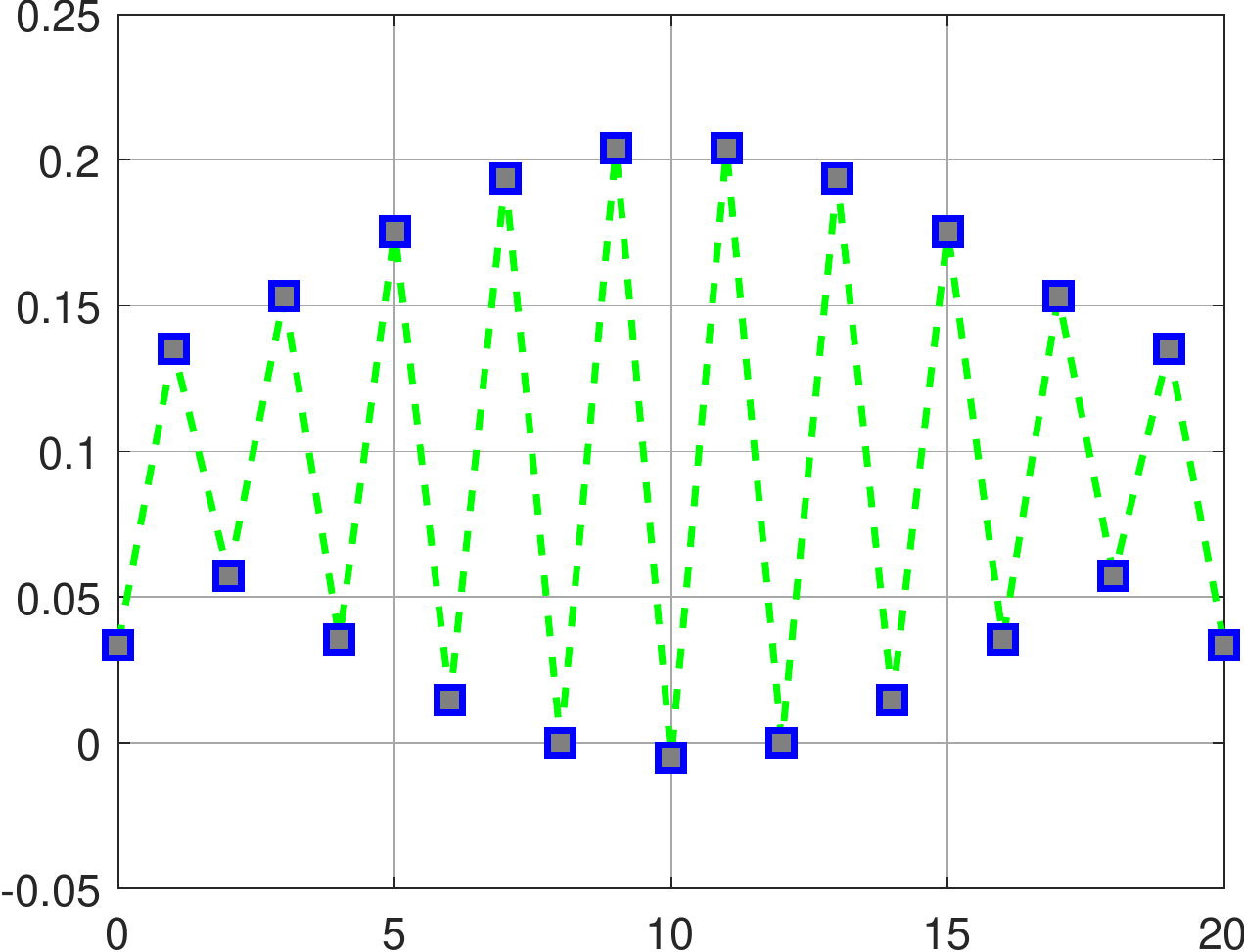}
		\caption*{$\alpha = 0.85$}
	\end{subfigure}\hfil 

	\begin{subfigure}{0.25\textwidth}
		\includegraphics[width=\linewidth]{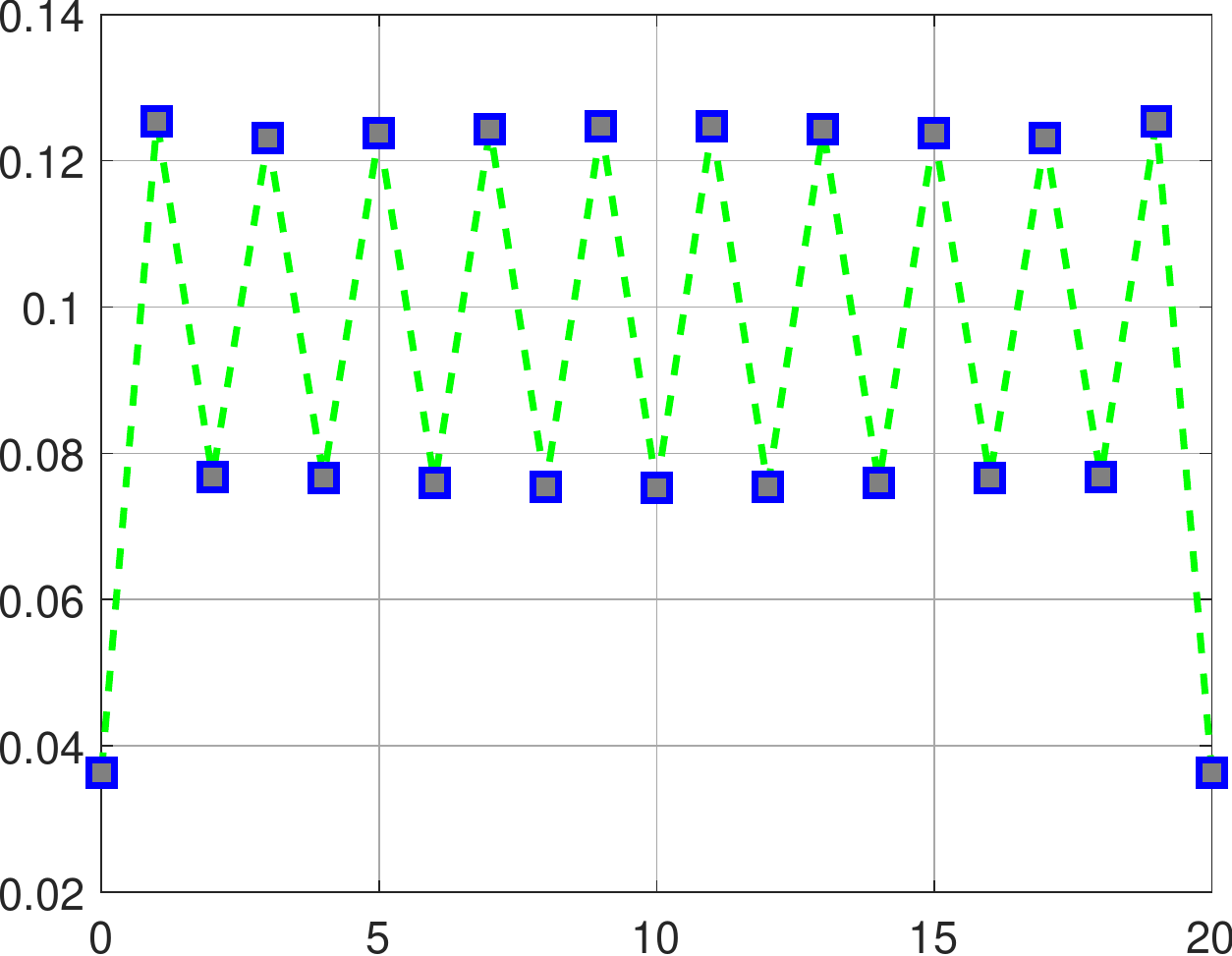}
		\caption*{$\alpha = 0.90$}
	\end{subfigure}\hfil 
	\begin{subfigure}{0.25\textwidth}
		\includegraphics[width=\linewidth]{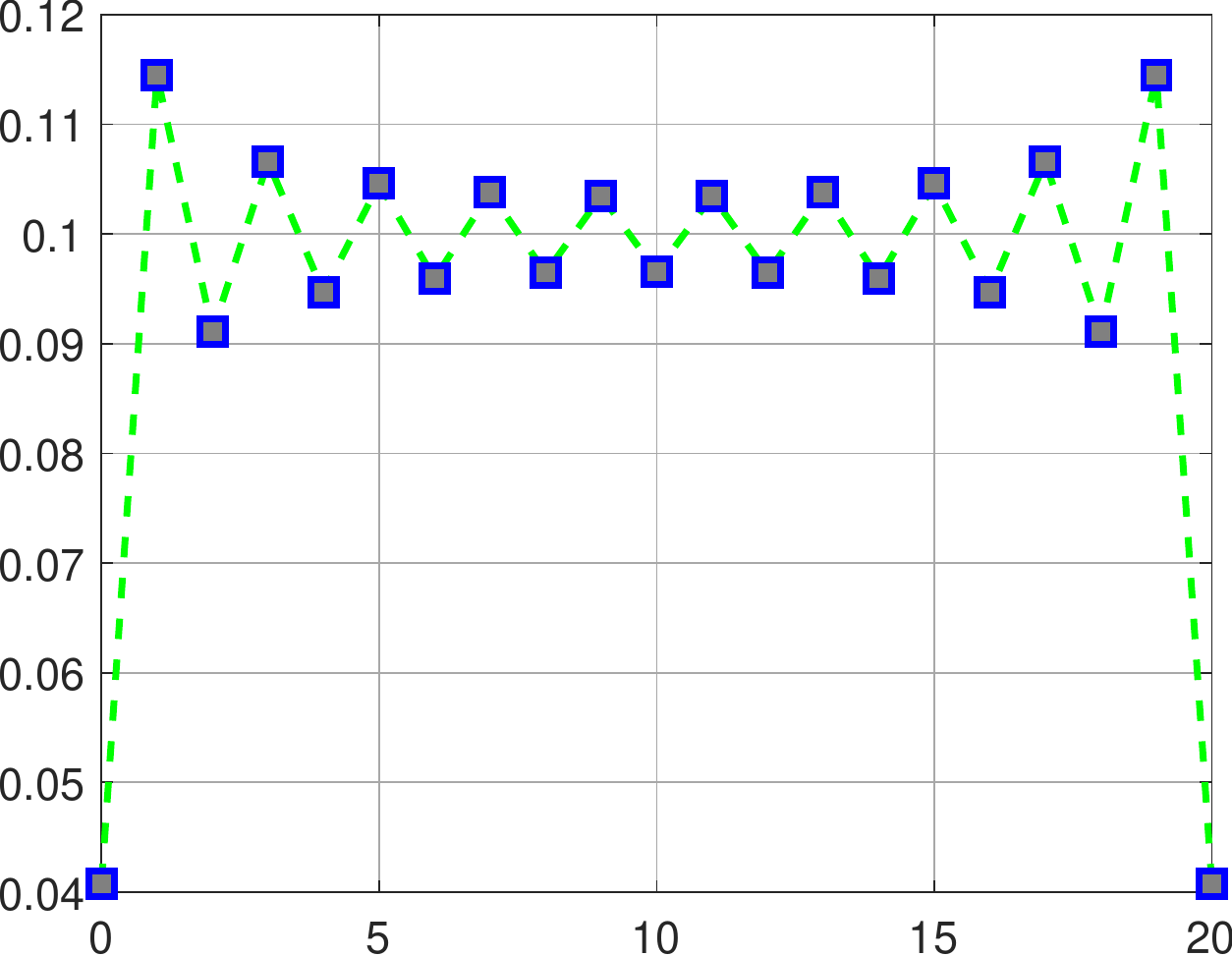}
		\caption*{$\alpha = 0.95$}
	\end{subfigure}\hfil 
	\begin{subfigure}{0.25\textwidth}
		\includegraphics[width=\linewidth]{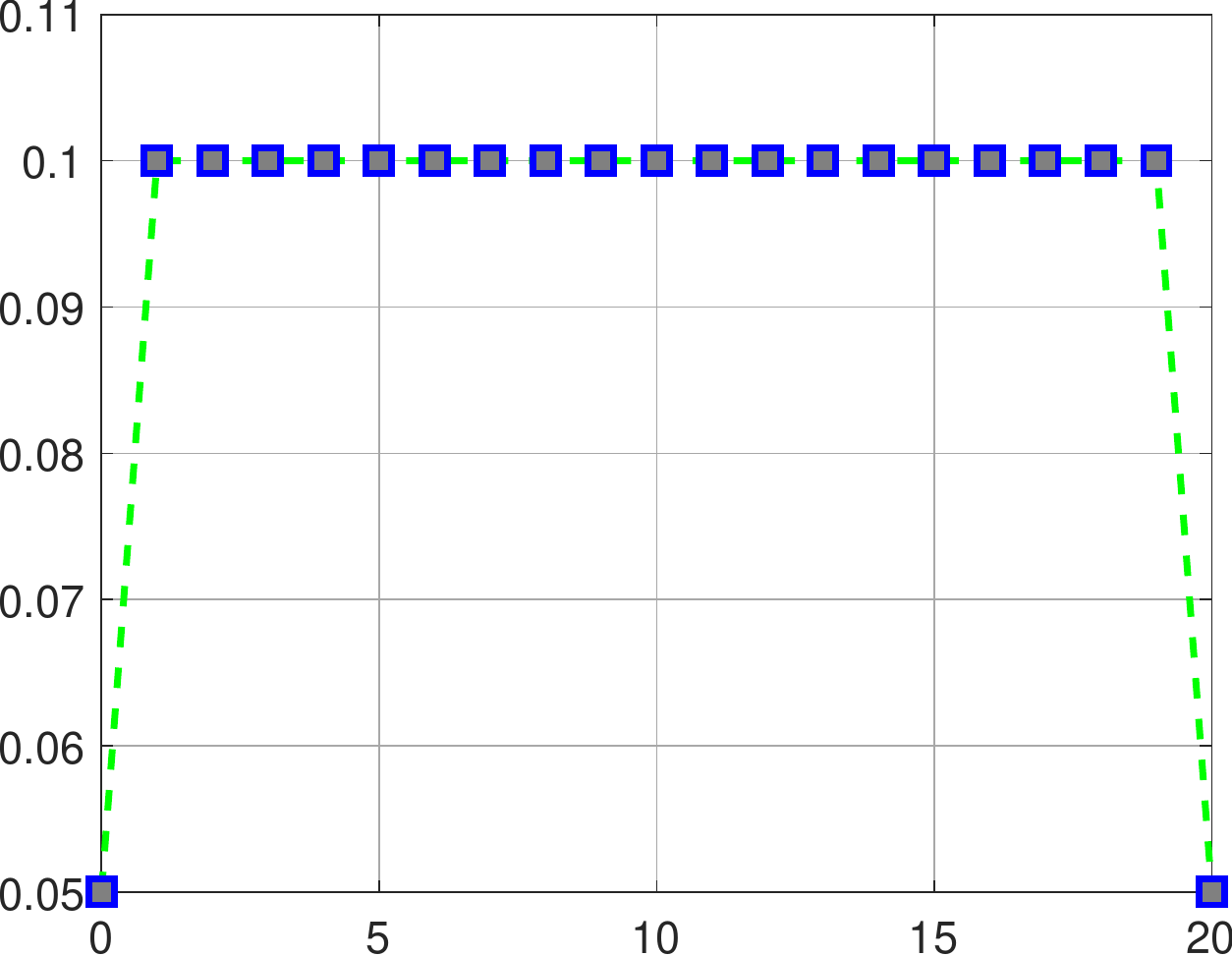}
		\caption*{$\alpha = 1$}
	\end{subfigure}\hfil 
\caption{KTI quadrature weights $ \{ w_{i}^{\alpha} : i=0, ... , 20 \}$ for 21 equispaced nodes in the interval $[-1,1]$ as $\alpha$ varies from $0$ to $1$.}
\label{fig:KTIalpha}
\end{figure}

We use again the notation of the last section, and, in particular, the representation \eqref{eq:intweights} for the interpolatory quadrature weights. For $\alpha \to 0^+$ and $\alpha \to 1^-$, we get the following limit results.

\begin{theorem}
    For the interpolatory KTI quadrature formula, we have the limit relations
	$$
	w_{i}^{\alpha} = \frac{2}{\alpha \pi} \int_{-1}^{1} \ell_{i}^{\alpha}(y) \frac{\sin(\alpha \pi/2)}{\sqrt{1-\sin(\alpha \pi/2)^{2}y^{2}}}\textrm{d}y \, \xrightarrow{\alpha \longrightarrow \, 0^{+}} \int_{-1}^{1} \ell_{i}^{0}(y) \textrm{d}y=\int_{-1}^{1} \ell_{i}(y) \textrm{d}y
	$$
	and
	$$
	w_{i}^{\alpha} = \frac{2 }{\alpha \pi} \int_{-1}^{1} \ell_{i}^{\alpha}(y) \frac{\sin(\alpha \pi/2)}{\sqrt{1-\sin(\alpha \pi/2)^{2}y^{2}}}\textrm{d}y \, \xrightarrow{\alpha \longrightarrow \, 1^{-}} \frac{2}{\pi} \int_{-1}^{1} \ell_{i}^{1}(y) \frac{1}{\sqrt{1-y^{2}}}\textrm{d}y.
	$$
	\label{teo_conv_pesi_quad_1}
\end{theorem}

\begin{proof}
We show the two limit relations by using Lebesgue's dominated convergence theorem.  
We first observe that
$$
|y-M_{\alpha}(x_{j})| \leq 2, \quad y \in [-1,1]
$$
and
$$
\sin(\alpha \pi/2)^{2} y^{2} \leq y^{2} \; \Leftrightarrow \; 1-\sin(\alpha \pi/2)^{2}y^{2} \geq 1-y^{2} \; \Leftrightarrow \;  \frac{1}{\sqrt{1-y^{2}}} \geq \frac{1}{\sqrt{1-\sin(\alpha \pi/2) y^{2}}}.
$$
We are now looking for an upper bound of $|\ell^{\alpha}_{i}|$ for $\alpha \longrightarrow 1^{-}.$ Elementary properties of sine and cosine give the inequality
\begin{equation*}
	\begin{split}
		 & \sin \Big(\frac{\alpha \pi}{2}x_{j+1}\Big)-\sin\Big(\frac{\alpha \pi}{2}x_{j}\Big) = 2 \sin\Big(\frac{\alpha \pi}{2}\frac{(x_{j+1}-x_{j})}{2}\Big) \cos\Big(\frac{\alpha \pi}{2}\frac{(x_{j+1}+x_{j})}{2}\Big) \geq \\
		 & \geq 2 \sin\Big(\frac{\bar{\alpha} \pi}{2}\frac{x_{j+1} - x_j}{2}\Big) \cos\Big(\frac{\pi}{2}\frac{(x_{j+1}+x_{j})}{2}\Big)\quad \textrm{for } 1 \geq \alpha \geq \bar{\alpha} \geq 0.
	\end{split}
\end{equation*}
As $x_{j+1} - x_j > 0$ and $x_{j+1} + x_j < 2$ for all $j = 0, \ldots, m-1$, we can thus find an $\epsilon > 0$ such that
\[\sin \Big(\frac{\alpha \pi}{2}x_{j+1}\Big)-\sin\Big(\frac{\alpha \pi}{2}x_{j}\Big) \geq \epsilon \quad \text{for $\alpha \geq \bar{\alpha}$}.\]
We therefore get the bound
$$
 |\ell^{\alpha}_{i}(y) | = \prod_{\substack{j = 0 \\ j\neq i}}^{m} \frac{|y-M_{\alpha}(x_{j})|}{|M_{\alpha}(x_{i})-M_{\alpha}(x_{j})|} \leq \frac{2^{m}}{\epsilon^{m}}, \qquad \alpha \geq \bar{\alpha},
$$
and therefore
$$
| \ell^{\alpha}_{i}(y) | \frac{1}{\sqrt{1-\sin(\alpha \pi/2)^{2} y^{2}}} \leq \frac{2^{m}}{\epsilon^{m}} \frac{1}{\sqrt{1-y^{2}}} \in L^{1}(I), \qquad \alpha \geq \bar{\alpha}.
$$
Similarly, we are also looking for an upper bound of $|\ell^{\alpha}_{i}|$ for $\alpha \longrightarrow 0^{+}$. Using the mean value theorem we obtain
$$
|M_{\alpha}(x_{j+1})-M_{\alpha}(x_{j})| = \left| \frac{\sin(\frac{\alpha \pi}{2}x_{j+1})-\sin(\frac{\alpha \pi}{2}x_{j})}{\sin(\frac{\alpha \pi}{2})} \right| = \frac{|\cos(\zeta)| \frac{\alpha \pi}{2} (x_{j+1}-x_{j})}{\sin(\frac{\alpha \pi}{2})},
$$
where $-\alpha \pi / 2 \leq \zeta \leq \alpha \pi / 2.$ Continuing the computations, we get the lower bound
$$
 \geq \frac{|\cos(\alpha \pi / 2)| \frac{\alpha \pi}{2} (x_{j+1}-x_{j})}{\sin(\frac{\alpha \pi}{2})} \geq |\cos(\alpha \pi / 2)| (x_{j+1}-x_{j}) \geq |\cos(\bar{\alpha} \pi / 2)| (x_{j+1}-x_{j}) \geq \epsilon
$$
for $0 \leq \alpha \leq \bar{\alpha} \leq 1$, where the second inequality arises from $|\sin(x)| \leq |x|$ for all $x \in \mathbb{R}$. 
We therefore get also in this second case that
$$
| \ell^{\alpha}_{i}(y) | \frac{1}{\sqrt{1-\sin(\alpha \pi/2)^{2} y^{2}}} \leq \frac{2^{m}}{\epsilon^{m}} \frac{1}{\sqrt{1-y^{2}}} \in L^{1}(-1,1), \qquad \alpha \leq \bar{\alpha}.
$$
Lebesgue's dominated convergence theorem now guarantees that we can pass to the limit inside the integral for $\alpha \longrightarrow 1^{-}$ and $\alpha \longrightarrow 0^{+}$. This immediately proves the statement.
\end{proof}
Based on this theorem we obtain more precise results for \we{certain types of nodes.}
If the starting nodes $\mathcal{X}$ are equispaced, the resulting KTI quadrature rules for $\alpha \to 1^-$ turn out to be particular composite Newton-Cotes formulas. As we have $M_{1}(x)=\sin(\pi / 2 x)$, we get for instance
$$
\sin\Big(\frac{\pi}{2} \Big(-1 +\frac{2i}{m}\Big)\Big) = \sin\Big(-\frac{\pi}{2} +\frac{\pi i}{m}\Big) = - \cos\Big(\frac{\pi i}{m}\Big), \quad i=0,...,m
$$
the Chebyshev-Lobatto nodes as mapped nodes, and the limits $w_{i}^{\alpha}$ for $\alpha \longrightarrow 1^{-}$ are the composite trapezoidal rule weights in $[-1,1]$ (Example \ref{form_trap_teo}).\\
On the other hand, if the starting nodes are the equidistant nodes $\{ x_{k} = -1+(2k+1)/(m+1), \, k=0,\ldots,m\}$ then the limits of the weights $w_{i}^{\alpha}$ for $\alpha \longrightarrow 1^{-}$ are the composite midpoint rule weights in $[-1,1]$ (Example \ref{form_mid_teo}).

\wee{In Fig. \ref{fig:KTIalpha}, the KTI weights on $21$ equidistant nodes for a varying parameter $\alpha$ are illustrated.} For $\alpha = 0$, we get the closed Newton-Cotes quadrature weights, whilst with $\alpha = 1$ we get the weights of the composite trapezoidal rule. \wee{Well-conditioning of numerical integration can be guaranteed by positive quadrature weights}. Therefore, from a computational point of view, a method parameter $\alpha$ close to $1$ is preferable. In Fig. \ref{fig_pesi_137} we can see that by increasing the number of nodes and if $ \alpha \nrightarrow \, 1^{-}$, then some of the weights get negative. This inevitably leads to numerical instability.

\begin{figure}[h]
	\begin{subfigure}{0.25\textwidth}
		\includegraphics[width=\linewidth]{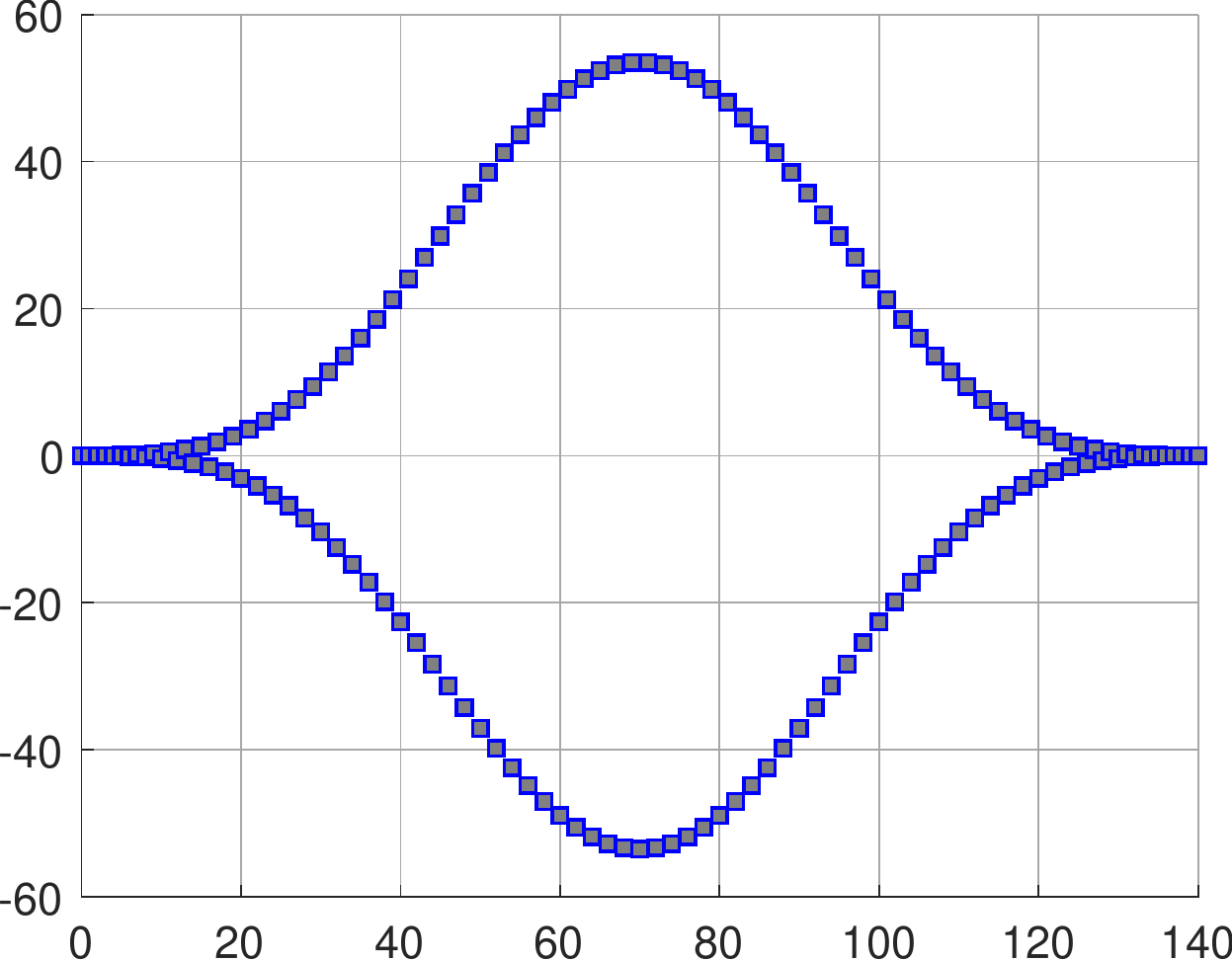}
		\caption*{$\alpha = 0.96$}
	\end{subfigure}\hfil 
	\begin{subfigure}{0.25\textwidth}
		\includegraphics[width=\linewidth]{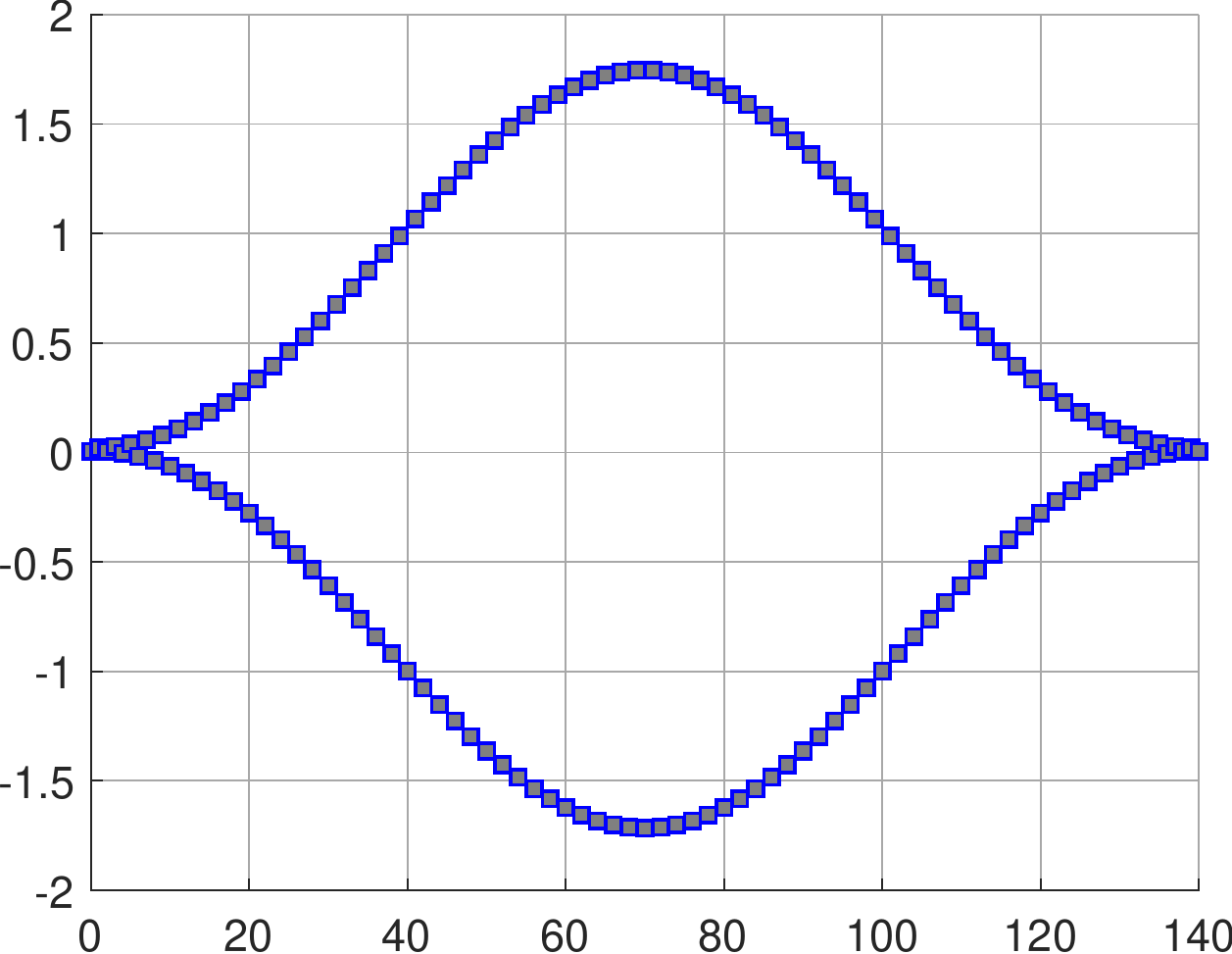}
		\caption*{$\alpha = 0.97$}
	\end{subfigure}\hfil 
	\begin{subfigure}{0.25\textwidth}
		\includegraphics[width=\linewidth]{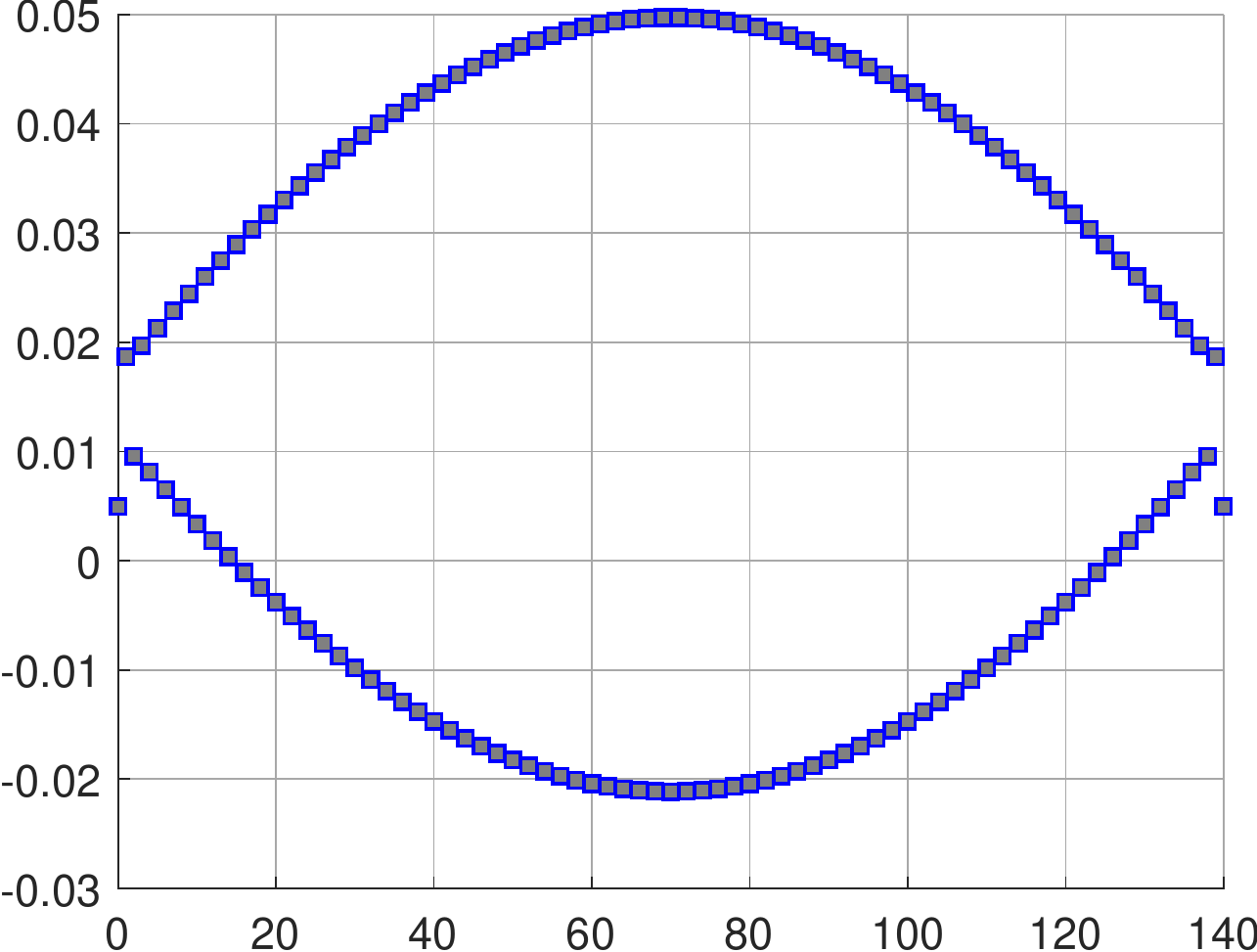}
		\caption*{$\alpha = 0.98$}
	\end{subfigure}\hfil 
	\begin{subfigure}{0.25\textwidth}
		\includegraphics[width=\linewidth]{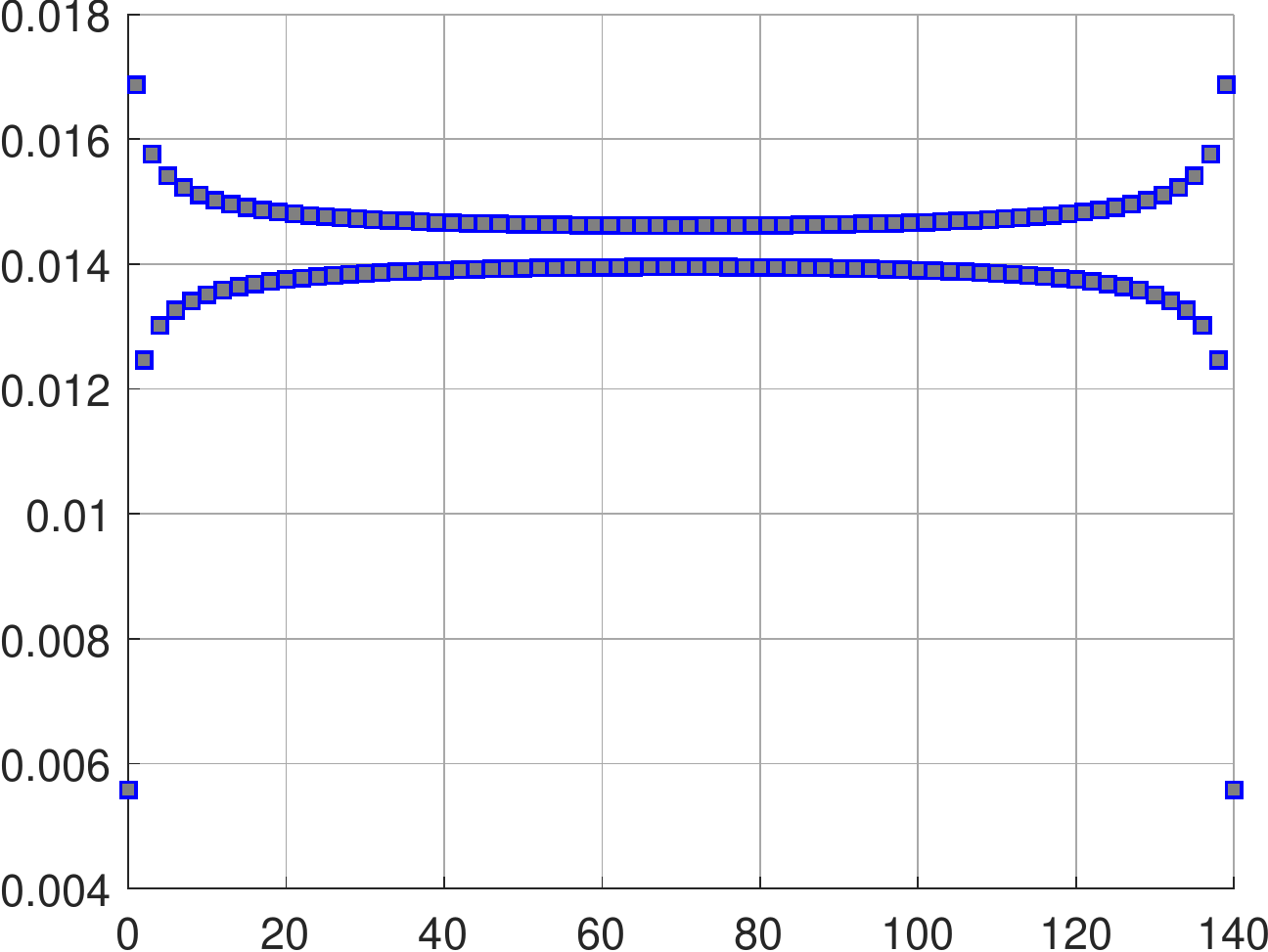}
		\caption*{$\alpha = 0.99$}
	\end{subfigure}\hfil 
\caption{\wee{KTI Quadrature weights $ \{ w_{i}^{\alpha} : i=0, ... , 140 \}$ for 141 equispaced nodes in the interval $[-1,1]$ and the mapping parameter $\alpha$ close to one.}}
\label{fig_pesi_137}
\end{figure}

\begin{figure}[h]
    \begin{subfigure}{0.25\textwidth}
		\includegraphics[width=\linewidth]{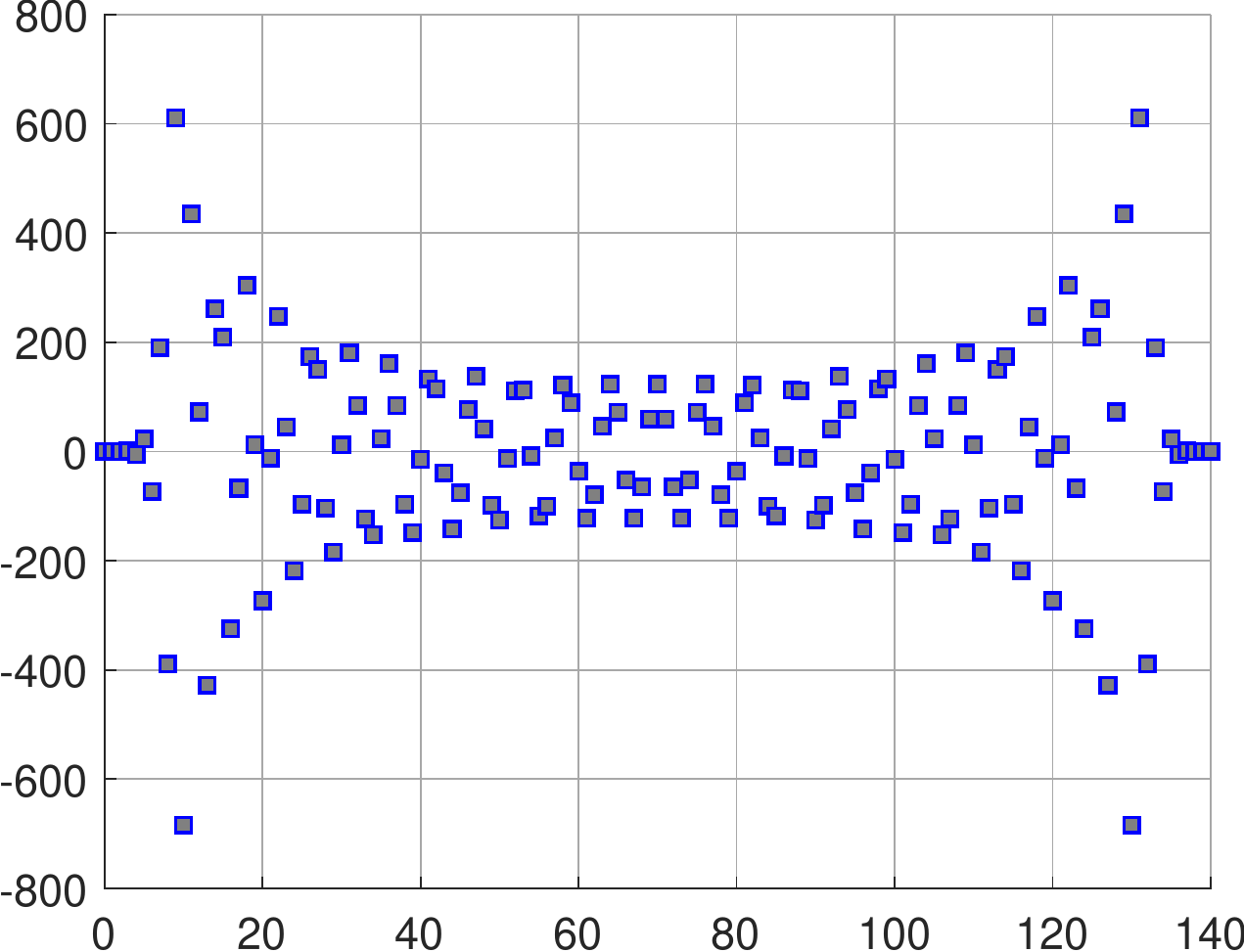}
		\caption*{$\alpha = 0.20$}
	\end{subfigure}\hfil 
	\begin{subfigure}{0.25\textwidth}
		\includegraphics[width=\linewidth]{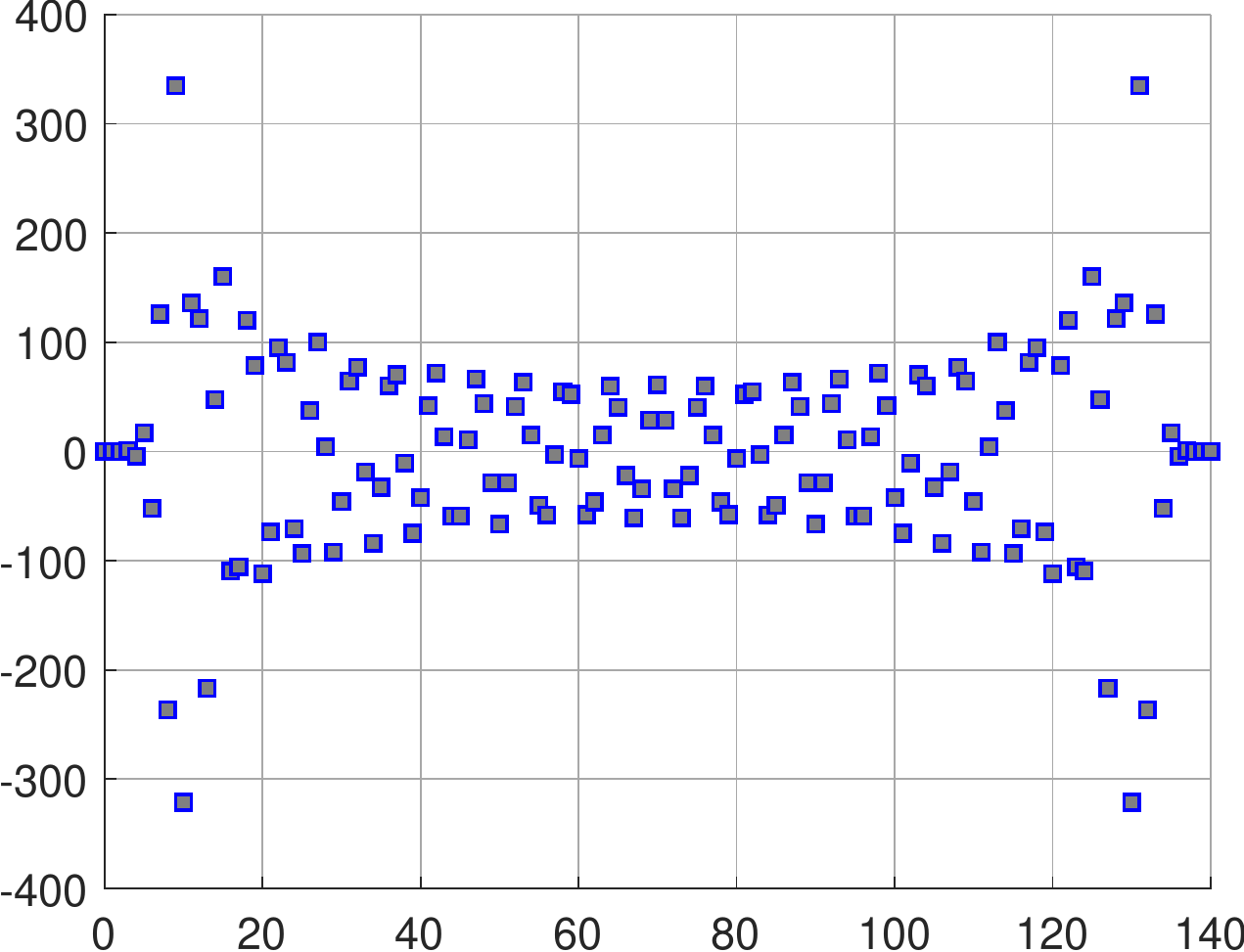}
		\caption*{$\alpha = 0.30$}
	\end{subfigure}\hfil 
	\begin{subfigure}{0.25\textwidth}
		\includegraphics[width=\linewidth]{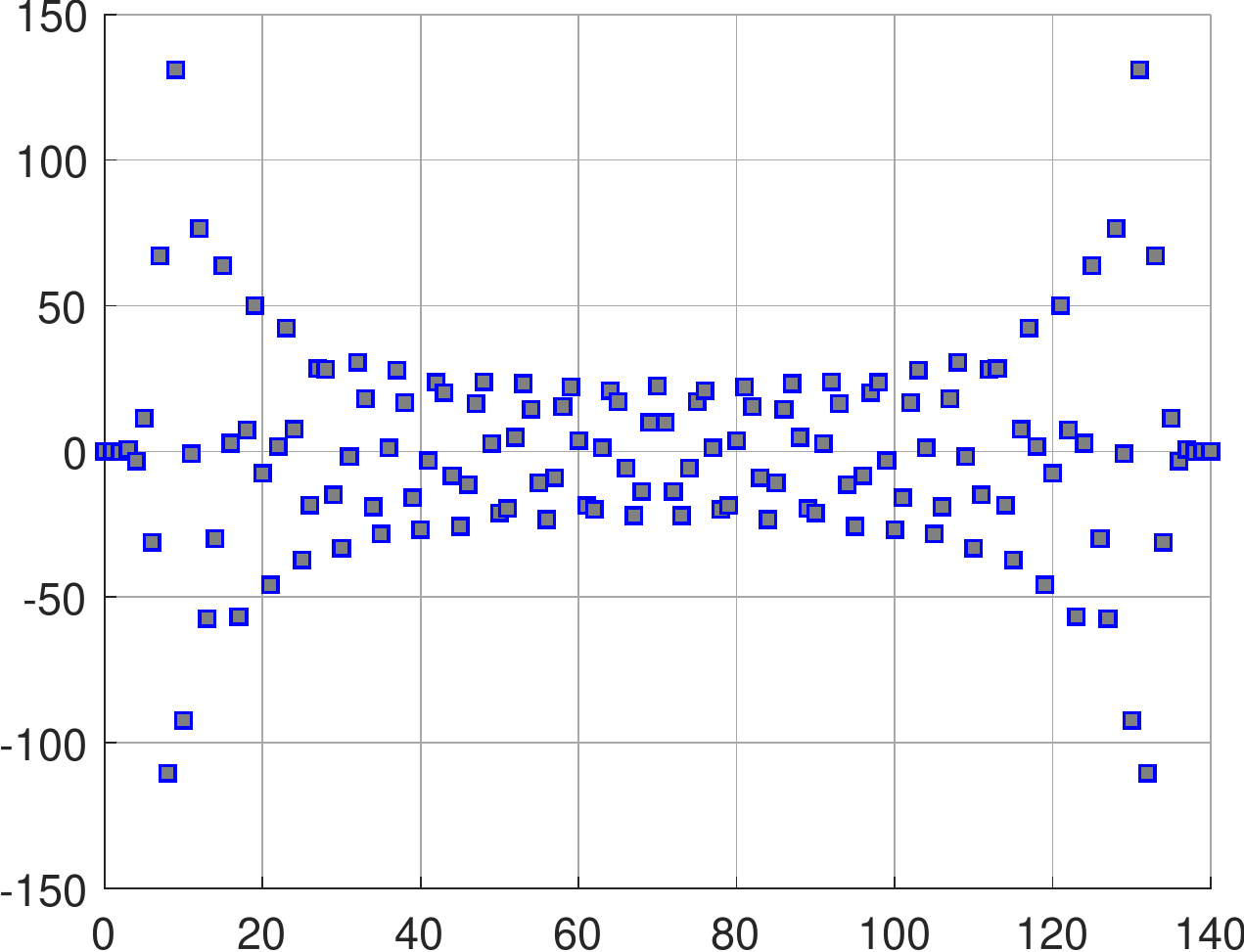}
		\caption*{$\alpha = 0.40$}
	\end{subfigure}\hfil 
	\begin{subfigure}{0.25\textwidth}
		\includegraphics[width=\linewidth]{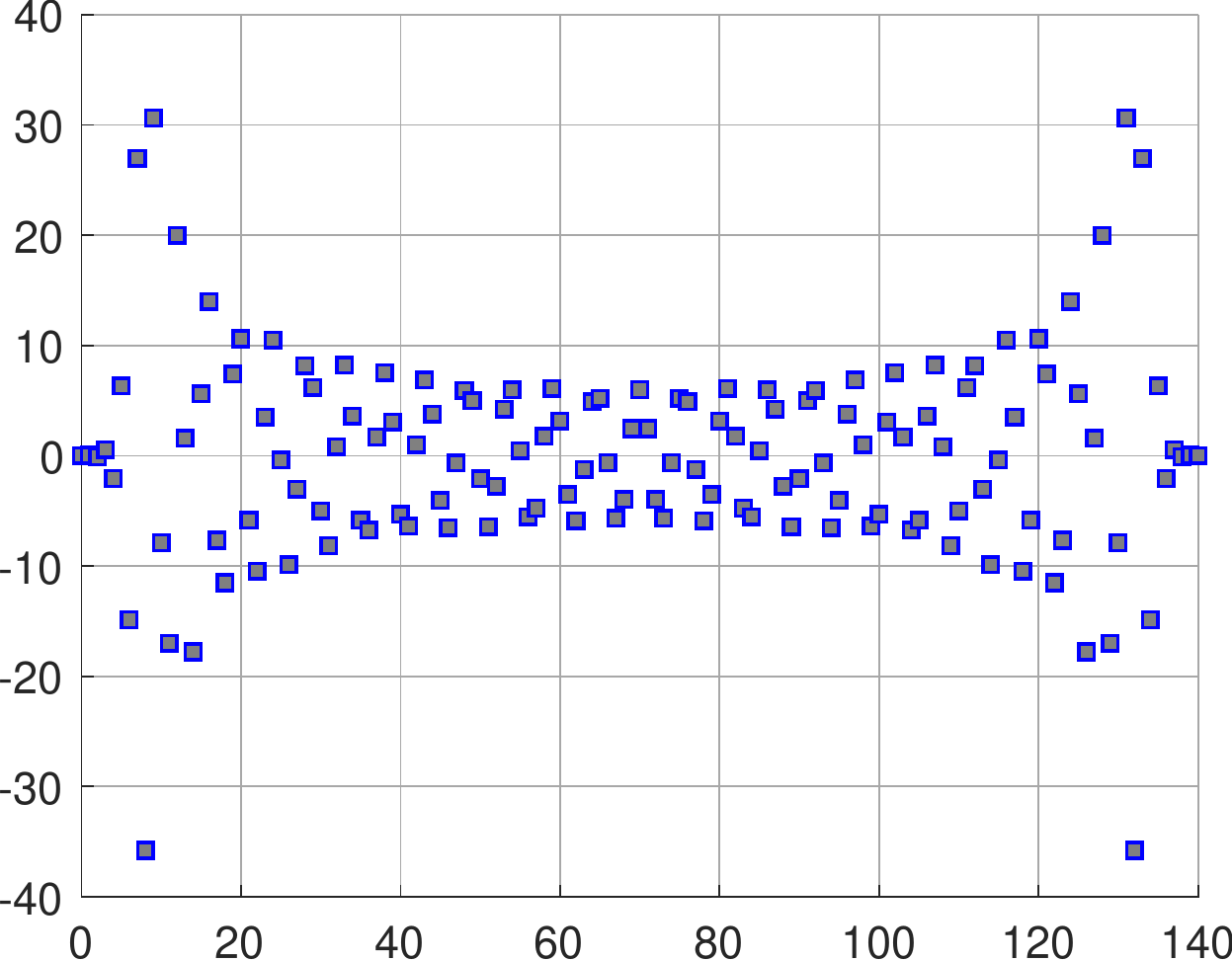}
		\caption*{$\alpha = 0.50$}
	\end{subfigure}\hfil 
    
	\begin{subfigure}{0.25\textwidth}
		\includegraphics[width=\linewidth]{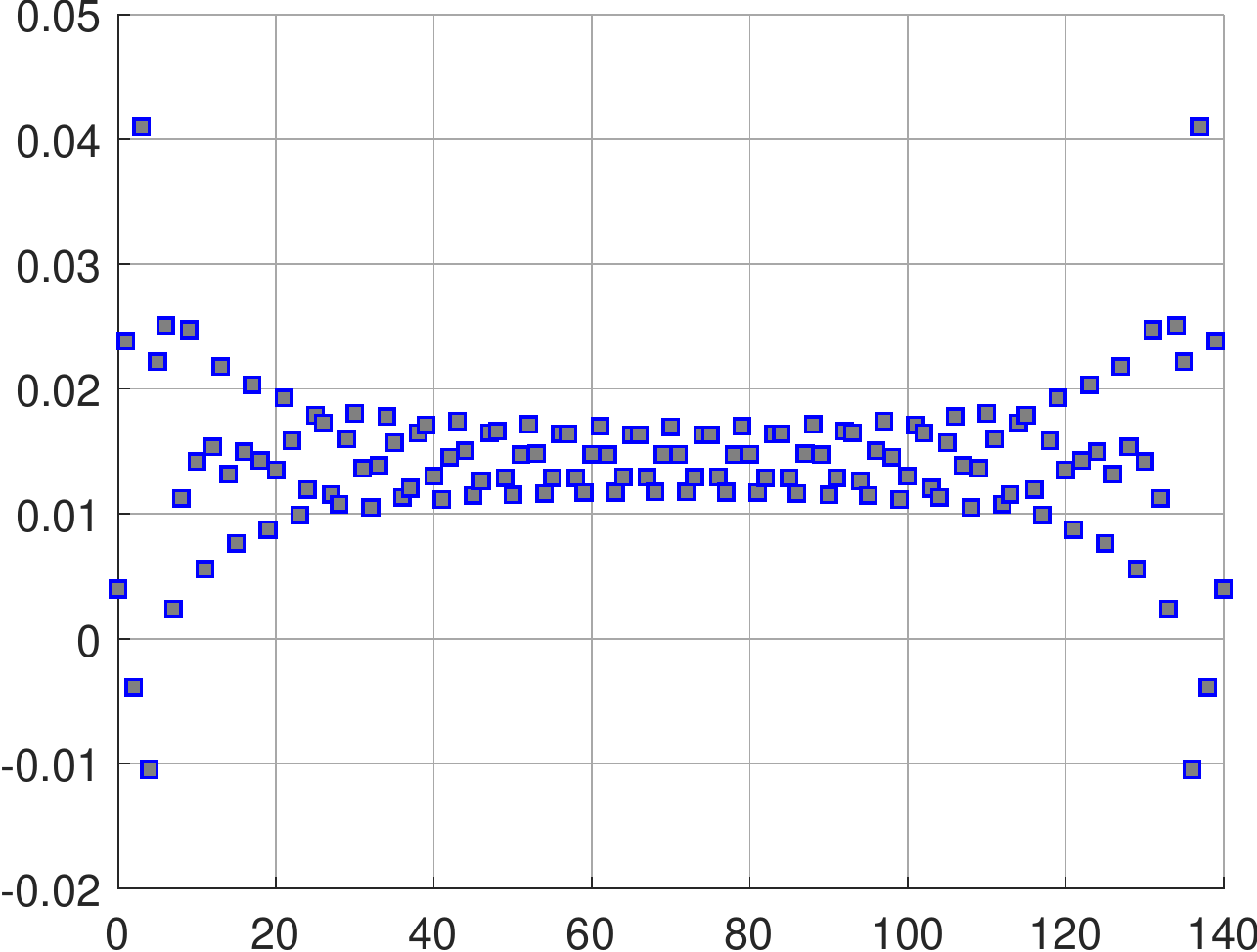}
		\caption*{$\alpha = 0.85$}
	\end{subfigure}\hfil 
	\begin{subfigure}{0.25\textwidth}
		\includegraphics[width=\linewidth]{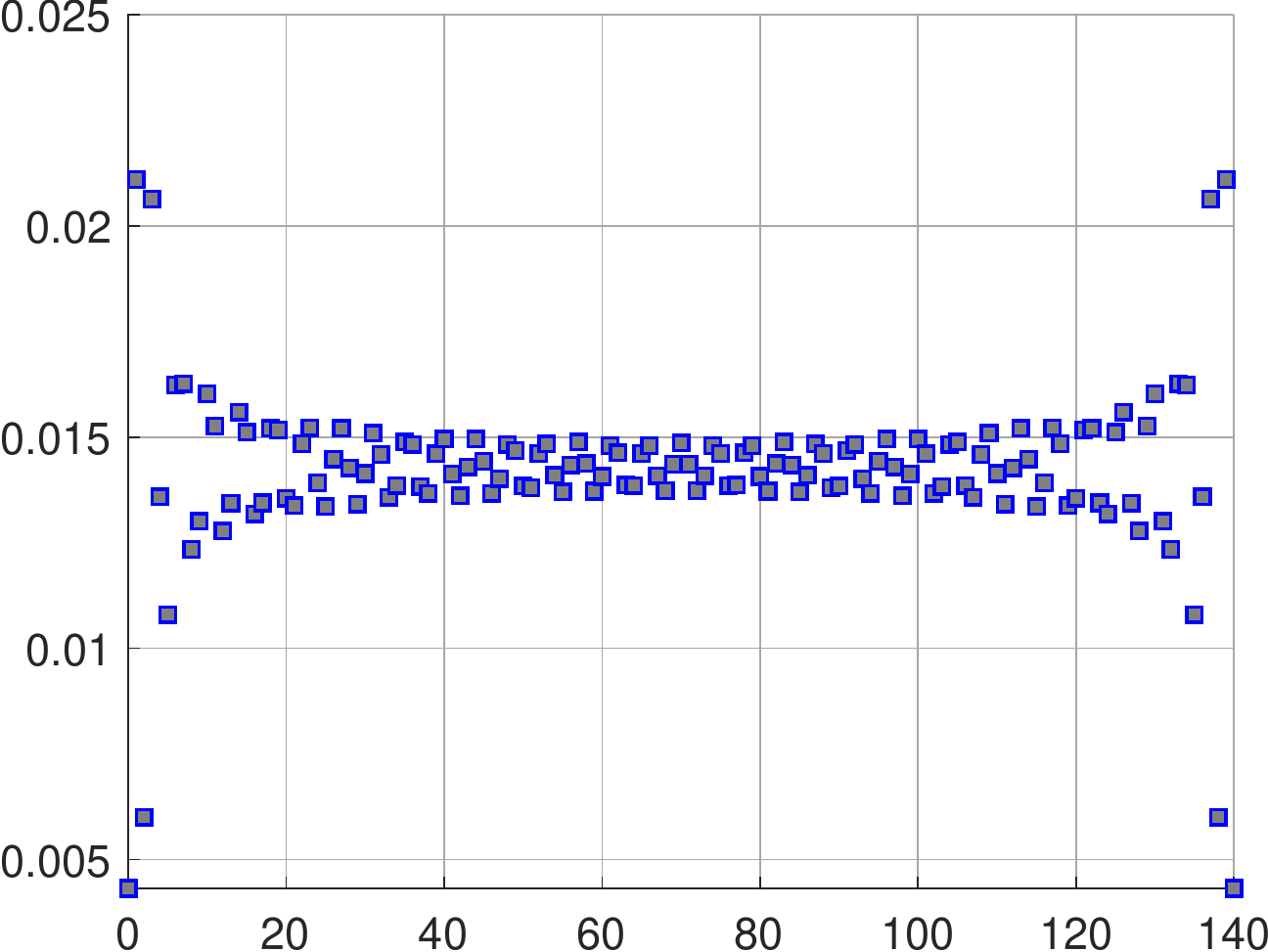}
		\caption*{$\alpha = 0.90$}
	\end{subfigure}\hfil 
	\begin{subfigure}{0.25\textwidth}
		\includegraphics[width=\linewidth]{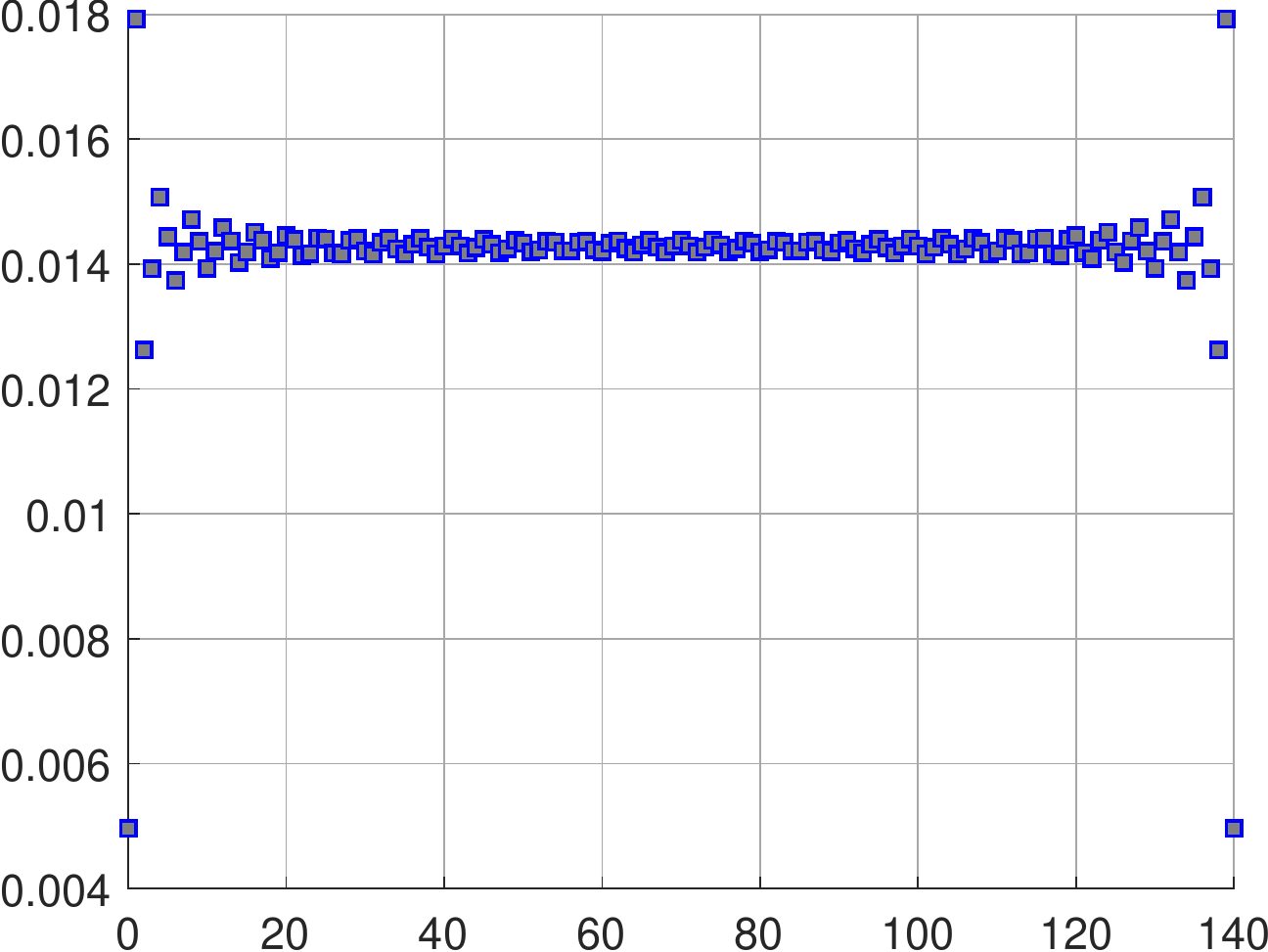}
		\caption*{$\alpha = 0.95$}
	\end{subfigure}\hfil 
	\begin{subfigure}{0.25\textwidth}
		\includegraphics[width=\linewidth]{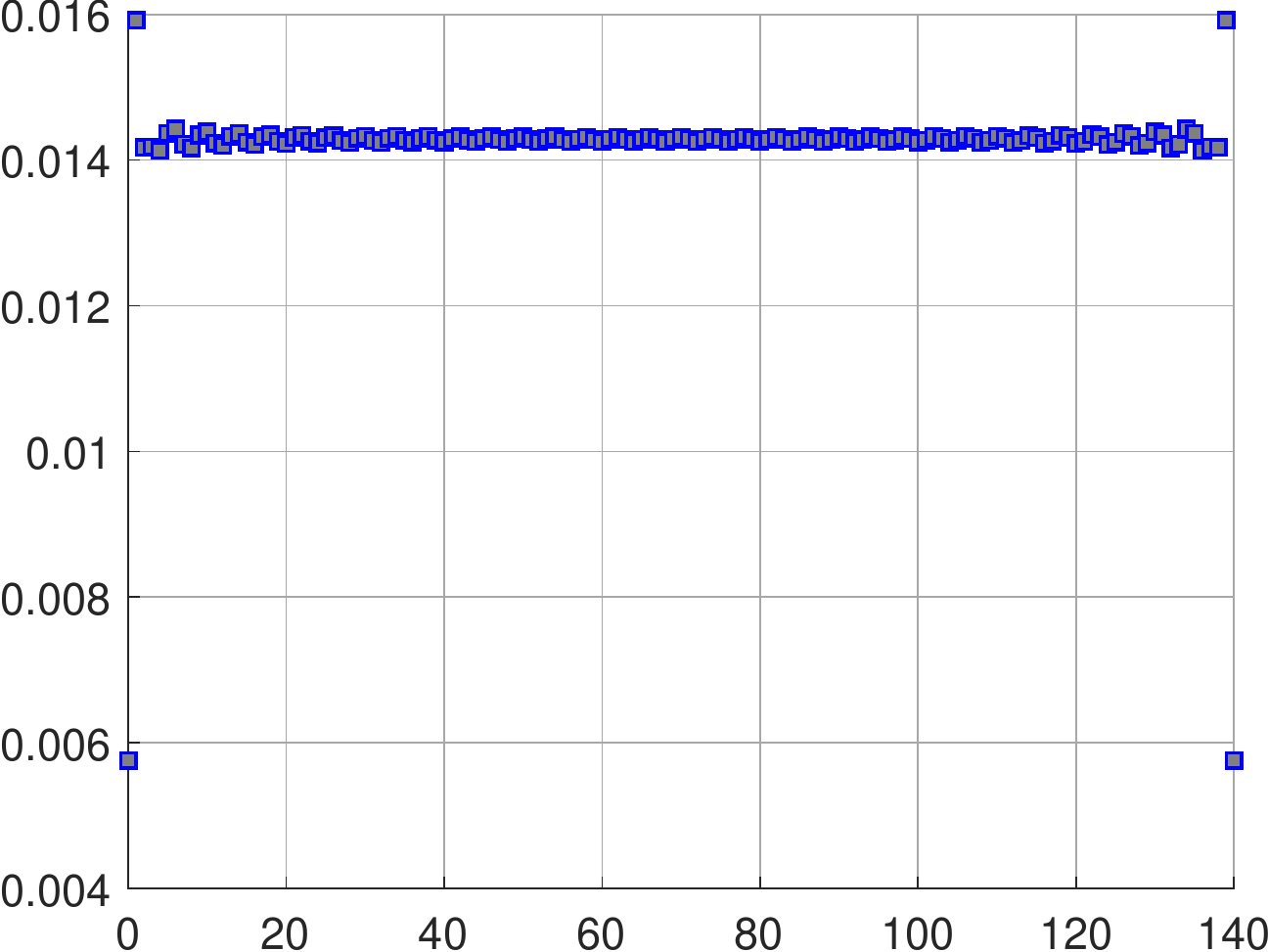}
		\caption*{$\alpha = 0.98$}
	\end{subfigure}\hfil 
\caption{\wee{Dependence of the KTL Quadrature weights $ \{ w_{i}^{\alpha} : i=0, ... , 140 \}$ on the parameter $\alpha$ for 141 equispaced nodes in the interval $[-1,1]$. Using $\textrm{dim}(\mathbb{P}_{70}^{\alpha})=71$, we have the ratio $n/m = 0.5$. }}
\label{fig_pesi_140old_KTL}
\end{figure}

\begin{figure}[h]
    \begin{subfigure}{0.25\textwidth}
		\includegraphics[width=\linewidth]{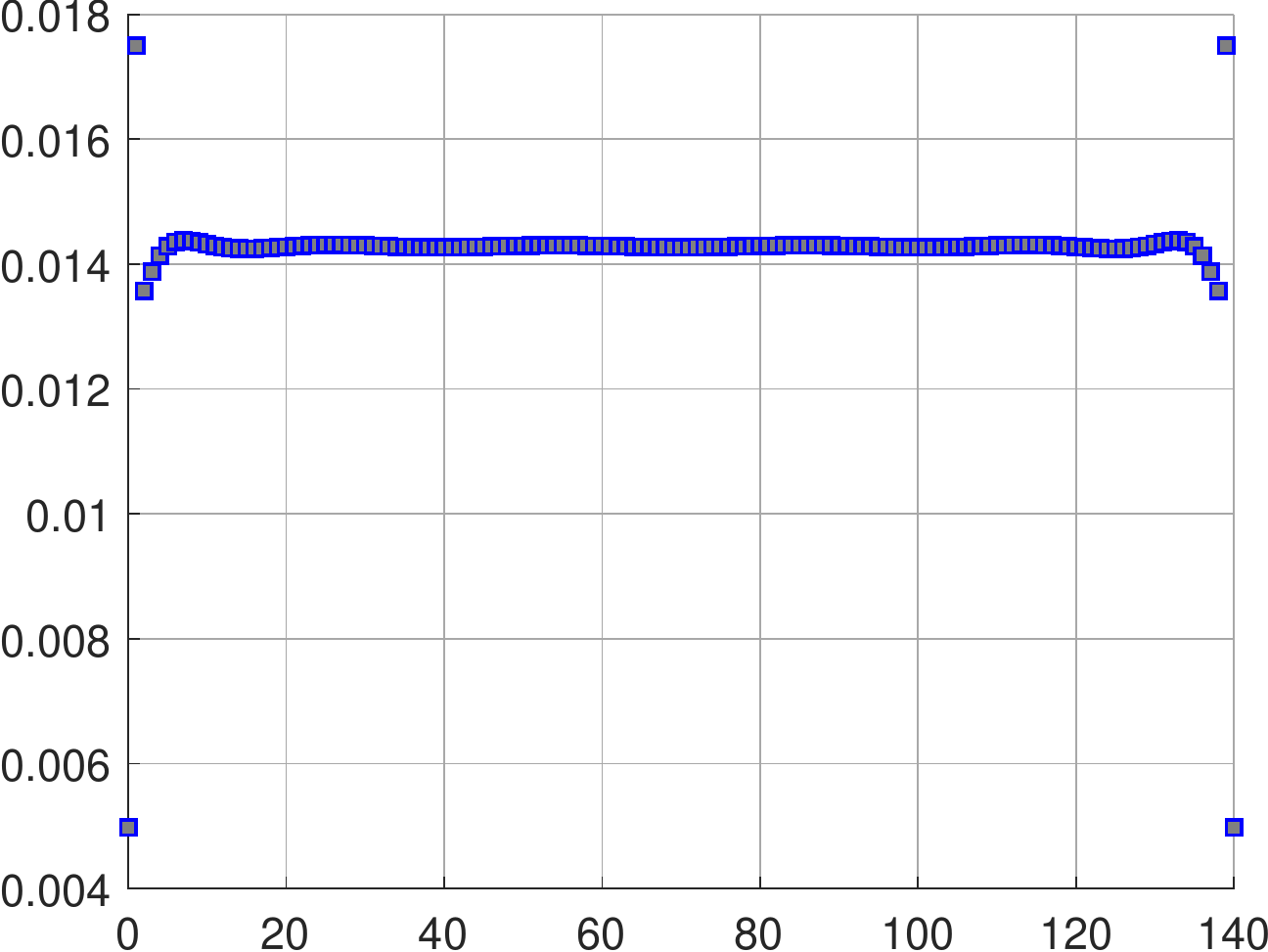}
		\caption*{$\alpha = 0.20$}
	\end{subfigure}\hfil 
	\begin{subfigure}{0.25\textwidth}
		\includegraphics[width=\linewidth]{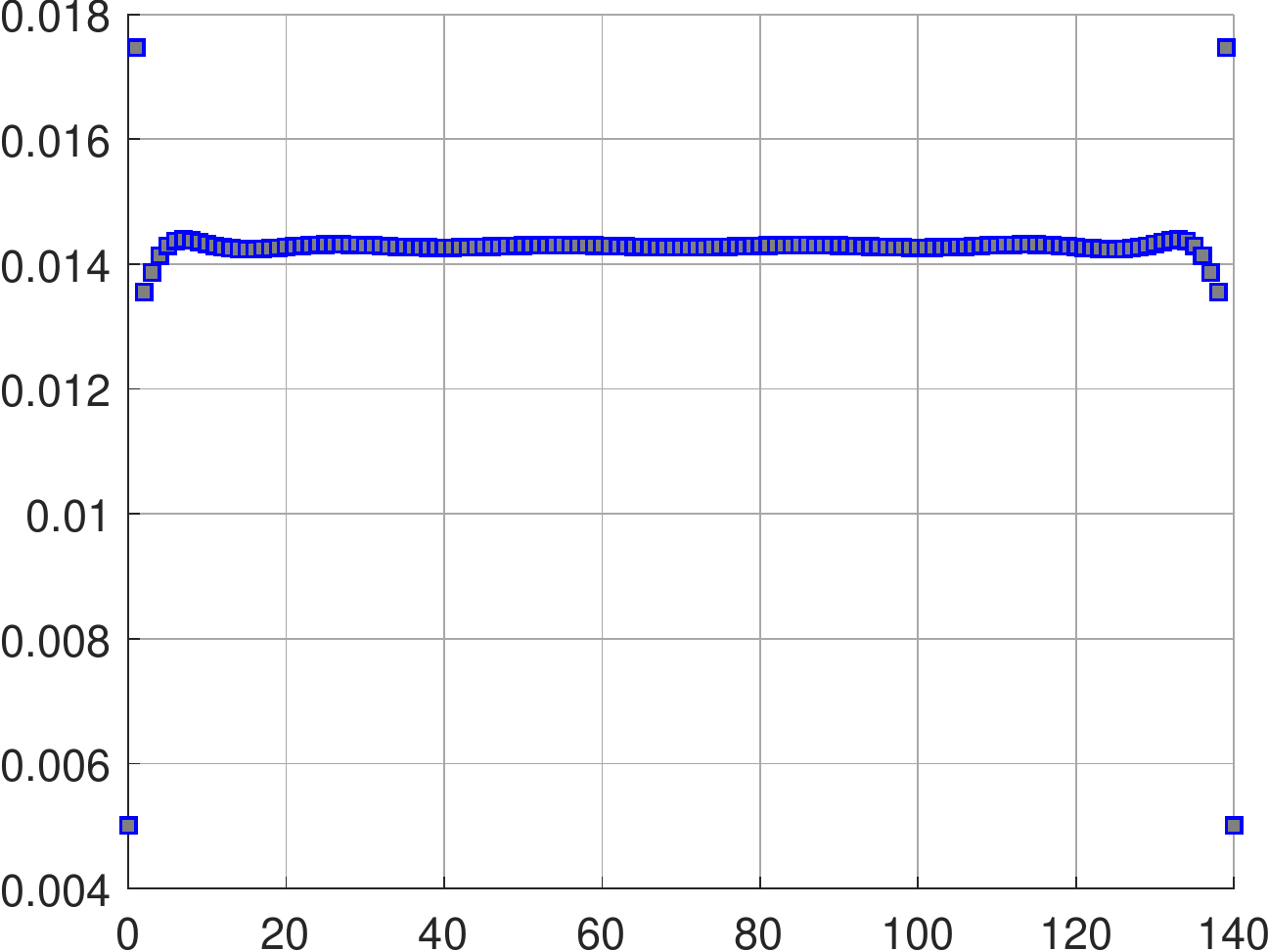}
		\caption*{$\alpha = 0.30$}
	\end{subfigure}\hfil 
	\begin{subfigure}{0.25\textwidth}
		\includegraphics[width=\linewidth]{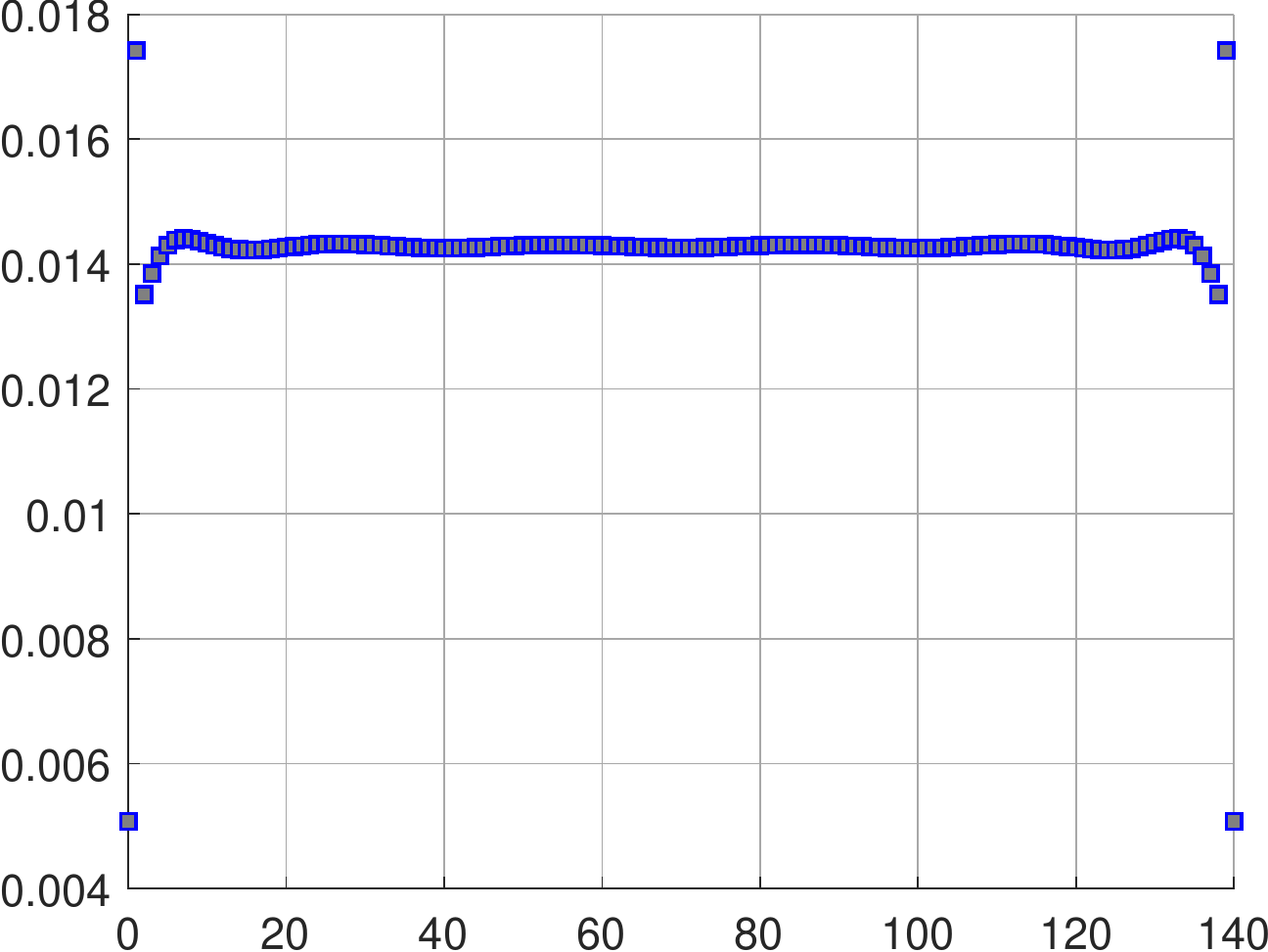}
		\caption*{$\alpha = 0.40$}
	\end{subfigure}\hfil 
	\begin{subfigure}{0.25\textwidth}
		\includegraphics[width=\linewidth]{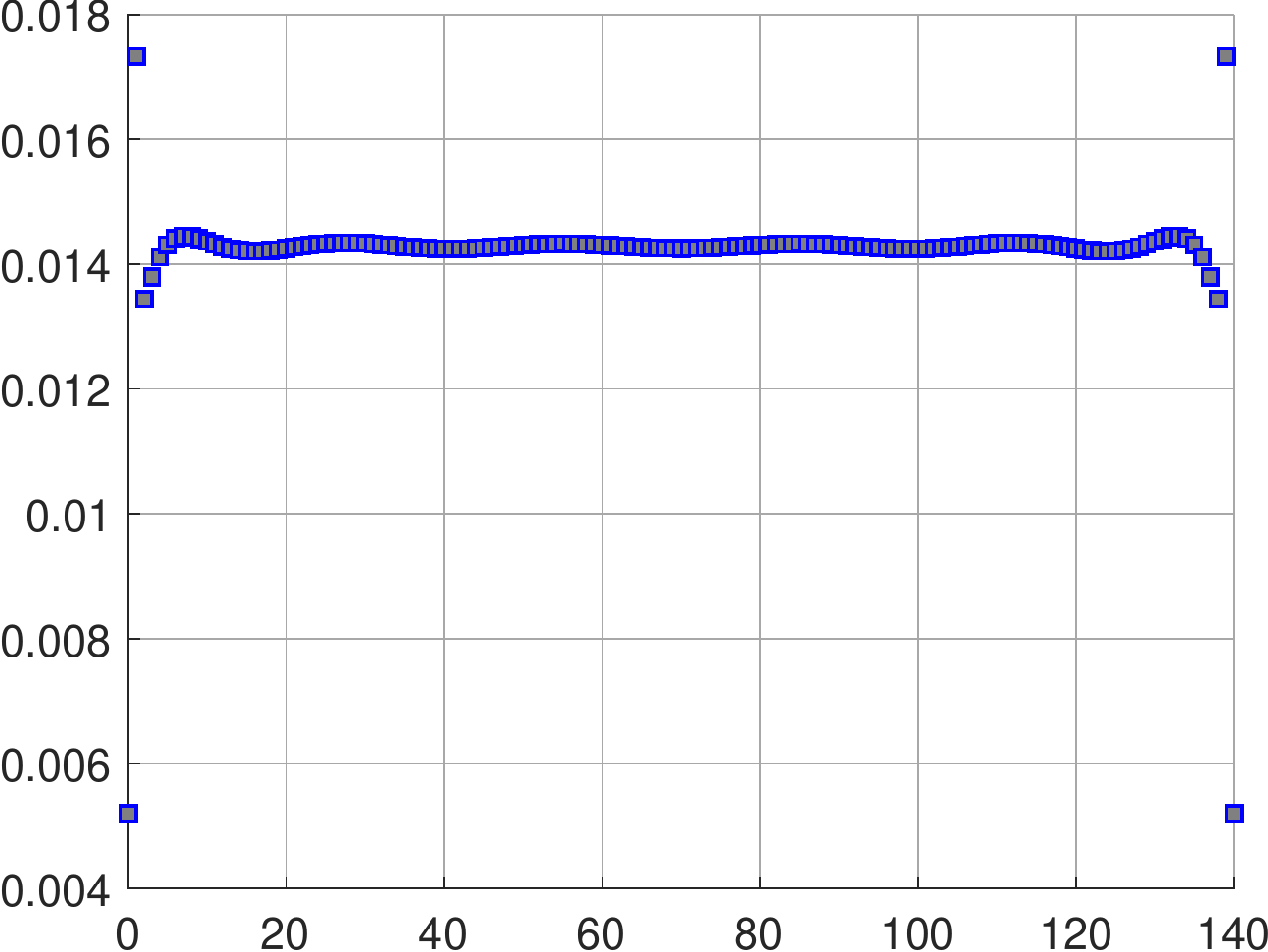}
		\caption*{$\alpha = 0.50$}
	\end{subfigure}\hfil 
    
	\begin{subfigure}{0.25\textwidth}
		\includegraphics[width=\linewidth]{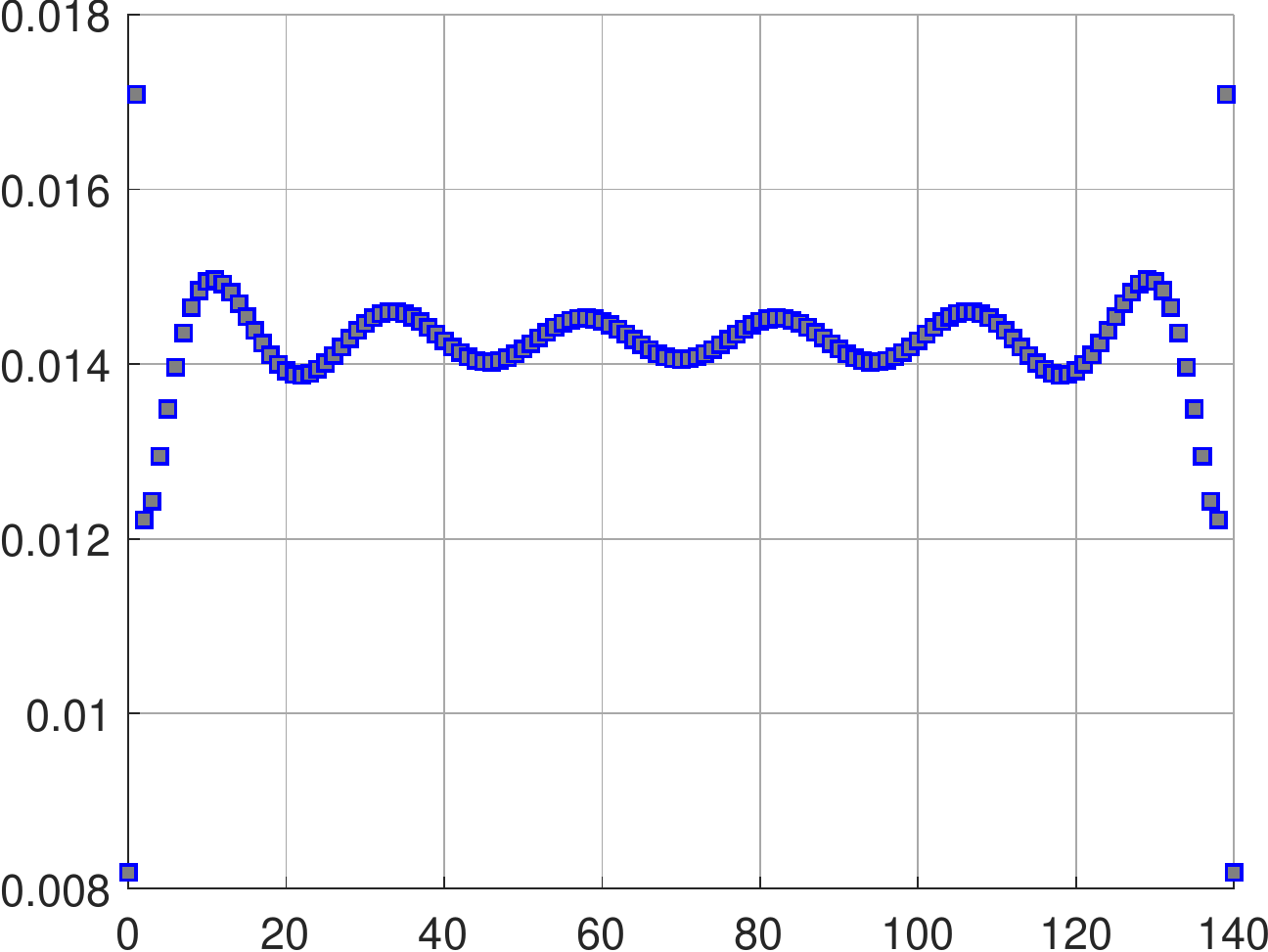}
		\caption*{$\alpha = 0.85$}
	\end{subfigure}\hfil 
	\begin{subfigure}{0.25\textwidth}
		\includegraphics[width=\linewidth]{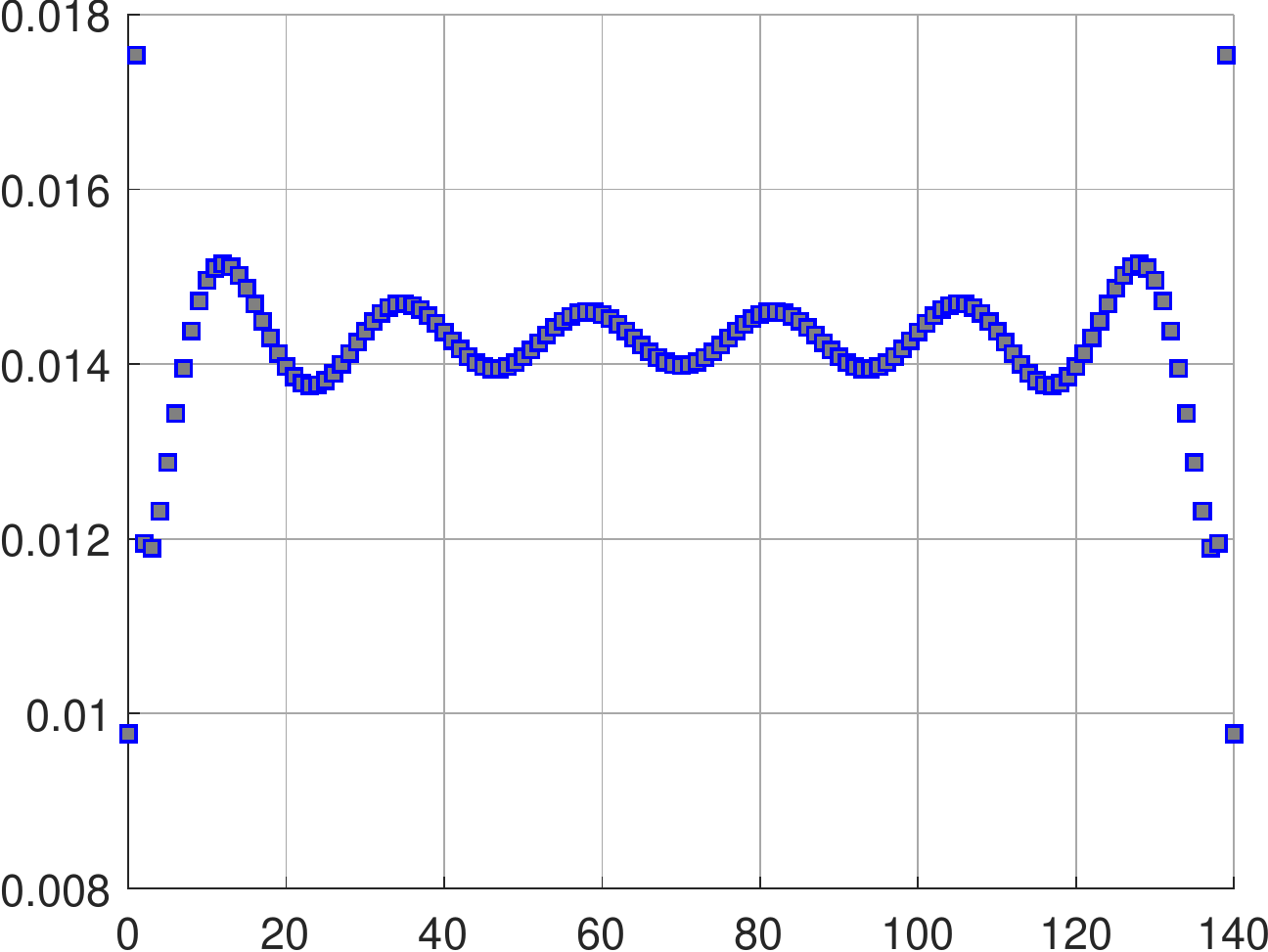}
		\caption*{$\alpha = 0.90$}
	\end{subfigure}\hfil 
	\begin{subfigure}{0.25\textwidth}
		\includegraphics[width=\linewidth]{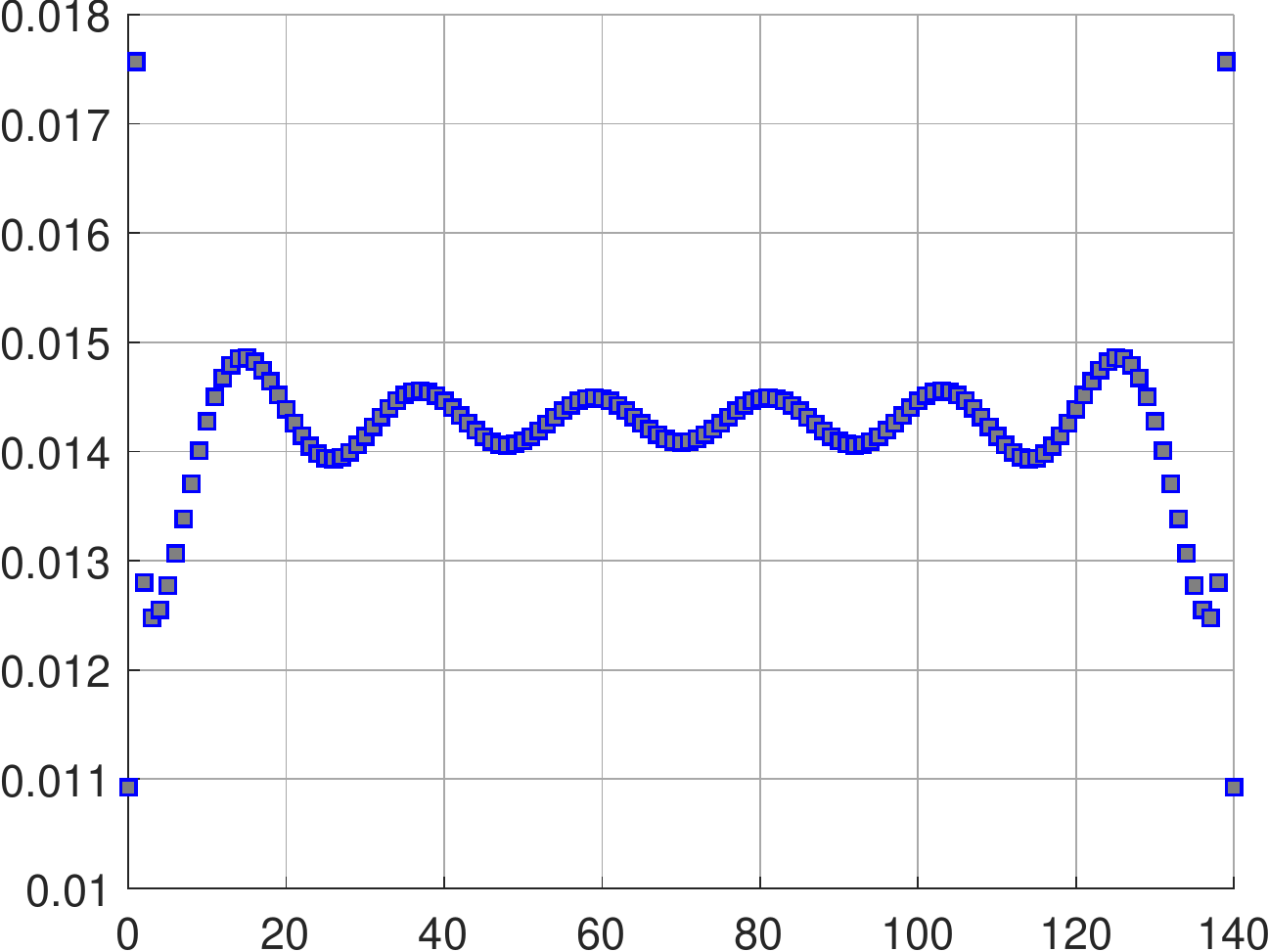}
		\caption*{$\alpha = 0.98$}
	\end{subfigure}\hfil 
	\begin{subfigure}{0.25\textwidth}
		\includegraphics[width=\linewidth]{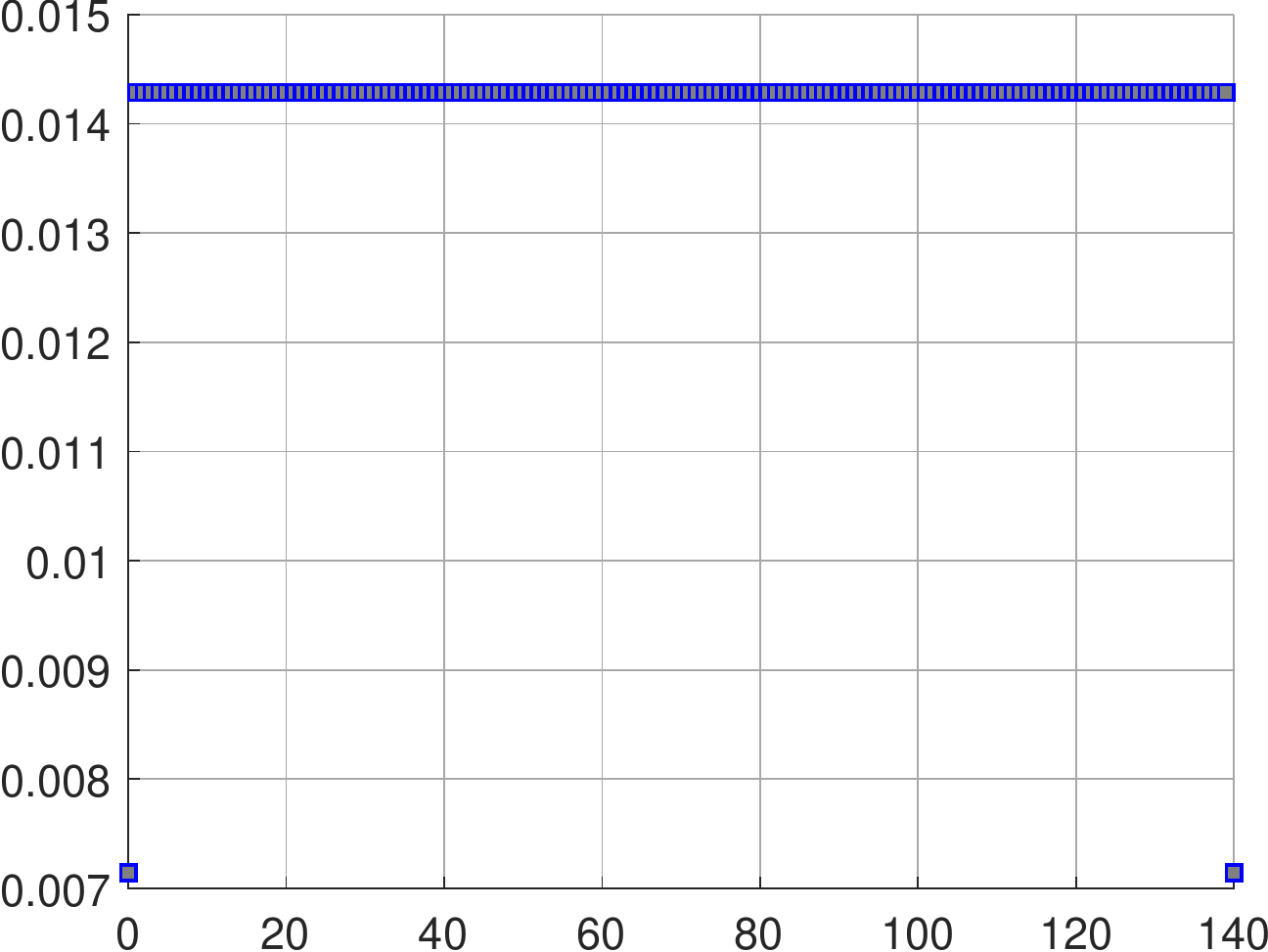}
		\caption*{$\alpha = 1$}
	\end{subfigure}\hfil 
\caption{\wee{Dependence of the KTL Quadrature weights $ \{ w_{i}^{\alpha} : i=0, ... , 140 \}$ on the parameter $\alpha$ for 141 equispaced nodes in the interval $[-1,1]$. Using $\textrm{dim}(\mathbb{P}_{12}^{\alpha})=13$, we have the ratio $m/n^2 \approx 0.97$. }}
\label{fig_pesi_140_KTL}
\end{figure}

\wee{If we compare Fig. \ref{fig_pesi_137}, Fig. \ref{fig_pesi_140old_KTL} and Fig. \ref{fig_pesi_140_KTL}, we see that the stability of the KTL quadrature in terms of the parameter $\alpha$ depends also strongly on the relation between the dimension $n+1$ of the polynomial approximation space and the number $m+1$ of grid points. While for the KTI quadrature rule with $n = m$ a choice of $\alpha$ close to $1$ is necessary to obtain positive quadrature weights, this limitation can be relaxed or dropped for the KTL scheme. In Fig. \ref{fig_pesi_140old_KTL} we observe that for a least-squares formula with ratio $n/m = 0.5$ we get positive quadrature weights (and thus stability) already for $\alpha \approx 0.9$ while smaller values still lead to negative weights. On the other hand Fig. \ref{fig_pesi_140_KTL} shows that for a ratio $m \approx n^2$ the parameter $\alpha$ has only a minor impact on the KTL weights and that in this case the quadrature rule is stable independently of the chosen $\alpha$.
In section \ref{sec:numericalexperiments}, we will further see that a fixed choice of $\alpha$ is less valuable than a dynamic strategy in which the $\alpha$ depends on the degree $n$.}

\subsection{Parameter dependent convergence of KTL quadrature}
\label{sec:parametersKTL}

Convergence properties of KTL quadrature formulas \we{can} be derived from the approximation behavior of the least-squares approximant $F_{n,\mathcal{X}}^{\alpha}(f)$. This behavior depends on the interplay of the space dimension $n$, the KT parameter $\alpha$ and the distribution of the quadrature nodes $\mathcal{X}$. As a main control parameter for the behavior of the node set $\mathcal{X}$ in the quadrature scheme we will consider the maximal distance 
\[ h = \max_{i = 0, \ldots, m+1} |x_i - x_{i-1}|\]
between two nodes in $\mathcal{X}$, where $x_{-1} = -1$ and $x_{m+1} =1$ denote the boundaries of the interval $I$. For equidistant nodes $x_i = -1 + 2i/m $, $i = 0, \ldots, m$, the distance $h$ corresponds to the spacing $h = 2/m$.   
A standard argument (see \cite[Theorem 3.2]{platte}) shows that the uniform approximation error 
$\|f - F_{n,\mathcal{X}}^{\alpha}(f)\|_{\infty}$ can be bounded by
\[\|f - F_{n,\mathcal{X}}^{\alpha}(f)\|_{\infty} \leq (1 + \mathcal{K}_{n,\mathcal{X}}^{\alpha}) E_n^{\alpha}(f),\]
where $\mathcal{K}_{n,\mathcal{X}}^{\alpha} = \sup_{\|f\|_\infty = 1} \| F_{n,\mathcal{X}}^{\alpha}(f)\|_{\infty}$ is the operator norm of the approximation operator $F_{n,\mathcal{X}}^{\alpha}$ (\wee{which we will refer to as Lebesgue constant}) and $E_n^{\alpha}(f) = \inf_{p \in \mathbb{P}_n^{\alpha}} \|f - p\|_{\infty}$ is the best approximation error in the space $\mathbb{P}_n^{\alpha}$. For the KTL quadrature formula this bound \we{immediately implies} the estimate
\[ |\mathcal{I}_{n,\mathcal{X}}^{\alpha}(f,I) - \mathcal{I}(f,I)| \leq \int_{-1}^1 |f(x) - F_{n,\mathcal{X}}^{\alpha}(f)(x)| \mathrm{d} x \leq 2 (1 + \mathcal{K}_{n,\mathcal{X}}^{\alpha}) E_n^{\alpha}(f).\]
In particular, this means that any estimates of the best approximation error $E_n^{\alpha}(f)$ and the Lebesgue constant $\mathcal{K}_{n,\mathcal{X}}^{\alpha}$ can be used directly also for the quadrature formulas studied in this work. We give a brief summary of the major statements derived in \cite{platte} and its implications for the convergence of the KTL scheme. \\

\textbf{The case $\boldsymbol{\alpha}=\boldsymbol{0}$.} 
The space $\mathbb{P}_n^0$ corresponds to the space $\mathbb{P}_n$ of algebraic polynomials of degree $n$. The term $E_n^{0}(f)$ therefore corresponds to the best approximation error in $\mathbb{P}_n$, which implies geometric convergence of $E_n(f)$ as $n \to \infty$ if the function $f$ is analytic in a Bernstein ellipse containing the interval $I$ \cite[Theorem 8.3]{trefappr}.

A sufficient condition for the parameters $n$ and $h$ to guarantee the boundedness of the Lebesgue constant $\mathcal{K}_{n,\mathcal{X}}^{0}$ is $n=\mathcal{O}(1/\sqrt{h})$. If the quadrature nodes $\mathcal{X}$ are equidistant with spacing $h = \frac{2}{m}$, this implies that $n = \mathcal{O}(\sqrt{m})$. This is a quite strong restriction on the choice of the polynomial approximation degree $n$. \we{It implies, however,} the root-exponential convergence rate
\[ |\mathcal{I}_{n,\mathcal{X}}^{\alpha}(f,I) - \mathcal{I}(f,I)| = \mathcal{O}(\rho^{-\sqrt{m}}),\]
in which $\rho>1$ denotes the index of the Bernstein ellipse. 

\textbf{The case $\boldsymbol{\alpha}=\boldsymbol{1}$.}
If the two parameters $n$ and $h$ are related as $n \leq c \frac{1}{h}$ with a proper constant $c > 0$, the Lebesgue constant $\mathcal{K}_{n,\mathcal{X}}^{1}$ is bounded. For a uniformly distributed set $\mathcal{X}$ this means that the degree $n$ can be chosen as a linear function $n = c m$ of $m$. On the other hand, the best approximation error $E_n^{1}(f)$ might decay quite slowly for $\alpha =1$, i.e., in the order of $\mathcal{O}(1/n)$ also for smooth functions $f$ (if they don't satisfy periodic boundary conditions). This will be visible also in our numerical tests and is a drawback for the choice $\alpha = 1$. 

\textbf{The case $\boldsymbol{0}<\boldsymbol{\alpha}<\boldsymbol{1}$.} For a fixed parameter $0 < \alpha < 1$, the convergence behavior of the formulas is principally the same as for $\alpha = 0$: the Lebesgue constant is bounded if $n$ and $h$ satisfy a relation of the form $n=\mathcal{O}(1/\sqrt{h})$ as $h \longrightarrow 0$. Furthermore, also the best approximation error $E_{n}^{\alpha}(f)$ decays geometrically if the function is analytic in a neighborhood of $I$, implying that $\mathcal{K}_{n,\mathcal{X}}^{\alpha} E_{n}^{\alpha}(f) = \mathcal{O}(\rho^{-\sqrt{m}})$ for equidistant nodes $\mathcal{X}$ and a proper $\rho > 1$. In \wee{\cite{platte2011}}, it was shown that the root-exponential rate $\mathcal{O}(\rho^{-\sqrt{m}})$, $\rho > 1$, \we{is best} possible for a stable algorithm approximating an analytic function on equidistant nodes.

\textbf{The case $\boldsymbol{\alpha_n}=\boldsymbol{4/\pi \arctan(\epsilon^{1/n})}$.}
An asymptotic analysis given in \cite{platte} shows that the choice $n=\mathcal{O}(1/h)$ for $h \longrightarrow 0$ is sufficient for the boundedness of the Lebesgue constant (similarly as for $\alpha = 1$). On the other hand, compared to $\alpha = 1$ a smaller approximation error $E_n^{\alpha_n}(f)$ can be expected till a small error tolerance $\epsilon > 0$ is reached (this is also visible numerically). However, geometric convergence towards $0$ as in the case $\alpha < 1$ can no longer be guaranteed.

\section{Numerical experiments} \label{sec:numericalexperiments}

In this section, we provide some numerical experiments that investigate the convergence and stability properties of the KTL and KTI quadrature formulas in more detail. 
All the code of the following experiments is publicly available at the GitHub page of this work
\begin{center}
    \url{https://github.com/GiacomoCappellazzo/KTL_quadrature} .\\
\end{center}

\noindent We consider the three test functions 

$$f_1 (x) = \frac{1}{1+100x^2}, \quad f_2 (x) = \frac{1}{1+16\sin^2(7x)}, \quad\text{and} \quad
f_3 (x) = \sqrt{1.01 + x}$$
and their respective integrals $\mathcal{I}(f_k,I)$, $k \in \{1,2,3\}$, over the interval $I = [-1,1]$.

In the following experiments we plot the relative errors obtained by comparing the KTL quadrature formula with the exact value of the integral:
$$
\mathcal{E}_{\mathrm{rel}}= \Bigg| \frac{\mathcal{I}_{n,\mathcal{X}}^{\alpha}(f,I) - \mathcal{I}(f,I)}{\mathcal{I}(f,I)} \Bigg|. 
$$
For the calculation of the exact integral $\mathcal{I}(f,I)$ with a high precision, we used the {\tt Matlab} command \texttt{integral(f,-1,1)}. All three test functions are analytic in an open neighborhood of $[- 1,1]$. The function $f_{1}$ has poles close to the interval $[-1,1] $ in the complex plane, $f_{2}$ is an highly oscillatory \we{entire} function and $f_{3} $ has a singularity close to $x=-1$.

\subsection{KTI formulas for a fixed parameter $\alpha$}
In the following graphs we show the results obtained by Algorithm \ref{alg:2} implementing the interpolatory KTI quadrature formula through mapped nodes with the Kosloff Tal-Ezer map and equidistant nodes $\mathcal{X}$. In particular, we have $m = n$ and use a fixed mapping parameter $0 \leq \alpha \leq 1$. From the discussion in Section \ref{sub_sec_limit_alpha} we know that if the mapping parameter $\alpha=1$ the quadrature weights correspond to those of the composed trapezoidal rule. In Fig. \ref{fig:4}, the blue, magenta and black curves display the relative quadrature errors using the KTI quadrature rule with $m+1$ equidistant nodes and with parameters $\alpha=1$, $\alpha=0.99$, and $\alpha=0.98$, respectively. 

\begin{figure}[H] 
	\begin{subfigure}{0.33\textwidth}
		\includegraphics[width=\linewidth]{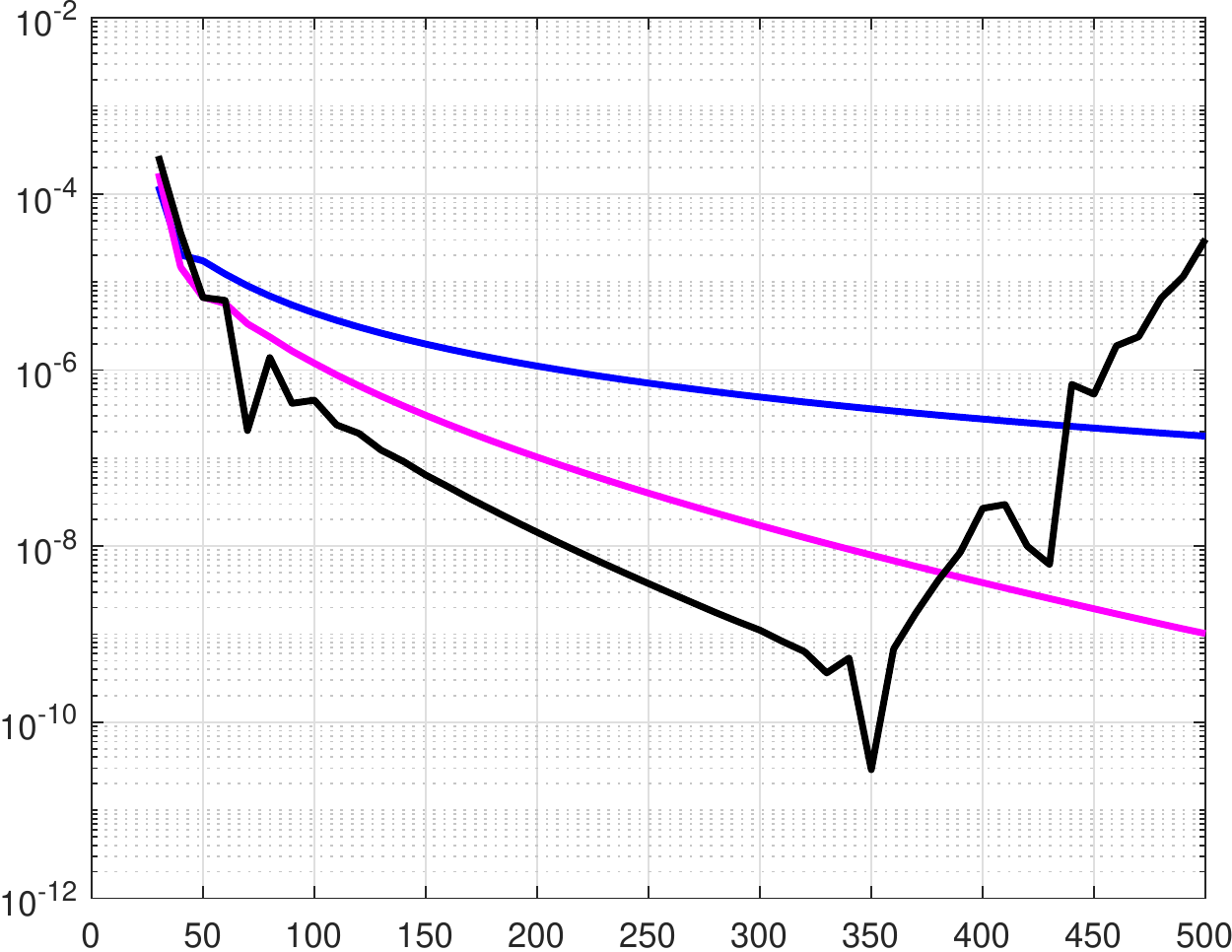}
		\caption*{$f_{1}=\frac{1}{1+100x^{2}}$}
	\end{subfigure}\hfil 
	\begin{subfigure}{0.33\textwidth}
		\includegraphics[width=\linewidth]{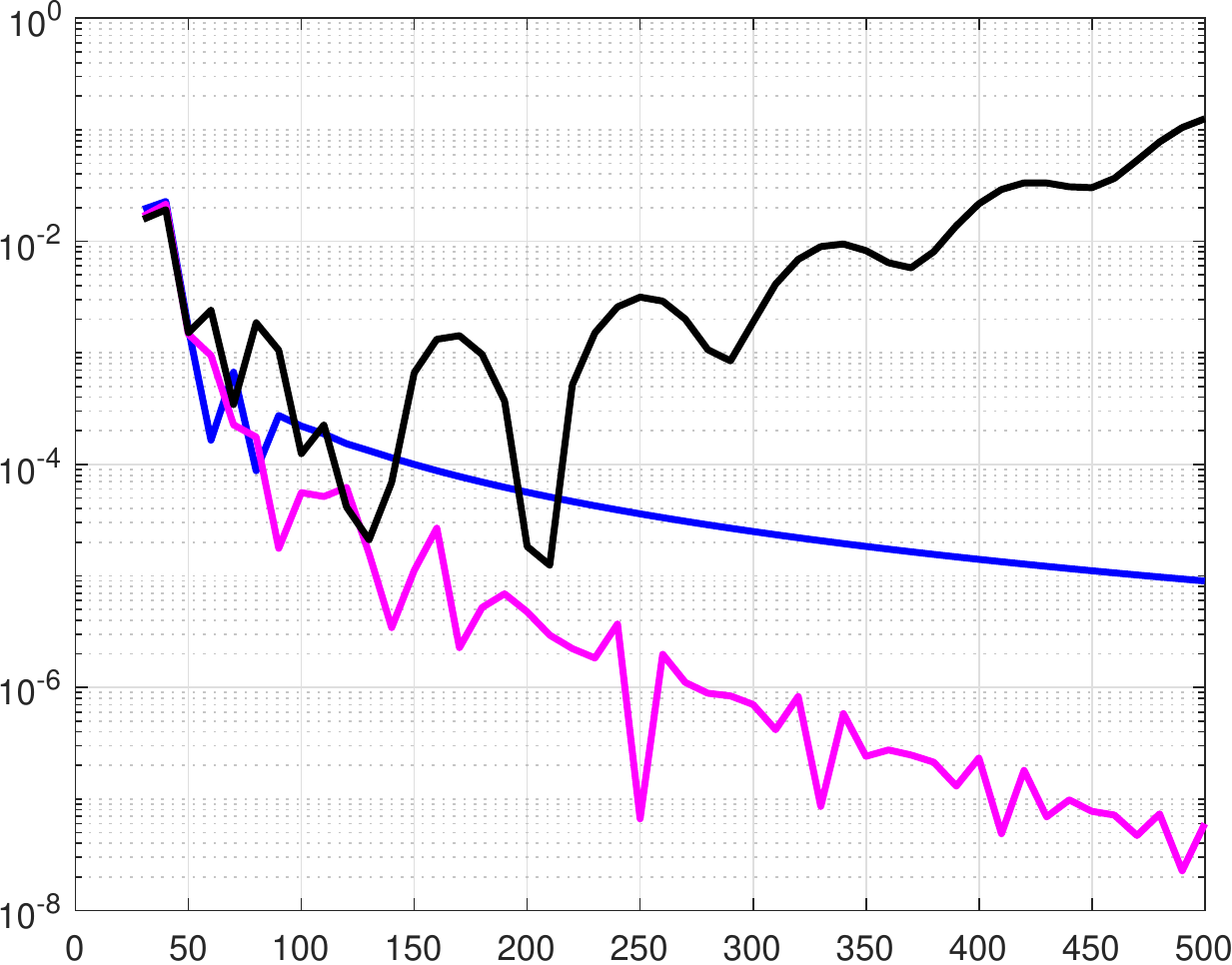}
		\caption*{$f_{2}=\frac{1}{1+16(\sin(7x)^{2})}$}
	\end{subfigure}\hfil 
	\begin{subfigure}{0.33\textwidth}
		\includegraphics[width=\linewidth]{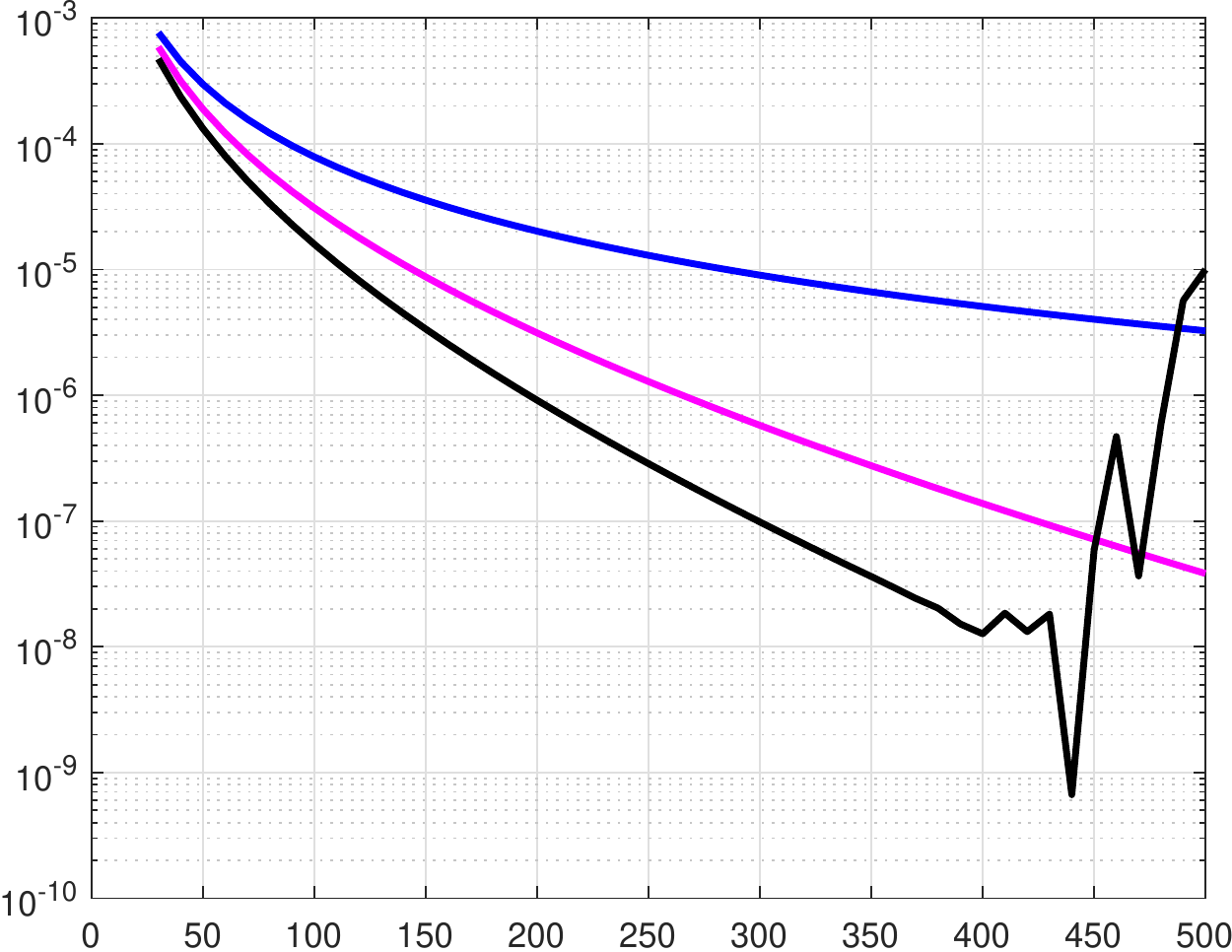}
		\caption*{$f_{3}=\sqrt{1.01+x}$}
	\end{subfigure}\hfil 
\captionsetup{justification=centering}
\caption{Relative quadrature error $\mathcal{E}_{\mathrm{rel}}$ for KTI quadrature using the parameters \\ \textcolor{blue}{\textbullet} $\alpha=1$, \textcolor{magenta}{\textbullet} $\alpha=0.99$, and \textcolor{black}{\textbullet} $\alpha=0.98$ in the Kosloff Tal-Ezer map.}
\label{fig:4}
\end{figure}
From a theoretical point of view (Section \ref{sec:parametersKTL} or \cite[Theorem 3.3]{platte}) we expect an algebraic convergence rate of index 1 ($\mathcal{O}(n^{-1})$) for $\alpha=1$, while for $0 \leq \alpha < 1$ the best approximation error decays geometrically ($\mathcal{O}(\rho^{-n})$) while the Lebesgue constant is not necessarily bounded. The numerical tests show that by lowering the value $\alpha$ the interpolation quadrature potentially improves but becomes unstable if $\alpha$ is not close to $1$ and the degree $n$ gets larger. This phenomenon can be explained by the behavior of the KTI quadrature weights displayed in Section \ref{sub_sec_limit_alpha}. Namely, the KTI weights get highly oscillatory and include negative values as soon as $\alpha$ is too far away from $1$ and the degree $n$ is large.  

\subsection{KTL quadrature: $\alpha$ increasing with the number of nodes}

In the previous numerical experiment we have seen that if the ratio between the number of nodes and the degree of the interpolation polynomial is $1$ phenomena of \we{ill-conditioning} occur. To avoid these instabilities, we go over to the KTL quadrature formula and approximate the integral using the formula $\mathcal{I}_{n,\mathcal{X}}^{\alpha}(f,I)$ calculated in Algorithm \ref{alg:1} and discussed in Section \ref{sec:KTLquadrature}.

In Fig. \ref{fig:5}, the $x$-axis displays the number $m$ (giving $m+1$ equidistant nodes) used to determine the KTL approximation of the integral while the $y$-axis shows the corresponding relative quadrature error. The blue curve describes the relative error of the composed trapezoidal rule ($\alpha = 1$), the magenta and black curve correspond to the errors obtained for the mapping parameter $\alpha=0.9$ and $\alpha=0.7$, respectively. To guarantee the boundedness of the Lebesgue constant it is sufficient that $n=\mathcal{O}(\sqrt {m})$, as shown in \cite[Corollary 5.2]{platte} (see also Section \ref{sec:parametersKTL}). In this numerical test we chose $n=4\sqrt{m}$.

\begin{figure}[H]
	\begin{subfigure}{0.33\textwidth}
		\includegraphics[width=\linewidth]{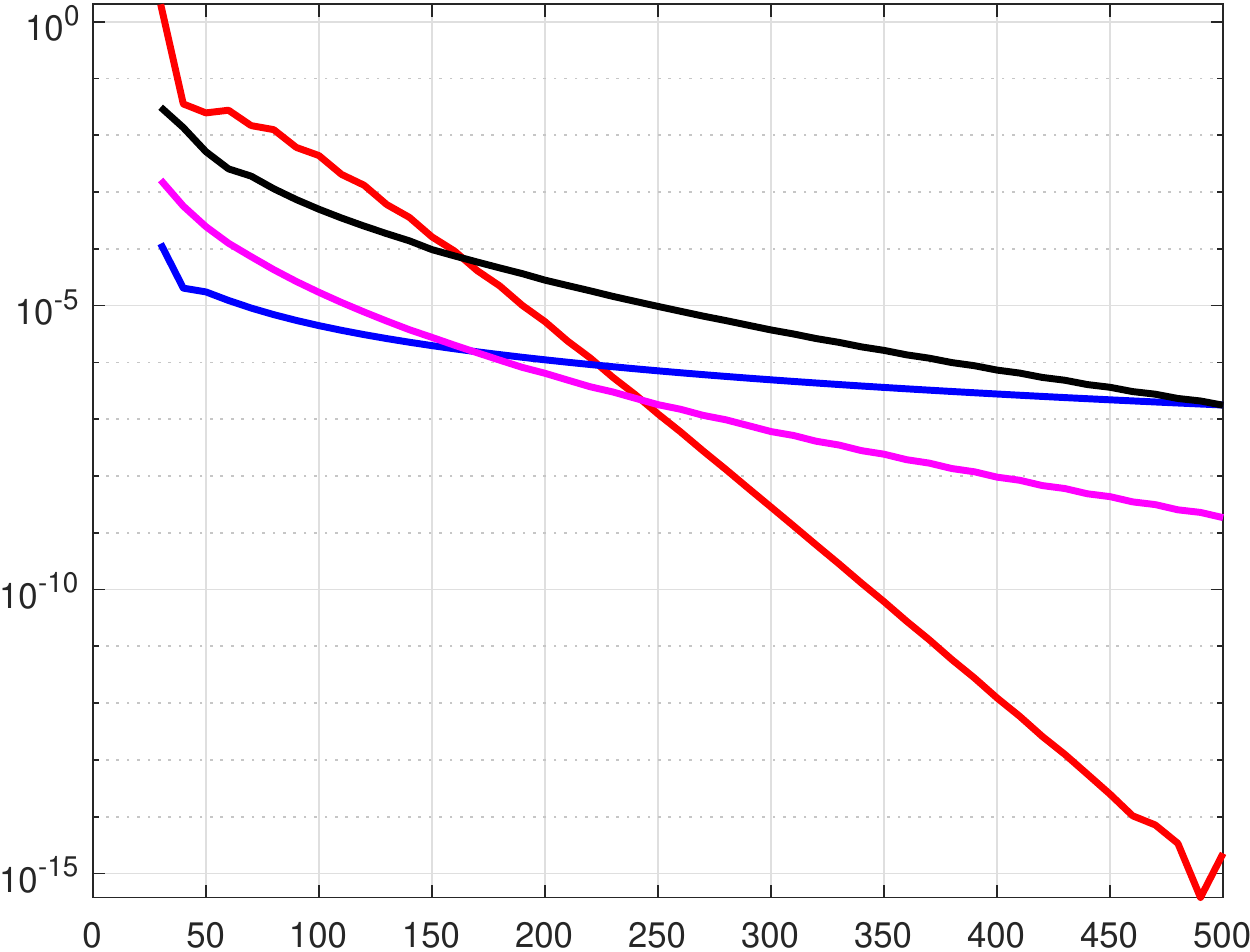}
		\caption*{$f_{1}=\frac{1}{1+100x^{2}}$}
	\end{subfigure}\hfil 
	\begin{subfigure}{0.33\textwidth}
		\includegraphics[width=\linewidth]{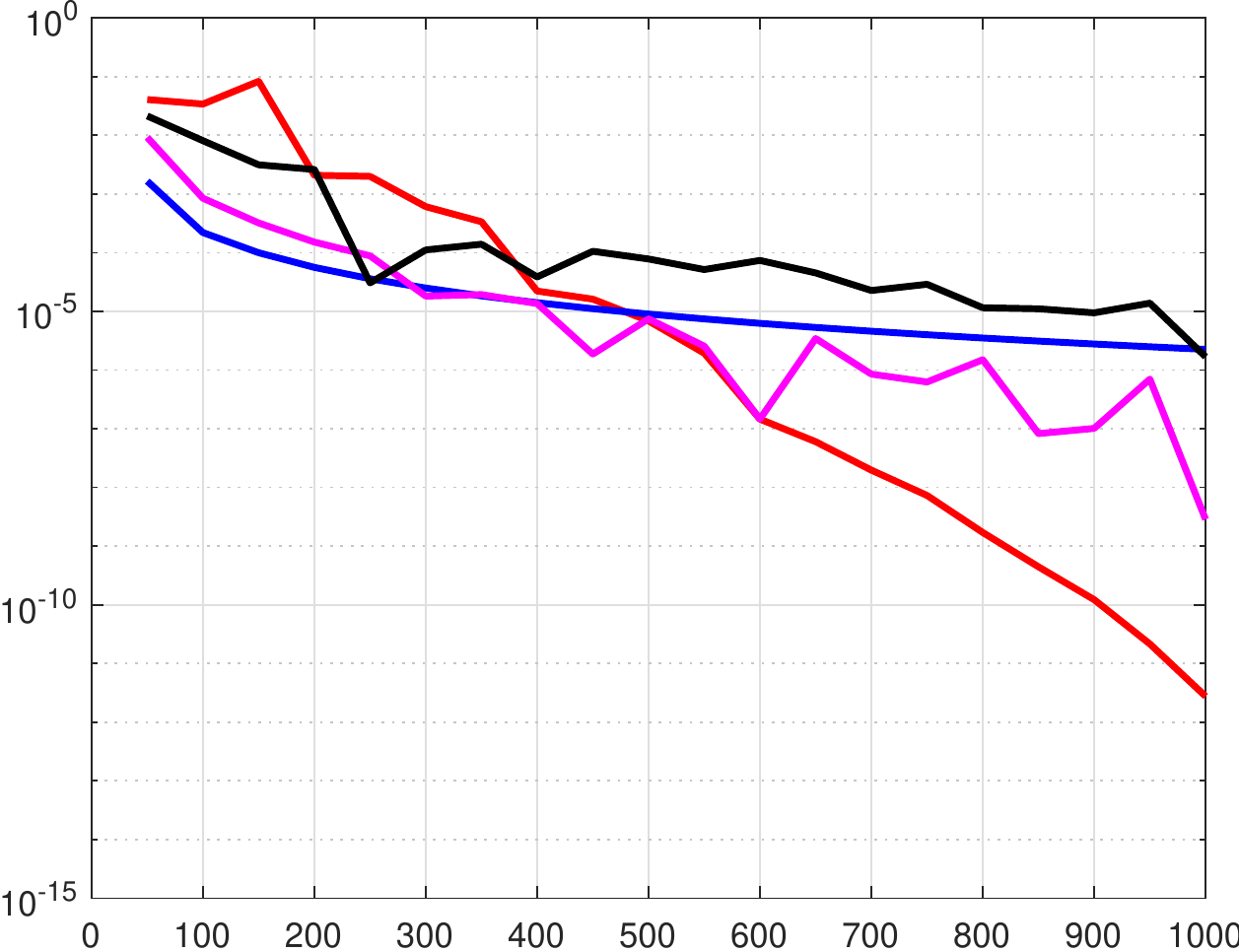}
		\caption*{$f_{2}=\frac{1}{1+16(\sin(7x)^{2})}$}
	\end{subfigure}\hfil 
	\begin{subfigure}{0.33\textwidth}
		\includegraphics[width=\linewidth]{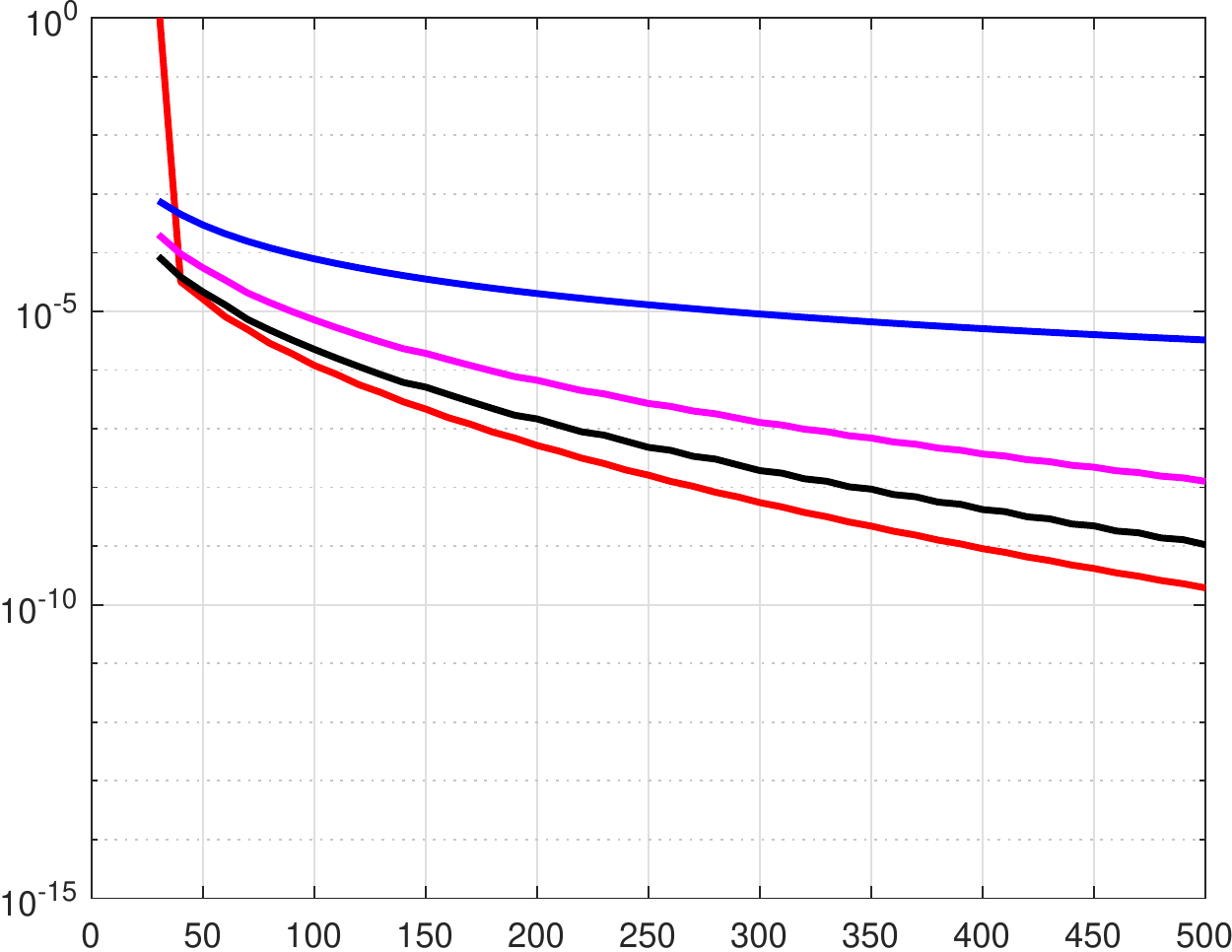}
		\caption*{$f_{3}=\sqrt{1.01+x}$}
	\end{subfigure}\hfil 
	\captionsetup{justification=centering}
	\caption{Relative quadrature error $\mathcal{E}_{\mathrm{rel}}$ for KTL quadrature using the parameters \\ \textcolor{blue}{\textbullet} $\alpha=1 , n=m$; \textcolor{magenta}{\textbullet} $\alpha=0.9 , n=4\sqrt{m}$; \textcolor{black}{\textbullet} $\alpha=0.7 , n=4\sqrt{m}$; \\ \textcolor{red}{\textbullet} $\alpha_n=1- \frac{2|\log(10^{-12})|}{n \pi} , n=\frac{1}{2}m$. } \label{fig:5}
\end{figure}
As before we expect for $\alpha = 1$ an algebraic convergence rate of index 1 ($\mathcal{O}(m^{-1})$), while for $0 \leq \alpha < 1$ the convergence rate is root exponential ($\mathcal{O}(\rho^{-\sqrt{m}})$) in the number $m$. The numerical experiment shows that by altering $\alpha$ the convergence rate of the KTL scheme slightly improves and remains stable even if $\alpha$ is not close to $1$. From Fig. \ref{fig:5} we do not notice an evident relationship between the displayed curves and a particular choice of the parameter $\alpha$. This indicates that a smart choice of the parameter is more appropriate for \wee{KTL schemes in which there is a linear relation between $n$ and $m$ and where $\alpha$ depends on the degree of the polynomial and/or the number of nodes}. As introduced in \cite[Theorem 3.3]{platte} (see also the exposition in Section \ref{sec:parametersKTL}) we choose
$$
\alpha_{n} = 1- \frac{2|\log(\epsilon)|}{n \pi} .
$$
In our experiments, we take $\epsilon=10^{-12}$ and $m=2n$. This choice of $\alpha_n$ corresponds to the red curve in Fig. \ref{fig:5}. We can report that the results significantly improve compared to the KTL quadrature with a fixed parameter $\alpha$. We also observe that the method remains stable for different values of $\alpha$ even if the number of nodes increases: the approximation of the integral of the function $f_{2} $ does not reach the machine precision in $500$ nodes, but if we increase $m$ further the approximation of the integral improves and the KTL method shows no signs of instability.

\subsection{KTL quadrature on non-equidistant nodes}

In the numerical tests presented so far equidistant nodes have been used as quadrature nodes. The KTL quadrature \we{rule, however, also works} for \we{quasi-uniform} distributions of the nodes \we{without any change in the numerical scheme as described in Algorithm \ref{alg:1}}. In the next numerical experiment the quadrature nodes are defined as follows:
$$
x_{i} = \delta_{i} + \Big(-1 + \frac{2i}{m}\Big), \, i= 0, \ldots , m,
$$
where $\delta_{i}$ is a uniform random variable in $]-1/m, 1/m[$ for $ i = 1,...,m-1$, $ \delta_{0}$ is a uniform random variable in $]0,1/m[$ and $\delta_{m}$ is a uniform random variable in $]-1/m,0[$. \we{This defines a set of quasi-equispaced points in the interval $[-1,1]$.} \we{The respective mapped nodes are not Chebyshev or Chebyshev-Lobatto nodes but 
their distribution is concentrated towards the boundaries of the interval when $\alpha$ is close to $1$.} \we{This feature of the distribution of the mapped nodes is a common property of well-conditioned polynomial interpolation schemes \cite[Chapter 5]{trefspect}.} 

\begin{figure}[H]
	\begin{subfigure}{0.33\textwidth}
		\includegraphics[width=\linewidth]{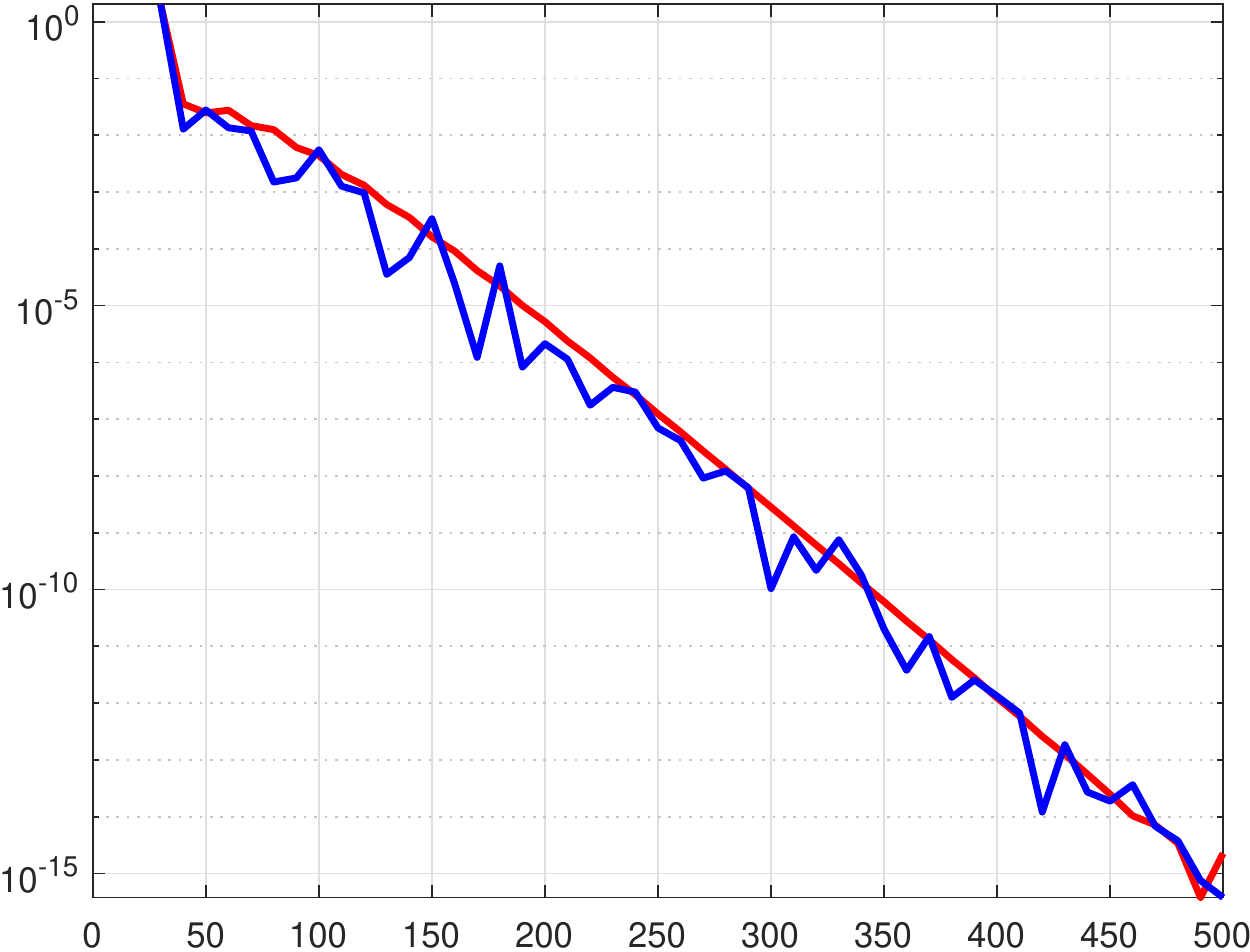}
		\caption*{$f_{1}=\frac{1}{1+100x^{2}}$}
	\end{subfigure}\hfil 
	\begin{subfigure}{0.33\textwidth}
		\includegraphics[width=\linewidth]{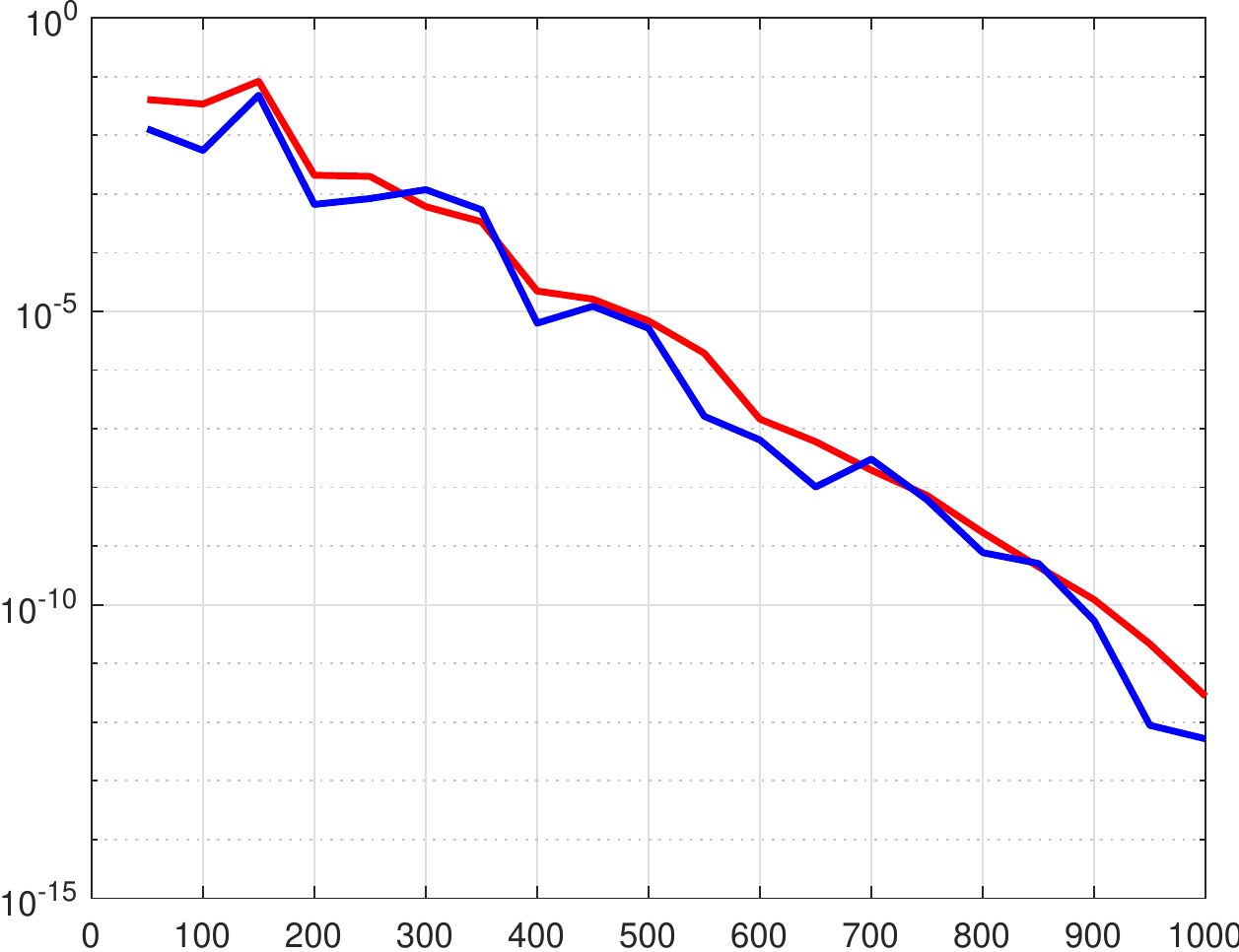}
		\caption*{$f_{2}=\frac{1}{1+16(\sin(7x)^{2})}$}
	\end{subfigure}\hfil 
	\begin{subfigure}{0.33\textwidth}
		\includegraphics[width=\linewidth]{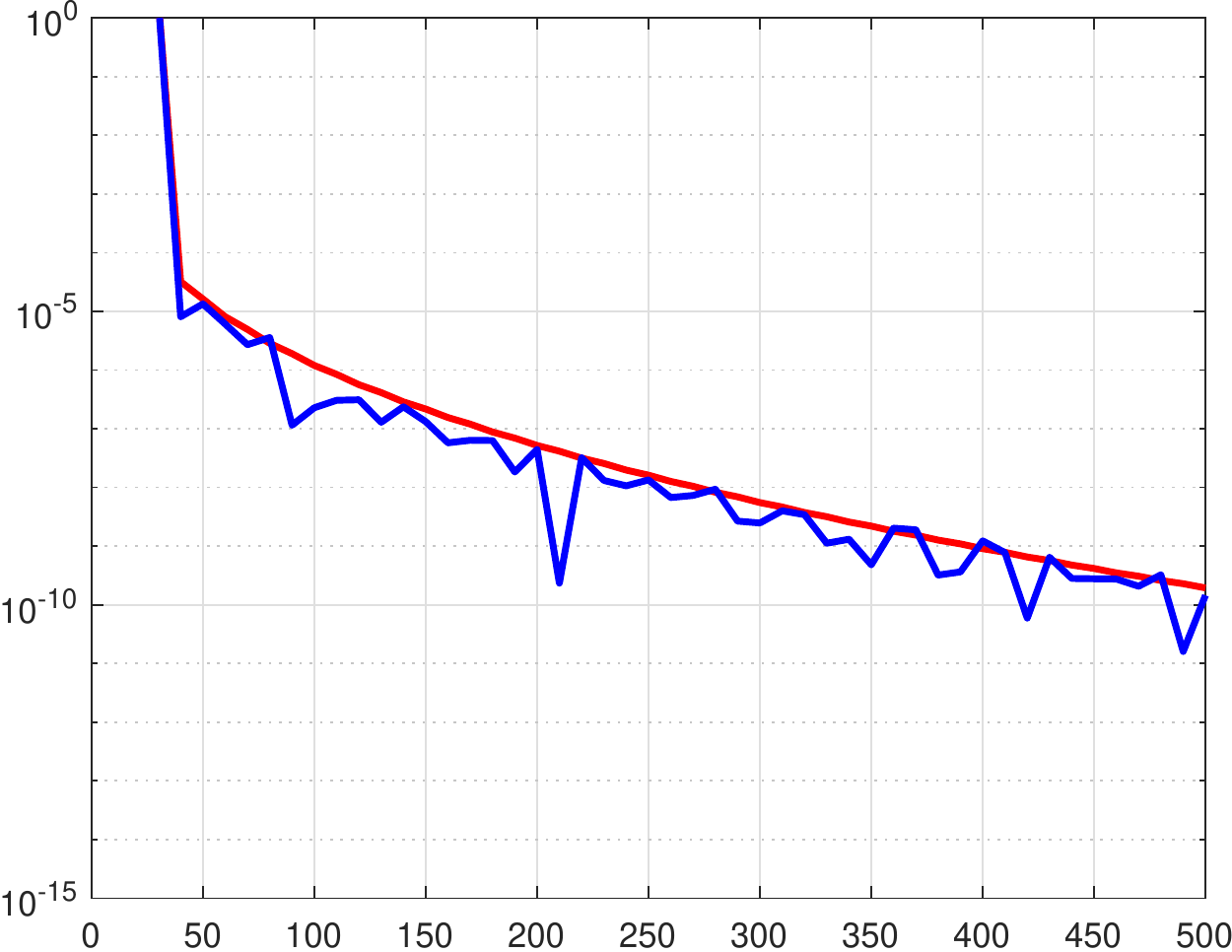}
		\caption*{$f_{3}=\sqrt{1.01+x}$}
	\end{subfigure}\hfil 
		\captionsetup{justification=centering}
		\caption{KTL quadrature for equidistant and perturbed nodes with the parameters \\ \textcolor{red}{\textbullet} $\alpha=1- \frac{2|\log(10^{-12})|}{n \pi} , n=\frac{1}{2}m$, equidistant nodes;  \\ \textcolor{blue}{\textbullet} $\alpha=1- \frac{2|\log(10^{-12})|}{n \pi} , n=\frac{1}{2}m$, perturbed nodes. }
		\label{fig:6}
\end{figure}
\we{For the proposed numerical scheme the distribution of the nodes plays only a minor role for the convergence as long as quasi-uniformity is given. We highlight this with a second numerical test using a low-discrepancy sequence of nodes such as the Halton points. Note that, when leaving the quasi-uniform setting and considering arbitrary grids, it is possible to construct unfavourable node distributions such that no convergence of the quadrature rule is achieved.}
\begin{figure}[H]
	\begin{subfigure}{0.33\textwidth}
		\includegraphics[width=\linewidth]{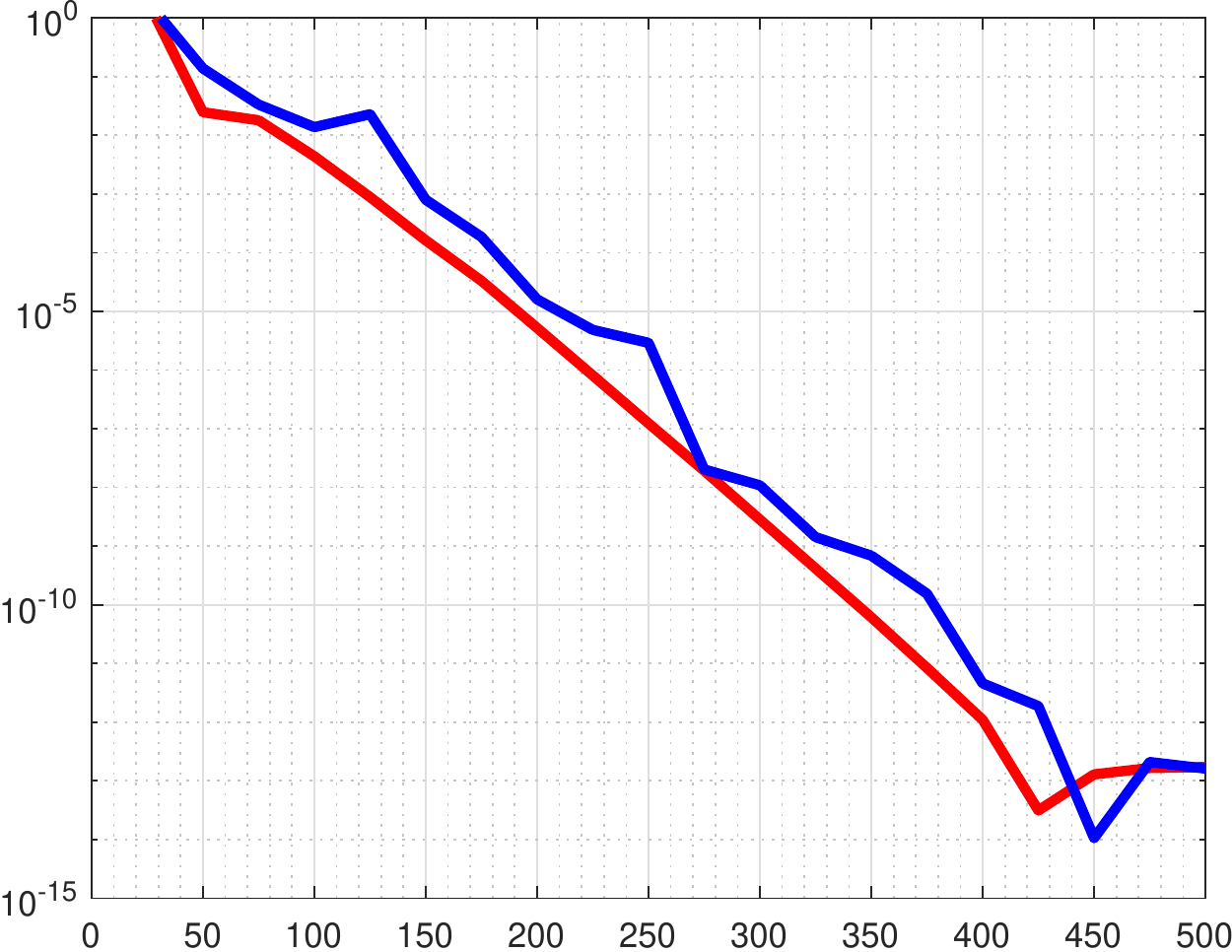}
		\caption*{$f_{1}=\frac{1}{1+100x^{2}}$}
	\end{subfigure}\hfil 
	\begin{subfigure}{0.33\textwidth}
		\includegraphics[width=\linewidth]{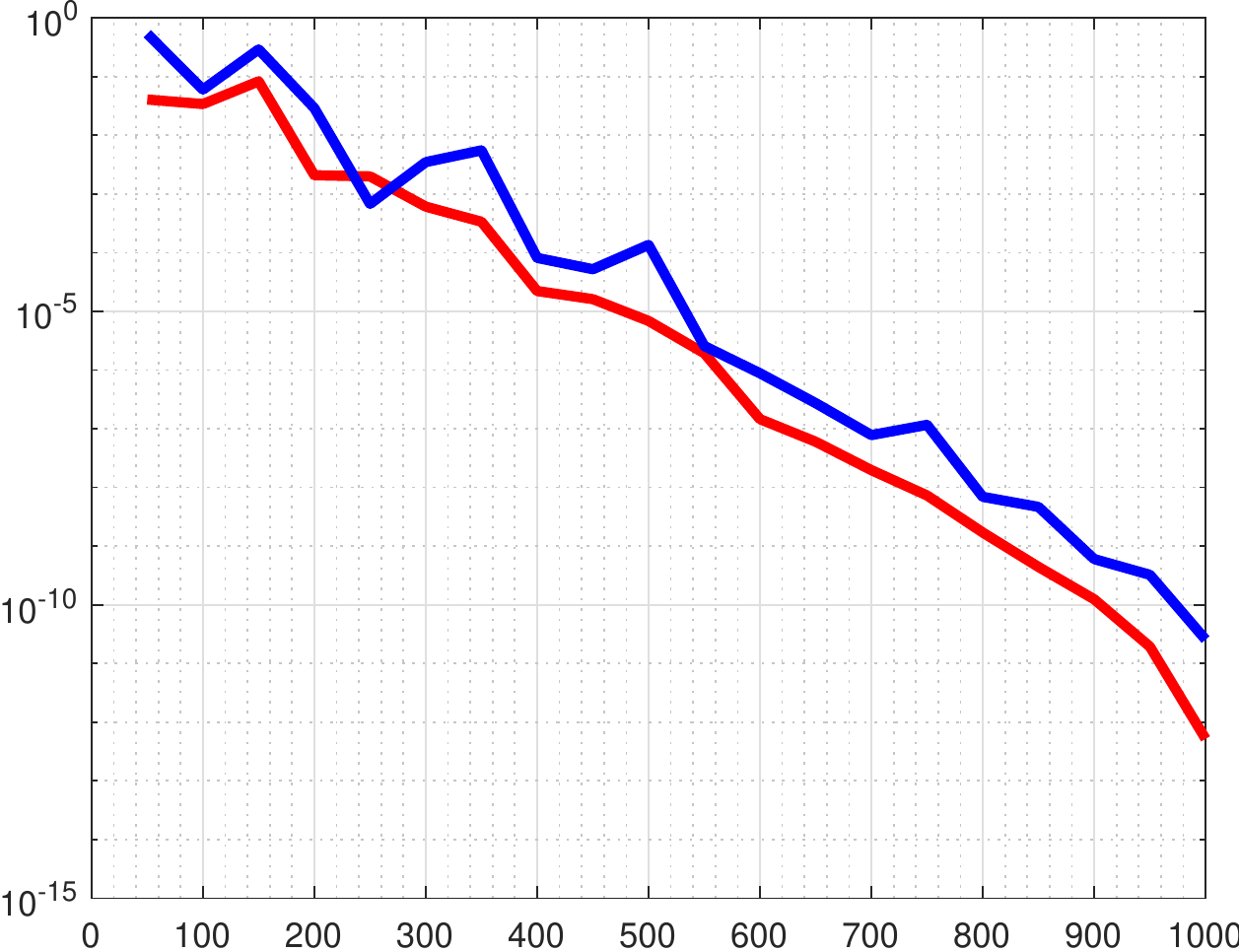}
		\caption*{$f_{2}=\frac{1}{1+16(\sin(7x)^{2})}$}
	\end{subfigure}\hfil 
	\begin{subfigure}{0.33\textwidth}
		\includegraphics[width=\linewidth]{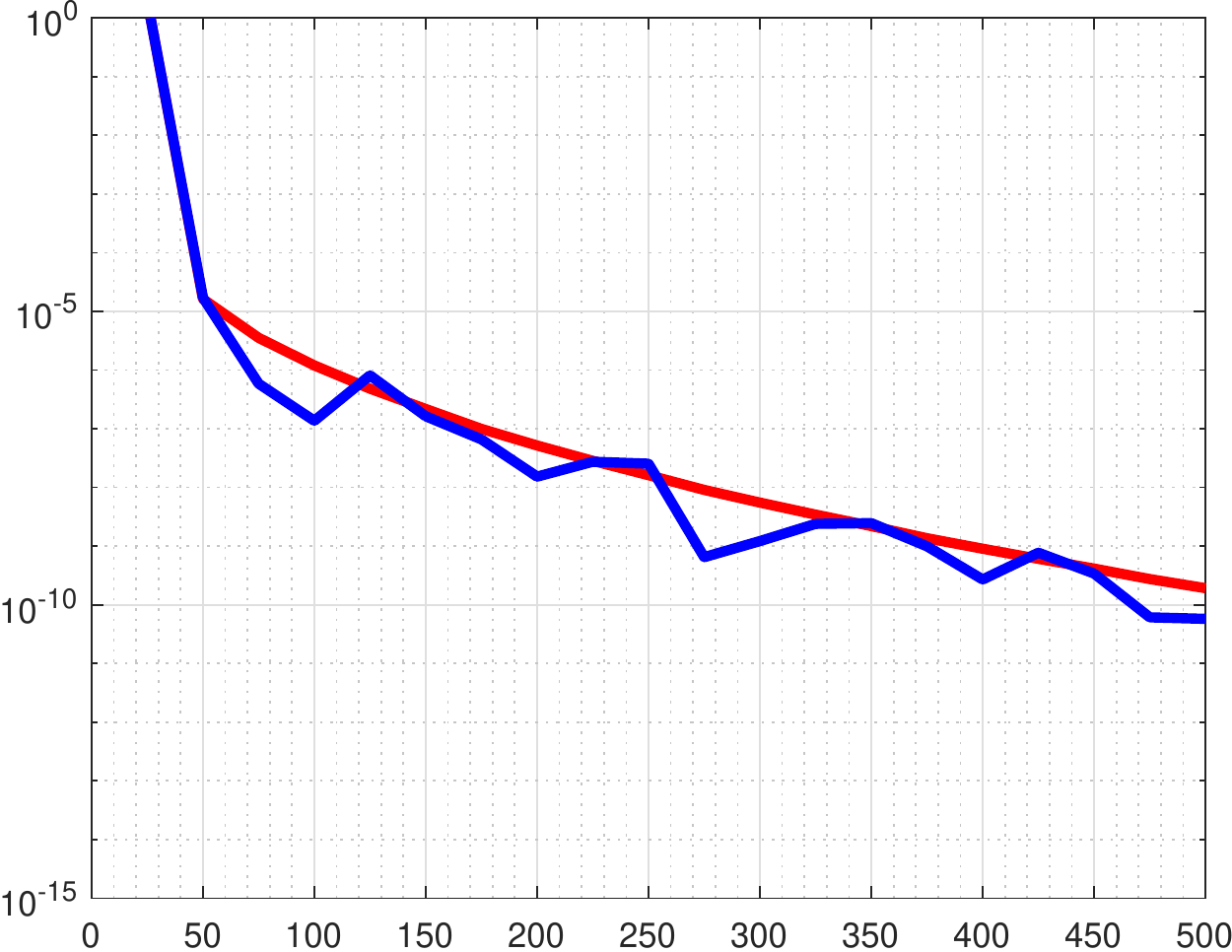}
		\caption*{$f_{3}=\sqrt{1.01+x}$}
	\end{subfigure}\hfil 
		\captionsetup{justification=centering}
		\caption{\we{KTL quadrature for equidistant and Halton points with the parameters \\} \textcolor{red}{\textbullet} $\alpha=1- \frac{2|\log(10^{-12})|}{n \pi} , n=\frac{1}{2}m$, equidistant nodes;  \\ \textcolor{blue}{\textbullet} $\alpha=1- \frac{2|\log(10^{-12})|}{n \pi} , n=\frac{1}{2}m$, Halton points. }
		\label{fig:7}
\end{figure}

The red curve in Fig. \ref{fig:6} and in Fig. \ref{fig:7} describes the relative error for equidistant nodes while the blue curve represents the error for the perturbed nodes and \we{Halton points, respectively}. We can conclude that the convergence rates for the perturbed and the equidistant nodes are approximately the same. 

\section{Conclusions}
We briefly summarize the most important points of our discussion. In order to get \we{well-conditioned} and quickly converging quadrature formulas at \we{quasi-uniform} grids of an interval, we improved classical interpolatory quadrature formulas using the following two strategies: (i) we included an auxiliary mapping that maps the quadrature nodes onto a new more suitable set of fake nodes. On these fake nodes the interpolatory quadrature is applied using the function values from the original set of nodes; (ii) we reduced the degree of the polynomial spaces with respect to the number of quadrature nodes leading to a least-squares quadrature formula instead of an interpolatory one. While the first strategy (i) alone \we{already yields}  an improvement with respect to a direct interpolatory formula, fast convergence for the integration of smooth functions is not guaranteed. Moreover, if one particular map is fixed also instabilities can occur if the the number of quadrature nodes gets large. In order to get both, fast convergence and stability, it turned out that the inclusion of the least-squares idea (ii) is necessary and that a smart choice of the mapping parameters is essential. We analyzed such quadrature strategies particularly for the Kosloff Tal-Ezer map and equidistant nodes. We derived several properties of the corresponding quadrature weights, including symmetries, limit relations and convergence properties depending on the central parameters of the scheme. We also showed how the KTL quadrature weights can be calculated efficiently by using a Chebyshev basis and a fast cosine transform for the computations. Our final numerical experiments confirm that the described parameter selections yield KTL quadrature formulas that converge quickly for smooth functions. 

\section*{Acknowledgements}
\wee{We thank the anonymous referees for their valuable feedback which improved the quality of the manuscript considerably.} This research has been accomplished within the Rete ITaliana di Approssimazione (RITA) and the thematic group on Approximation Theory and Applications of the Italian Mathematical Union. We received the support of GNCS-IN$\delta$AM. FM is funded by the ASI - INAF grant  \lq\lq Artificial Intelligence for the analysis of solar FLARES data (AI-FLARES)\rq\rq.

\we{ \section*{Declarations}
The authors have no conflicts of interest to declare that are relevant to the content of this article.}

\end{document}